\documentclass[11pt,reqno]{amsart}

\usepackage{amsmath, amsthm, amsopn, amssymb, graphicx, enumerate, enumitem, geometry, afterpage, caption, subcaption, color, textcomp, stmaryrd, bbm, tikz, epstopdf}
\usepackage[nobysame,msc-links]{amsrefs}
\usepackage[english]{babel}
\usetikzlibrary{arrows,automata}
\usetikzlibrary{decorations,decorations.pathmorphing}

\newtheorem{thm}{Theorem}[section]
\newtheorem{lemma}[thm]{Lemma}
\newtheorem{prop}[thm]{Proposition}

\newtheorem{cor}[thm]{Corollary}
\newtheorem{claim}[thm]{Claim}
\theoremstyle{definition}

\newcommand{\cA}{\mathcal{A}}

\newcommand{\cP}{\mathcal{P}}

\newcommand{\cH}{\mathcal{H}}

\newcommand{\cB}{\mathcal{B}}
\newcommand{\cR}{\mathcal{R}}
\newcommand{\cQ}{\mathcal{Q}}

\newcommand{\cF}{\mathcal{F}}
\newcommand{\cS}{\mathcal{S}}

\newcommand{\cX}{\mathcal{X}}
\newcommand{\cU}{\mathcal{U}}
\newcommand{\cV}{\mathcal{V}}
\newcommand{\cI}{\mathcal{I}}
\newcommand{\cW}{\mathcal{W}}

\newcommand{\Z}{\mathbb{Z}}

\renewcommand{\SS}{\mathbb{S}}
\newcommand{\T}{\mathbb{T}}
\newcommand{\E}{\mathbb{E}}

\renewcommand{\Pr}{\mathbb{P}}
\newcommand{\cE}{\mathcal{E}}
\newcommand{\1}{\mathbbm{1}}

\makeatletter
\DeclareRobustCommand{\cev}[1]{%
	\mathpalette\do@cev{#1}%
}
\newcommand{\do@cev}[2]{%
	\fix@cev{#1}{+}%
	\reflectbox{$\m@th#1\vec{\reflectbox{$\fix@cev{#1}{-}\m@th#1#2\fix@cev{#1}{+}$}}$}%
	\fix@cev{#1}{-}%
}
\newcommand{\fix@cev}[2]{%
	\ifx#1\displaystyle
	\mkern#23mu
	\else
	\ifx#1\textstyle
	\mkern#23mu
	\else
	\ifx#1\scriptstyle
	\mkern#22mu
	\else
	\mkern#22mu
	\fi
	\fi
	\fi
}
\makeatother

\newcommand{\dpartial}{\vec{\partial}}
\newcommand{\dpartialrev}{\cev{\partial}}
\newcommand{\intB}{\partial_{\bullet}}
\newcommand{\extB}{\partial_{\circ}}
\newcommand{\intextB}{\partial_{\ins \out}}
\newcommand{\ins}{\bullet}
\newcommand{\out}{\circ}
\newcommand{\up}{\,\uparrow\,}
\newcommand{\down}{\,\downarrow\,}

\DeclareMathOperator\zero{\bold{0}}
\DeclareMathOperator\rev{rev}
\newcommand{\Even}{\mathrm{Even}}
\newcommand{\Odd}{\mathrm{Odd}}
\DeclareMathOperator\dist{dist}
\DeclareMathOperator\diam{diam}
\newcommand{\distTV}{\mathrm{d_{TV}}}
\DeclareMathOperator\Int{int}

\newcommand{\phase}{\mathcal{P}}
\newcommand{\phasedom}{\mathcal{P}}

\newcommand{\bad}{\mathsf{none}}
\newcommand{\unbal}{\mathsf{unbal}}
\newcommand{\overlap}{\mathsf{overlap}}
\newcommand{\Ent}{\mathsf{Ent}}
\DeclareMathOperator\supp{supp}

\newcommand{\rest}{\mathsf{rest}}

\newcommand{\unique}{\mathsf{uniq}}
\newcommand{\nondom}{\mathsf{nondom}}
\newcommand{\iso}{\mathsf{iso}}
\newcommand{\even}{\mathsf{even}}
\newcommand{\odd}{\mathsf{odd}}
\newcommand{\bdry}{\mathsf{bdry}}
\newcommand{\inner}{\mathsf{int}}

\newcommand{\breakups}{\cX}

\makeatletter
\newcommand*\rel@kern[1]{\kern#1\dimexpr\macc@kerna}
\newcommand*\widebar[1]{%
  \begingroup
  \def\mathaccent##1##2{%
    \rel@kern{0.8}%
    \overline{\rel@kern{-0.8}\macc@nucleus\rel@kern{0.2}}%
    \rel@kern{-0.2}%
  }%
  \macc@depth\@ne
  \let\math@bgroup\@empty \let\math@egroup\macc@set@skewchar
  \mathsurround\z@ \frozen@everymath{\mathgroup\macc@group\relax}%
  \macc@set@skewchar\relax
  \let\mathaccentV\macc@nested@a
  \macc@nested@a\relax111{#1}%
  \endgroup
}
\makeatother

\title{Rigidity of proper colorings of $\Z^{\MakeLowercase{d}}$}
\date{\today}

\author{Ron Peled}
\address{Ron Peled\hfill\break
	Tel Aviv University\\
	School of Mathematical Sciences\\
	Tel Aviv, 69978, Israel.}
\email{peledron@post.tau.ac.il}
\urladdr{http://www.math.tau.ac.il/~peledron}

\author{Yinon Spinka}
\address{Yinon Spinka\hfill\break
	University of British Columbia\\
	Department of Mathematics\\
	Vancouver, BC V6T 1Z2, Canada.}
\email{yinon@math.ubc.ca}
\urladdr{http://www.math.ubc.ca/~yinon/}

\thanks{Research of both authors was supported by the Israel Science Foundation grant 861/15 and the European Research Council starting grant 678520 (LocalOrder). Research of Y.S. was additionally supported by the Adams Fellowship Program of the Israel Academy of
Sciences and Humanities.}

\begin{document}

\begin{abstract}
A proper $q$-coloring of a domain in $\Z^d$ is a function assigning one of $q$ colors to each vertex of the domain such that adjacent vertices are colored differently. Sampling a proper $q$-coloring uniformly at random, does the coloring typically exhibit long-range order? It has been known since the work of Dobrushin that no such ordering can arise when $q$ is large compared with $d$. We prove here that long-range order does arise for each $q$ when $d$ is sufficiently high, and further characterize all periodic maximal-entropy Gibbs states for the model. Ordering is also shown to emerge in low dimensions if the lattice $\Z^d$ is replaced by $\Z^{d_1}\times\T^{d_2}$ with $d_1\ge 2$, $d=d_1+d_2$ sufficiently high and $\T$ a cycle of even length. The results address questions going back to Berker--Kadanoff (1980), Koteck\'y (1985) and Salas--Sokal (1997).
\end{abstract}

\maketitle

\section{Introduction and results}

What does a typical proper coloring with $q$ colors of the integer lattice $\Z^d$ look like? By proper we mean that adjacent vertices must be colored differently. As the lattice $\Z^d$ is bipartite, having an \emph{even} and an \emph{odd} sublattice, it admits proper $q$-colorings for any $q\ge 2$. The $q=2$ case is degenerate with only two possible (proper) colorings -- the chessboard coloring and its translation by one lattice site. For $q\ge 3$ the number of colorings of bounded domains is exponentially large in the volume of the domain, as witnessed by the following important construction: Partition the $q$ colors into two subsets $A,B$ and consider the family of colorings obtained by coloring sites in the even sublattice with colors from $A$ and sites in the odd sublattice with colors from $B$. On a domain $\Lambda$ with an equal number of even and odd sites this gives $(|A|\cdot |B|)^{|\Lambda|/2}$ colorings, and this quantity is maximized when $\{|A|,|B|\}=\{\lfloor \frac{q}{2}\rfloor, \lceil \frac{q}{2}\rceil\}$. Certainly most colorings are not obtained this way, but could it be that most colorings coincide with such a ``pure $(A,B)$-coloring'' at \emph{most vertices}? This is evidently not so in dimension $d=1$ (when $q\ge 3$) and, in fact, is not the case in any dimension provided the number of colors is large compared with the dimension ($q>4d$ suffices; see the discussion after Theorem~\ref{thm:long-range-order}). The main result presented here deals with the opposite regime -- when the dimension is large compared with the number of colors -- where it is shown that coincidence at most vertices with a ``pure $(A,B)$-coloring'' does in fact take place. More precisely, when $\{A,B\}$ partitions the $q$ colors into sets of sizes $\lfloor \frac{q}{2}\rfloor$ and $\lceil \frac{q}{2}\rceil$, then picking a coloring uniformly among colorings of a domain which follow the $(A,B)$-pattern on its boundary, the coloring at any vertex in the domain is very likely to follow the $(A,B)$-pattern as well.

We proceed to state our main result, following required notation. A \emph{pattern} is a pair $(A,B)$ of disjoint subsets of $[q]:=\{1,\ldots, q\}$ (we stress that $(A,B)$ and $(B,A)$ are distinct patterns). It is called \emph{dominant} if $\{|A|,|B|\}=\left\{\lfloor \tfrac{q}{2}\rfloor, \lceil \tfrac{q}{2}\rceil\right\}$.
A \emph{domain} is a non-empty finite $\Lambda \subset \Z^d$ such that both $\Lambda$ and $\Z^d\setminus\Lambda$ are connected. Its \emph{internal vertex-boundary}, the set of vertices in $\Lambda$ adjacent to a vertex outside $\Lambda$, is denoted $\intB \Lambda$.
Given a proper $q$-coloring $f$, we say that
\[ \text{a vertex $v$ is \emph{in the $(A,B)$-pattern} if either $v$ is even and $f(v) \in A$, or $v$ is odd and $f(v) \in B$}.\]
We also say that a set of vertices is in the $(A,B)$-pattern if all its elements are such.

\begin{thm}\label{thm:long-range-order}
There exists $C\ge 1$ such that for any number of colors $q \ge 3$ and any dimension
\begin{equation}\label{eq:dim-assump}
d \ge Cq^{10} \log^2 q,
\end{equation}
the following holds. Let $(A,B)$ be a dominant pattern. Let $\Lambda\subset\Z^d$ be a domain. Let $\Pr_{\Lambda,(A,B)}$ be the uniform measure on proper $q$-colorings $f$ of $\Lambda$ satisfying that $\intB \Lambda$ is in the $(A,B)$-pattern. Then
\begin{equation}\label{eq:main_thm_bound}
 \Pr_{\Lambda,(A,B)}\big(v\text{ is not in the $(A,B)$-pattern}\big) \le e^{-\frac{d}{q^3(q+\log d)}},\qquad v\in\Lambda.
\end{equation}
\end{thm}
The theorem establishes the existence of long-range order, as the effect of the imposed boundary conditions on the distribution of $f(v)$ does not vanish in the limit as the domain $\Lambda$ increases to the whole of $\Z^d$. Indeed, by symmetry among the colors, the bound~\eqref{eq:main_thm_bound} implies that for some $\epsilon>0$, any domain $\Lambda$ and, for concreteness, any even vertex $v\in\Lambda$,
\begin{equation}\label{eq:color_distribution}
  \Pr_{\Lambda,(A,B)}\big(f(v)=i\big) \ge \tfrac{1}{q}+\epsilon\quad \text{if } i\in A\qquad\text{and}\qquad \Pr_{\Lambda,(A,B)}\big(f(v)=i\big) \le \tfrac{1}{q}-\epsilon\quad \text{if } i\in B.
\end{equation}
The above statements quantify the probability of single-site deviations from the boundary pattern. Extensions to larger spatial deviations are provided in Section~\ref{sec:large-violations} and a consequence for the enumeration of proper $q$-colorings is discussed in Section~\ref{sec:max-entropy-states}.

It is natural to wonder whether other restrictions on the boundary values besides the one used in Theorem~\ref{thm:long-range-order} would lead to other behaviors of the coloring in the bulk of the domain. This idea is captured by the notion of a Gibbs state: a probability measure on proper $q$-colorings of $\Z^d$ for which the conditional distribution on any finite set, given the coloring outside the set, is uniform on the proper colorings extending the boundary values (see Section~\ref{sec:gibbs} for a precise definition).
A fundamental problem in statistical physics is to understand the set of Gibbs states corresponding to a given model. In many models, including proper $q$-colorings, it is evident that there is at least one Gibbs state and the next question arising is to ascertain whether there is more than one. Dobrushin gave a fundamental sufficient condition for the uniqueness of Gibbs states~\cite{Dobrushin1968TheDe}. Applied to proper $q$-colorings, it implies uniqueness whenever $q>4d$ (due to Koteck\'y~\cite[pp.~148-149,457]{georgii2011gibbs} and Salas--Sokal~\cite{salas1997absence}). This bound was improved to $q>\frac{11}{3}d$ by Vigoda~\cite{vigoda2000improved}, with a further improvement to approximately $q>3.53d$ by Goldberg--Martin--Paterson~\cite{goldberg2005strong}, relying on the fact that $\Z^d$ has no triangles.

In the opposite direction, results showing multiplicity of Gibbs states are in general more difficult to obtain. For the $q$-coloring model, this question may be trivial to answer due to the existence of ``frozen Gibbs states'' -- measures supported on a single proper coloring $f$, with the property that $f$ cannot be modified on any finite set while staying proper -- which are known to exist if and only if $q\le d+1$~\cite{alon2019mixing}. To avoid this degenerate situation, one often restricts consideration to Gibbs states of \emph{maximal entropy} -- Gibbs states invariant under translations by a full-rank sublattice of $\Z^d$, termed \emph{periodic} Gibbs states, whose measure-theoretic entropy equals the topological entropy of proper $q$-colorings (see Section~\ref{sec:max-entropy-states}) -- and the challenge is then to determine whether there is more than one such measure. A concrete question, which has received significant attention in the literature (see Section~\ref{sec:background}), is to determine whether multiple Gibbs states of maximal entropy exist for any number of colors~$q$, when the dimension~$d$ is sufficiently high. In fact, the result~\eqref{eq:color_distribution} immediately implies the existence of multiple Gibbs states, one for each dominant pattern $(A,B)$, and it is not overly difficult to establish that these have maximal entropy. This fact, along with additional properties, constitutes our second main result.

\begin{thm}\label{thm:existence_Gibbs_states}
Let $q \ge 3$ and suppose that the dimension $d$ satisfies~\eqref{eq:dim-assump}. For each dominant pattern $(A,B)$ there exists a Gibbs state $\mu_{(A,B)}$ such that, for any sequence of domains $\Lambda_n$ increasing to $\Z^d$, the measures $\Pr_{\Lambda_n,(A,B)}$ converge weakly to $\mu_{(A,B)}$ as $n \to \infty$. In particular, $\mu_{(A,B)}$ is invariant to automorphisms of $\Z^d$ preserving the two sublattices. Moreover, the $(\mu_{(A,B)})$ are distinct, extremal and of maximal entropy.
\end{thm}

Together with Theorem~\ref{thm:long-range-order} we see that the Gibbs state $\mu_{(A,B)}$ has a tendency towards the $(A,B)$-pattern at all vertices. Our techniques yield stronger facts, showing that large spatial deviations from the $(A,B)$-pattern are exponentially suppressed (see Section~\ref{sec:large-violations}). The techniques further yield that $\mu_{(A,B)}$ is strongly mixing with an exponential rate (see Lemma~\ref{lem:almost-independence-of-colorings}).

Theorem~\ref{thm:existence_Gibbs_states} shows that there are at least $\binom{q}{q/2}$ extremal maximal-entropy Gibbs states for even~$q$ and $2\binom{q}{\lfloor q/2 \rfloor}$ such Gibbs states for odd $q$. Our third result shows that these exhaust all possibilities.
\begin{thm}\label{thm:characterization_of_Gibbs_states}
  Let $q \ge 3$ and suppose that the dimension $d$ satisfies~\eqref{eq:dim-assump}. Then any (periodic) maximal-entropy Gibbs state is a mixture of the measures $\{\mu_{(A,B)}\}$.
\end{thm}

The main results are not valid in low dimensions due to the uniqueness results discussed above. Nonetheless, they are applicable in any dimension $d\ge 2$ provided the underlying graph is suitably modified. Precisely, the above results remain true as stated when $\Z^d$ is replaced by a graph of the form $\Z^{d_1}\times\T_{2m}^{d_2}$, $m\ge 1$ integer, provided $d_1\ge 2$ and $d=d_1+d_2$ satisfies \eqref{eq:dim-assump}, where $\T_{2m}$ is the cycle graph on $2m$ vertices (the path on $2$ vertices if $m=1$). The graph $\Z^{d_1}\times\T_{2m}^{d_2}$ may be viewed as a subset of $\Z^d$ in which the last $d_2$ coordinates are restricted to take value in $\{0,1,\ldots, 2m-1\}$ and are endowed with periodic boundary conditions. In this sense, it is only the local structure of $\Z^d$ which matters to the results. To keep the discussion focused, we present the proofs of the results only in the $\Z^d$ case and comment on the minor adjustments (beyond obvious notational changes) required for graphs of the above form.

\subsection{General spin systems}\label{sec:general_spin_systems}
The methods introduced in this paper allow a vast generalization: In the companion paper~\cite{peledspinka2018spin}, we extend the ideas from the proper $q$-coloring setting to general discrete spin systems satisfying suitable conditions. The results characterize the set of maximal-pressure Gibbs states of such systems, showing that a typical sample from such a Gibbs state mainly follows an $(A,B)$ pattern for suitable sets $A,B$. We briefly describe here the main results of~\cite{peledspinka2018spin}. An introduction aimed at a physics audience appears in~\cite{peled2017condition}.

The spin systems considered are described by a finite \emph{spin space} $\SS$, a collection $(\lambda_i)_{i \in \SS}$ of positive numbers called the \emph{single-site activities}, and a collection $(\lambda_{i,j})_{i,j \in \SS}$ of non-negative numbers called the \emph{pair interactions}.
The pair interactions are symmetric, i.e., $\lambda_{i,j}=\lambda_{j,i}$ for all $i,j \in \SS$, and at least one is positive. The probability of a \emph{configuration} $f \colon \Lambda \to \SS$ is proportional to
\begin{equation}\label{eq:config-weight}
\prod_{v \in \Lambda} \lambda_{f(v)} \prod_{\{u,v\} \in E(\Lambda)} \lambda_{f(u),f(v)} ,
\end{equation}
where $E(\Lambda)$ is the set of edges of $\Z^d$ whose two endpoints belong to $\Lambda$. Classical models obtained as special cases include the Ising, Potts, hard-core, Widom--Rowlinson, beach and clock models.

The $q$-state antiferromagnetic Potts model at temperature $T$ is obtained when $\SS = [q]$ and $\lambda_{i,j}=\1_{\{i\neq j\}}+e^{-\frac{1}{T}}\1_{\{i=j\}}$. The $(\lambda_i)$ encode external magnetic fields. The proper $q$-coloring model is obtained in the zero-temperature limit, when $\lambda_{i,j}=\1_{\{i\neq j\}}$, taking all $\lambda_i = 1$.

The emergent long-range order will involve spins interacting with the maximal pair interaction weight. In this setting, a \emph{pattern} is thus defined as a pair $(A,B)$ of subsets of $\SS$ such that
\[ \lambda_{a,b}=\max_{i,j \in \SS} \lambda_{i,j}\qquad\text{for all $a \in A$ and $b \in B$}.\]
The single-site activities then play a role in singling out \emph{dominant} patterns, defined as patterns maximizing $(\sum_{a\in A} \lambda_a)(\sum_{b\in B} \lambda_b)$ among all patterns. These definitions extend the ones used above for proper $q$-colorings.

Two patterns $(A,B)$ and $(A',B')$ are called \emph{equivalent} if there is a bijection $\varphi \colon \SS \to \SS$ such that
\begin{equation*}
\{\varphi(A),\varphi(B)\}=\{A',B'\},\qquad \lambda_{\varphi(i)}=\lambda_i,\qquad \lambda_{\varphi(i),\varphi(j)}=\lambda_{i,j}\qquad\text{for all }i,j \in \SS.
\end{equation*}
The results of the companion paper apply to spin systems in which all dominant patterns are equivalent.

As for proper colorings, here too we wish to avoid degenerate situations, and thus restrict attention to (periodic) maximal-pressure Gibbs states (which are the analogues of maximal-entropy Gibbs states in this more general setting).

\begin{thm}[{\cite{peledspinka2018spin}}]\label{thm:equilibrium-states-general}
For each spin system as above (fixing $\SS$, $(\lambda_i)$ and $(\lambda_{i,j})$) in which all dominant patterns are equivalent there exists $d_0$ such that the following holds in any dimension $d \ge d_0$.
\begin{enumerate}[leftmargin=20pt]
  \item For each dominant pattern $(A,B)$ there exists a Gibbs state $\mu_{(A,B)}$ which is extremal, invariant to automorphisms of $\Z^d$ preserving the two sublattices and of maximal pressure.
  \item The Gibbs states $(\mu_{(A,B)})$ are distinct, with samples from $\mu_{(A,B)}$ having a strong tendency to follow the $(A,B)$-pattern in the sense that $\mu_{(A,B)}(f(u) \in A,\,f(v) \in B)\ge 1-\epsilon(d)$, for even $u \in \Z^d$ and odd $v \in \Z^d$, where $\epsilon(d) \to 0$ as $d \to \infty$.
  \item Every (periodic) maximal-pressure Gibbs state is a mixture of the measures $\{\mu_{(A,B)}\}$.
\end{enumerate}
\end{thm}

A quantitative estimate for $d_0$ in terms of $(\lambda_i)$ and $(\lambda_{i,j})$ is possible, encapsulating conditions of ``low-temperature'' and ``significant weight difference between dominant and non-dominant patterns'', as described in~\cite{peledspinka2018spin}. These imply, for instance, that the results obtained for the proper $q$-coloring model extend to the low-temperature regime of the antiferromagnetic $q$-state Potts model, with the temperature even allowed to \emph{grow} with $d$ at a power-law rate. Also described in~\cite{peledspinka2018spin} are properties of the Gibbs state $\mu_{(A,B)}$ which is in correspondence with the dominant pattern $(A,B)$, among which are quantitative bounds for $\epsilon(d)$ and convergence of finite-volume measures with $(A,B)$ boundary conditions to $\mu_{(A,B)}$. Applications to other classical models including the hard-core, lattice Widom--Rowlinson, beach and clock models are also discussed.


As for the $q$-coloring model, a version of Theorem~\ref{thm:equilibrium-states-general} remains valid on $\Z^{d_1}\times\T_{2m}^{d_2}$ provided $d_1\ge 2$ and $d=d_1+d_2$ is at least the threshold $d_0$ of the theorem.

\subsection{Discussion and background}\label{sec:background}

Long-range ordering results of the type obtained here are ubiquitous in statistical physics. Starting from the classical result of Peierls~\cite{peierls1936ising} that the Ising model orders at low temperature, such results have been obtained for a wide range of models. In the example of the Ising model, where the state space is $S = \{+, -\}$, the probability distribution biases against different values being placed at adjacent vertices. In the limit of zero temperature, this bias becomes absolute and the only allowed configurations in a domain are the fully~$+$ or fully~$-$ configurations. The result of Peierls may thus be viewed as saying that the zero-temperature ordering persists to the low-temperature regime, an idea which received systematic treatment starting with the work of Pirogov and Sinai~\cite{pirogov1975phase,pirogov1976phase} (see Friedli--Velenik~\cite[Chapter~7]{friedli2017statistical} for a pedagogical introduction). In contrast, the proper $q$-coloring model studied here is already a zero-temperature model (for the antiferromagnetic $q$-state Potts model), with the difficulty in its analysis stemming from the fact that it has \emph{residual entropy} -- configurations are sampled uniformly from a set whose cardinality is exponential in the volume. As such, any long-range order present in the model is \emph{entropically driven} and its rigorous justification requires new tools.

The question of understanding the type of emergent long-range order, or its absence, in the antiferromagnetic $q$-state Potts model, including proper $q$-colorings, has received significant attention. In the physics literature, to our knowledge, the problem was first considered by Berker--Kadanoff~\cite{berker1980ground} who suggested in 1980 that a phase with algebraically decaying correlations may occur at low temperatures (including zero temperature) with fixed $q$ when $d$ is large. This prediction was challenged by numerical simulations and an $\varepsilon$-expansion argument of Banavar--Grest--Jasnow~\cite{banavar1980ordering} who predicted a Broken-Sublattice-Symmetry (BSS) phase at low temperatures for the $3$ and $4$-state models in three dimensions. The BSS phase is exactly of the type proved to occur here, with a global tendency towards a pure $(A,B)$-ordering for a dominant pattern $(A,B)$. Koteck\'y~\cite{kotecky1985long} in 1985 argued for the existence of the BSS phase at low temperature when $q=3$ and $d\ge 3$ by analyzing the model on a decorated lattice. This prediction became known as \emph{Koteck\'y's conjecture}. While our concern here is with the zero-temperature case, we briefly mention that the behavior of the antiferromagnetic Potts model at intermediate temperature regimes is also unclear. The interested reader is directed to the paper of Rahman--Rush--Swendsen~\cite{rahman1998intermediate}, where the $3$-state model in three dimensions is considered, conflicting predictions regarding Permutationally-Symmetric-Sublattice (PSS) and Rotationally-Symmetric (RS) phases are surveyed and the controversy between them is addressed. We are not aware of mathematically rigorous results on such intermediate-temperature regimes. We also mention that irregularities in a lattice (i.e., having different sublattice densities) often promote the formation of order. This may be used, for instance, to find for each $q$ a \emph{planar} lattice on which the proper $q$-coloring model is ordered~\cite{huang2013two}. However, irregularities also modify the nature of the resulting phase, leading to long-range order in which a single spin value appears on most of the lower-density sublattice~\cite{kotecky2014entropy}, or to partially ordered states~\cite{qin2014partial}.

In the mathematically rigorous literature, Koteck\'y's conjecture remained open for 25 years until its high-dimensional case was verified at zero temperature by the first author~\cite{peled2010high} and by Galvin--Kahn--Randall--Sorkin~\cite{galvin2012phase} (following closely related papers by Galvin--Randall~\cite{galvin2007torpid} and Galvin--Kahn~\cite{galvin2004phase}). The high-dimensional case of the conjecture was fully resolved some years later by Feldheim and the second author~\cite{feldheim2015long}. The results of \cite{peled2010high, galvin2012phase} correspond to the $q=3$ case of Theorem~\ref{thm:long-range-order}, and to the existence of $6$ extremal maximal-entropy Gibbs states which results from it (the fact that the measures have maximal entropy is shown in~\cite[Section~5]{galvin2012phase}), while the convergence result in Theorem~\ref{thm:existence_Gibbs_states} and the characterization result given in Theorem~\ref{thm:characterization_of_Gibbs_states} are new also for this case. Periodic boundary conditions were considered in~\cite{galvin2007torpid, feldheim2013rigidity} and in~\cite{peled2010high} for the corresponding height function (also on tori with non-equal side lengths). Following Koteck\'y, it is quite natural to predict that multiple maximal-entropy Gibbs states exist for proper $q$-colorings with any $q\ge3$ provided the dimension is sufficiently large as a function of~$q$. Related questions and conjectures have been made by several authors:
\begin{itemize}[leftmargin=20pt]
  \item Salas--Sokal~\cite{salas1997absence} write in 1997 that any lattice $G$ should admit a value $q_c(G)$ such that the anitferromagnetic $q$-state Potts model on $G$ is disordered at all $q>q_c(G)$ and all temperatures, has a critical point at zero temperature when $q=q_c(G)$, and often (though not always) has a phase transition at non-zero temperature for any $q<q_c(G)$;
  \item Koteck{\'y}--Sokal--Swart~\cite[Section 1.4, (3)]{kotecky2014entropy} ask to prove the existence of an entropy-driven phase transition on $\Z^d$ for suitable pairs of $(q,d)$ and suggest that this holds for $q<q_c(\Z^d)$ for some function $q_c(\Z^d)$, possibly satisfying $q_c(\Z^d)\approx 2d$.
  \item Engbers--Galvin~\cite[Section 6.3]{engbers2012h2} write that it would be of great interest to prove long-range order for weighted graph homomorphisms on $\Z^d$ (including proper $q$-colorings) and deduce information on the Gibbs states of the model.
  \item Galvin--Kahn--Randall--Sorkin~\cite[Conjecture 1.3]{galvin2012phase} conjecture that, for any $q>3$, there are multiple maximal-entropy Gibbs states for proper $q$-colorings of $\Z^d$ when $d$ is sufficiently large.
  \item Feldheim and the authors ask in \cite[Section 8]{feldheim2013rigidity} and \cite[Section 1.3]{feldheim2015long} to show long-range order of the BSS type (with $\lfloor \frac q2 \rfloor$ colors predominant on one sublattice and the remaining $\lceil \frac q2 \rceil$ colors on the other sublattice) for each $q$ when $d$ is sufficiently large.
\end{itemize}
Our work resolves the prediction by exhibiting long-range order for all $q$ when $d$ is sufficiently large, and further allows for a quantitative power-law dependence between $q$ and $d$ (the companion paper \cite{peledspinka2018spin} addresses more general models including weighted graph homomorphisms). Compared with the aforementioned uniqueness of Gibbs states results which hold when $q>C d$, we see that a power-law dependence is best possible though the precise power is yet to be determined.

The previously addressed case of $q=3$ colors has a special additional structure as proper $3$-colorings of $\Z^d$ admit a height function representation. This special structure manifests in a natural cyclic order on the $6$ dominant patterns and is essential to the analysis presented in \cite{peled2010high} and  \cite{galvin2012phase}. Already the extension to low temperatures in \cite{feldheim2015long} is quite significant as the global height representation is lost, but the analysis there still relies on the height function existing locally, away from the rare places where the coloring is not proper. As nothing of this structure remains when the number of colors increases beyond 3, the previously used methods are insufficient for the analysis of proper $q$-colorings with any $q \ge 4$. Specific new challenges arising include the difficulty in identifying ordered regions (which, if any, dominant pattern does a vertex follow?), the many more ways in which the proper coloring can order and transition between the different orders (the large number of dominant patterns and their complex ``adjacency structure''), and the more significant role played by disordered regions (which do not follow any pattern) and sub-optimally ordered regions (which follow a non-dominant pattern). Consequently, finding a useful definition of ordered and disordered regions in a given coloring is already a non-trivial first step in the analysis of the $q\ge 4$ case (this was true also for the low-temperature $q=3$ case but a number of additional difficulties arise for proper $q$-colorings with $q\ge 4$).

A common ingredient in the proofs of long-range order for $q=3$ colors in $\Z^d$, as well as for the hard-core model, is the use of sophisticated \emph{contour methods}. The underlying idea is similar to the argument of Peierls -- identify regions of ``excitations'', i.e., deviations from the ordered state, show that any specific excitation is unlikely and use a union bound to show that the probability that there exists an excitation is small. However, the idea in this form fails for the proper $3$-coloring and hard-core models, as the probability of specific excitations is not sufficiently small to allow the use of the union bound. As a remedy, one is led to a ``coarse-graining'' technique, in which several different excitations are grouped together according to a common ``approximation'', the probability of each approximation is shown to be small, the number of approximations is shown to be small (compared with the number of excitations) and a union bound over approximations is then applied to show that the probability that there exists an excitation is small. The notion of approximation which turns out to be fruitful takes advantage of the following geometric property of the excitation regions in the $3$-coloring and hard-core models -- these regions have all their vertex boundary on one of the two sublattices of~$\Z^d$.
Such regions have been termed ``odd cutsets'' in \cite{peled2010high}. The idea to group such regions according to a common approximation can be traced back to the works of Korshunov and Sapozhenko~\cite{korshunov1981number,Korshunov1983Th,sapozhenko1987onthen,sapozhenko1989number,sapozhenko1991number} in the context of general bipartite graphs, with further developments and applications to statistical physics questions on $\Z^d$ made by Galvin~\cite{Galvin2003hammingcube,galvin2007sampling,galvin2008sampling}, Galvin--Kahn~\cite{galvin2004phase}, Galvin--Kahn--Randall--Sorkin~\cite{galvin2012phase}, Galvin--Randall~\cite{galvin2007torpid}, Galvin--Tetali~\cite{galvin2004slow,galvin2006slow}, Feldheim--Spinka~\cite{feldheim2015long,feldheim2016growth}, Peled~\cite{peled2010high} and Peled--Samotij~\cite{peled2014odd}.
This core idea is also used and further developed in this work.

In a parallel development, \emph{entropy methods} have been identified as a powerful tool to analyze models of graph homomorphisms. Pioneered by Kahn--Lawrentz~\cite{kahn1999generalized} in 1999 and Kahn~\cite{kahn2001entropy, Kahn2001hypercube} in 2001, the ideas were further developed by Galvin--Tetali~\cite{galvin2004weighted} (see also Lubetzky--Zhao~\cite{lubetzky2015replica}), Galvin~\cite{galvin2006bounding}, Madiman--Tetali \cite{madiman2010information} and Engbers--Galvin~\cite{engbers2012h1,engbers2012h2}. The basic method applies to graph homomorphisms from a finite bipartite regular (or bi-regular) graph $G$ to a general finite graph~$H$. Relying on Shearer's inequality \cite{chung1986some}, it implies that most such graph homomorphisms are \emph{locally} ordered at most vertices, in the sense that the neighborhood of all but $\epsilon(\Delta(G))$ fraction of the vertices follow some dominant pattern (as in Section~\ref{sec:general_spin_systems}), where $\Delta(G)$ is the degree of $G$ and $\epsilon(\Delta)$ is a function satisfying $\epsilon(\Delta)\to0$ as $\Delta\to\infty$. This suffices to estimate rather accurately the exponential growth rate of the number of graph homomorphisms, up to an error term which decreases as the degree of $G$ grows (for proper colorings of $\Z^d$ the obtained error decays as $C(q)/d$ as $d\to\infty$. Our results imply improved error bounds, see Section~\ref{sec:max-entropy-states}). Generalizations from graph homomorphisms to discrete spin systems of the type considered in Section~\ref{sec:general_spin_systems} are possible~\cite{galvin2004weighted, galvin2006bounding}. The method does not generally imply \emph{global} ordering in typical graph homomorphisms, as it allows for different regions to be ordered according to different dominant patterns. Nonetheless, it was discovered in \cite{engbers2012h2} that global $(A,B)$-ordering follows on hypercube graphs -- discrete tori with vertex set $\{0,1,\dots, 2m-1\}^d$ which are considered with $m\ge 1$ \emph{fixed} and $d\to\infty$ -- due to the interplay between their isoperimetric properties and the smallness of the function $\epsilon$ above. One may further allow $m$ to grow slowly with $d$ but this approach does not extend to the $\Z^d$ lattice~\cite[Section 6.3]{engbers2012h2}.

The main technical novelty introduced in this paper is a non-trivial synthesis of the contour and entropy methods discussed above. Our approach begins by identifying ordered and disordered regions in a given coloring, where vertices are classified according to the coloring of their local neighborhoods. The abundance of possible local colorings gives rise to a complicated classification where regions ordered according to one dominant pattern may overlap with those of another and where many types of disordered behavior may arise. The contours separating the different regions are then approximated with a similar, albeit more involved, technique to that used in the $q=3$ case. It then remains to prove that any given picture of approximated contours is unlikely, in order to deduce long-range order via a union bound. This is resolved here by use of the entropy method extended in the following two manners: (i) The method is applied to a partial set of colorings, restricted by various pieces of information known from the contour picture, and these restrictions are taken into account by the entropy estimates to produce a sufficiently tight bound. (ii) The method is applied to colorings defined on bounded subsets of $\Z^d$, specifically on the disordered regions and on the interfaces between ordered regions. This is in contrast with previous applications of the method where it was applied to the full set of colorings (or graph homomorphisms), which were themselves defined on a regular graph. New difficulties thus arise in integrating the external information with the entropy estimates and in carefully tracking and cancelling the boundary terms arising from the irregularity of the bounded subsets. A detailed overview of the method is given in Section~\ref{sec:proof-overview}.

We end the discussion with several questions for future research.
\begin{enumerate}[leftmargin=20pt]
  \item Determine for all pairs $(q,d)$ whether there is a unique maximal-entropy Gibbs state. Is the dependence on $q$ monotone in the sense that there is a $q_c(d)$ with multiple maximal-entropy Gibbs states existing if and only if $q<q_c(d)$? Does $\frac{q_c(d)}{d}$ tend to a positive limit as $d\to\infty$? The same may be asked regarding uniqueness among all Gibbs states (not necessarily of maximal entropy). As mentioned in the introduction, frozen Gibbs states exist if and only if $q\le d+1$~\cite{alon2019mixing} while uniqueness (among all Gibbs states) is known when $q>3.53d$~\cite{goldberg2005strong}.

      For comparison, we mention that the $\Delta$-regular tree case was studied by Brightwell--Winkler~\cite{brightwell2002random} who noted that frozen Gibbs states exist whenever $q\le \Delta$, and by Jonasson~\cite{jonasson2002uniqueness} who proved uniqueness whenever $q\ge \Delta+1$ and $\Delta$ is large.
  \item Prove an analogous result to Theorem~\ref{thm:long-range-order} for free and periodic boundary conditions. Our methods should be relevant also for these cases, with the periodic case with even side length possibly being a direct extension (see~\cite[Section~8]{feldheim2013rigidity} for a prediction regarding $3$-colorings of tori with odd side length), and the free case seeming more difficult as issues regarding excitations (deviations from the long-range order) touching the boundary of the domain must be dealt with carefully. Of course, the characterization of Gibbs states given in Theorem~\ref{thm:characterization_of_Gibbs_states} does not depend on the choice of boundary conditions.
  \item As discussed, our results apply also in low dimensions provided that the underlying lattice is enhanced to $\Z^{d_1}\times\T_{2m}^{d_2}$, $m\ge 1$ integer, $d_1\ge 2$ and $d=d_1+d_2$ satisfying~\eqref{eq:dim-assump} (proper $3$-colorings were considered in this setting in~\cite{peled2010high}). Another natural enhancement used in low-dimensional lattices, e.g., in the context of percolation~\cite{slade2006lace}, is the spread-out lattice. In our context, this corresponds to $\Z^d$ with additional edges connecting every two vertices of different parity whose graph distance in $\Z^d$ is at most some fixed threshold $M$. We expect our results to hold also with this enhancement provided $d\ge 2$ and $M$ is sufficiently large as a function of $q$ (raising $d$ should only assist the long-range order).
\end{enumerate}

\subsection{Organization}
The rest of the paper is organized as follows.
In Section~\ref{sec:proof-overview}, we provide an overview of the proof. In Section~\ref{sec:preliminaries}, definitions and preliminary results which will be needed throughout the paper are given. In Section~\ref{sec:high-level-proof}, we give the main steps of the proof of Theorem~\ref{thm:long-range-order}, including the definitions of breakups and approximations and the statements of several propositions which are then used to deduce Theorem~\ref{thm:long-range-order}. In Section~\ref{sec:breakup}, we prove the propositions about breakups (existence of non-trivial breakup, almost-sure absence of infinite breakups, bounds on the probability of breakups). In Section~\ref{sec:shift-trans}, we prove Lemma~\ref{lem:bound-on-pseudo-breakup} which provides a general bound on the probability of an event and which is used in the proofs in Section~\ref{sec:prob-of-given-breakup} and Section~\ref{sec:prob-of-approx}.
In Section~\ref{sec:approx}, we prove Proposition~\ref{prop:family-of-odd-approx} about the exists of a small family of approximations. Finally, in Section~\ref{sec:gibbs}, we prove results about the infinite-volume Gibbs states, namely, Theorem~\ref{thm:existence_Gibbs_states} and Theorem~\ref{thm:characterization_of_Gibbs_states}.

\subsection{Acknowledgments}
We thank Raimundo Brice\~no, Nishant Chandgotia, Ohad Feldheim and Wojciech Samotij for early discussions on proper colorings and other graph homomorphisms. We are grateful to Christian Borgs for valuable advice on the way to present the material of this paper and its companion~\cite{peledspinka2018spin}. We thank Michael Aizenman, Jeff Kahn, Eyal Lubetzky, Dana Randall, Alan Sokal, Prasad Tetali and Peter Winkler for useful discussions and encouragement. The presentation benefited significantly from the insightful comments of two anonymous referees.

\section{Overview of proof}
\label{sec:proof-overview}

In this section we give a high-level view of the proof of Theorem~\ref{thm:long-range-order}. Apart from the definitions in Section~\ref{sec:proof-overview-Z}, this overview will be not be used in the detailed proofs of the later sections.

We recall that $(A,B)$ is a dominant pattern if $A,B\subset[q]$ are disjoint and $\{|A|,|B|\}=\{\lfloor \frac q2 \rfloor, \lceil \frac q2 \rceil\}$. Throughout this section, we fix a domain $\Lambda \subset \Z^d$ and a dominant pattern
\begin{equation}\label{eq:P_0_def}
P_0=(A_0,B_0) \qquad\text{such that}\qquad |A_0| = \lfloor\tfrac{q}{2}\rfloor,~|B_0| =\lceil\tfrac{q}{2}\rceil .
\end{equation}
We think of $P_0$ as the boundary pattern so that we will later consider a coloring chosen from $\Pr_{\Lambda,P_0}$.

We use $\partial U$ to denote the edge-boundary of a set $U \subset \Z^d$, and $N(U)$ to denote its neighborhood (vertices adjacent to some vertex in $U$). We also denote $\intB U := U \cap N(U^c)$, $\extB U := N(U) \setminus U$, $\intextB U := \intB U \cup \extB U$, $U^+ = U^{+1} := U \cup N(U)$ and, inductively, $U^{+j}:=(U^{+(j-1)})^+$ for $j>1$. We say that $U$ is an even (odd) set if $\intB U$ is contained in the even (odd) sublattice of $\Z^d$. An even (odd) set $U$ is called \emph{regular} if both it and its complement contain no isolated vertices. See Section~\ref{sec:preliminaries} for more notation and definitions.

\subsection{A toy scenario}\label{sec:toy scenario}
\begin{figure}
		\includegraphics[scale=0.55]{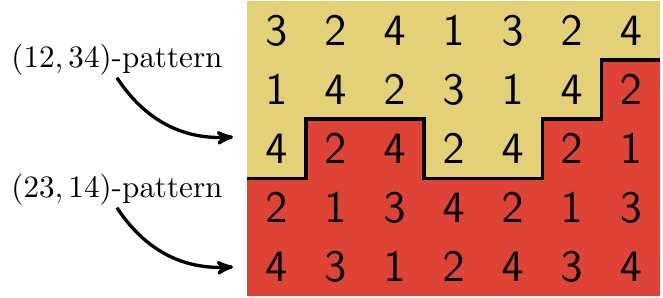}\qquad\qquad
		\includegraphics[scale=0.55]{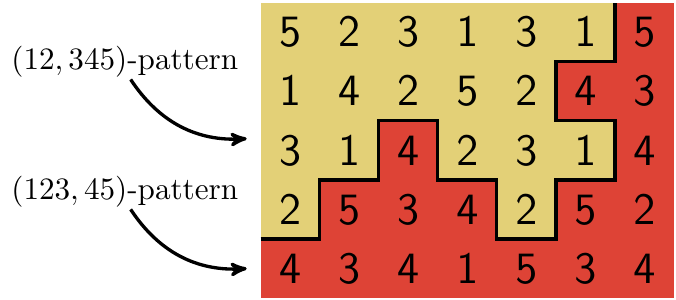}
	\caption{Part of the interface between the regions associated to different dominant patterns in the toy scenario of Section~\ref{sec:toy scenario} (left: $q=4$ colors, right: $q=5$ colors). In practice, our definitions are motivated by the odd $q$ case and always associate sets with fixed boundary parity (even or odd sets) to each dominant pattern, regardless of the parity of $q$ (see Section~\ref{sec:proof-overview-Z}).}
%
%
	\label{fig:interfaces}
\end{figure}
To gain intuition, let us analyze the ``entropic loss'' in the toy scenario in which the $P_0$-pattern is disturbed by a single droplet of a different dominant pattern $P=(A,B)$; see Figure~\ref{fig:interfaces}. More precisely, let $U\subset\Z^d$ be such that $U^+\subset\Lambda$ and let $n(U)$ be the number of proper colorings of $\Lambda$, for which $U^+$ is in the $P$-pattern and $(\Lambda\setminus U)^+$ is in the $P_0$-pattern. A straightforward computation yields that, when $q$ is even,
\[ \frac{n(U)}{n(\emptyset)} \le \left(\frac{q-2}{q}\right)^{|\intextB U|},\]
with equality if and only if $|A_0 \Delta A|=2$. When $q$ is odd, a straightforward (though somewhat more involved) computation yields that
\[ \frac{n(U)}{n(\emptyset)} \le \left(\frac{q-1}{q+1}\right)^{\frac{1}{2d} |\partial U|},\]
with equality if and only if either $U$ is an odd set and $A_0 \subset A$ or $U$ is an even set and $B_0 \subset B$. This example shows a difference in behavior between the even and odd $q$ cases, with the odd case more difficult due to the lower cost of creating interfaces between $P_0$- and $P$-ordered regions.
It is the odd $q$ case that motivates many of our definitions and ideas, including the idea that such interfaces should be even or odd, according to the relative size of $A_0$ and $A$.
Thus some of our definitions are somewhat less natural in the even $q$ case.

\subsection{Identification of ordered and disordered regions}
\label{sec:proof-overview-Z}
Given a proper $q$-coloring $f$ of $\Z^d$, we wish first to identify regions where $f$ follows, in a suitable sense, a dominant pattern. A first idea is that the decision regarding a vertex $v$ will be made based on the values that $f$ takes on the \emph{neighbors} of $v$. Indeed, the color that $v$ takes cannot itself be sufficient as it has only $q$ options whereas there are many more dominant patterns, but the colors of the neighbors turn out to suit the job. A second idea, motivated by the toy scenario described earlier and also by questions of approximation of contours which will be soon described, is that each region will be a (regular) even or odd set. More precisely, the region associated with a dominant pattern $(A,B)$ is an even set if $|A|\le|B|$ and an odd set if $|A|>|B|$ (thus odd sets appear only if $q$ is odd). Let us now describe the regions precisely. Let $\phasedom$ be the set of all dominant patterns. For each $P=(A,B)\in\phasedom$, define the terms
\begin{equation}\label{eq:P-even-odd}
  \text{$P$-even} = \begin{cases}
    \text{even} &\text{if } |A|\le |B|\\
    \text{odd} &\text{if } |A|>|B|
  \end{cases}\qquad\text{and similarly}\qquad \text{$P$-odd} = \begin{cases}
    \text{odd} &\text{if } |A|\le |B|\\
    \text{even} &\text{if } |A|>|B|
  \end{cases}.
\end{equation}
Thus, for instance, if $|A|\le |B|$ then even vertices (having even sum of coordinates) are $P$-even and odd vertices are $P$-odd. The region associated to $P$ is denoted $Z_P(f)$ and defined by
\begin{equation}\label{eq:Z-def}
Z_P = Z_P(f) := \big\{ v \in \Z^d : v\text{ is $P$-odd},~ N(v)\text{ is in the $P$-pattern} \big\}^+.
\end{equation}
Figure~\ref{fig:breakup} depicts these sets in examples. For technical reasons, only $P$-odd vertices whose neighbors are in the $P$-pattern are included in $Z_P$, and then $Z_P$ is taken to be the smallest $P$-even set containing them. Note that a $P$-odd vertex in $Z_P$ is not itself required to be in the $P$-pattern, whereas a $P$-even vertex in $Z_P$ is necessarily in the $P$-pattern, but need not have its neighbors in the $P$-pattern. In addition, there may be $P$-even vertices which are not in $Z_P$ although their neighbors are in the $P$-pattern. These somewhat undesirable consequences of our definition are allowed in order to ensure that $Z_P$ is a regular $P$-even set, which will be important in the proof.

Having defined the regions $(Z_P)_{P\in\phasedom}$, let us examine more closely their interrelations. It is possible for a vertex $v$ to belong to two (or more) of the $Z_P$ and also possible that it lies outside all of the $Z_P$. These possibilities are captured by the following definitions (see Figure~\ref{fig:breakup}):
\[ Z_\overlap := \bigcup_{P \neq Q} (Z_P \cap Z_Q) \qquad\text{and}\qquad Z_\bad := \bigcap_P (Z_P)^c. \]
Regions of these types, along with the boundaries of $Z_P$, are regions where the coloring $f$ does not achieve its maximal entropy per vertex, in a way which is quantified later. It will be our task to prove that such regions are not numerous and this will lead to a proof of Theorem~\ref{thm:long-range-order}. To this end, we define
\begin{equation}\label{eq:Z_*-def}
Z_* = Z_*(f) := \bigcup_P \intextB Z_P \cup Z_\overlap \cup Z_\bad .
\end{equation}
The region $Z_*$ plays a similar role in our analysis as the contours used in arguments of the Peierls or Pirogov-Sinai type.

\begin{figure}
	\centering
	\includegraphics[scale=0.37]{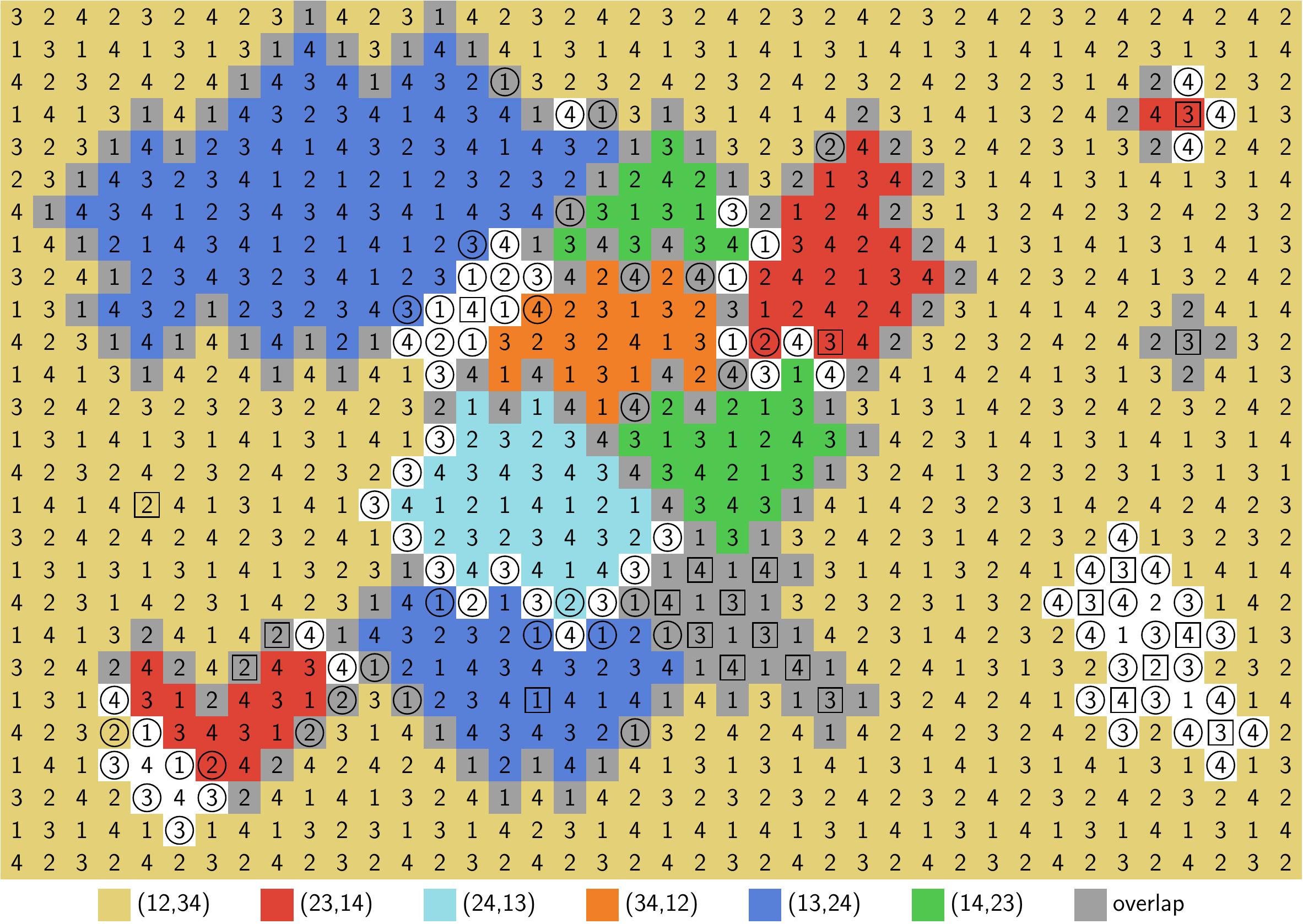}\vspace{4pt}
	
	\includegraphics[scale=0.37]{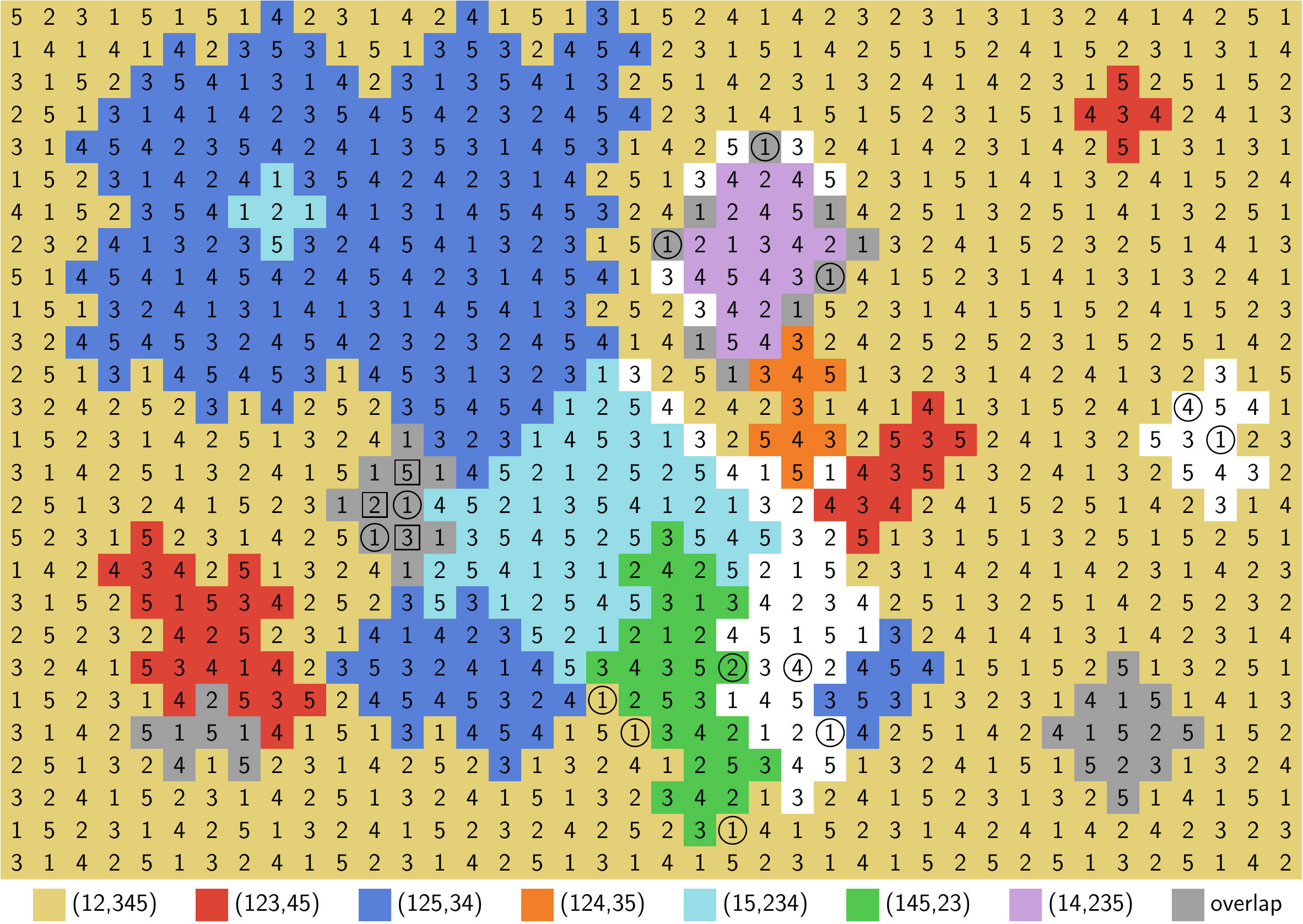}
	\captionsetup{width=0.95\textwidth, font=small}
	\caption{Proper $q$-colorings (top: $q=4$, bottom: $q=5$) and the associated identification of ordered and disordered regions (each $Z_P$ has a different color, with gray indicating $Z_\overlap$ and white indicating $Z_\bad$). Non-dominant vertices (defined in Section~\ref{sec:Shearer_overview}) are depicted: squares indicate vertices whose neighbors are assigned less than $\lfloor \frac q2 \rfloor$ different values, whereas circles indicate vertices whose neighbors are assigned more than $\lceil \frac q2 \rceil$ different values.}
	\label{fig:breakup}
\end{figure}

\begin{figure}
	\centering
	\includegraphics[scale=0.204]{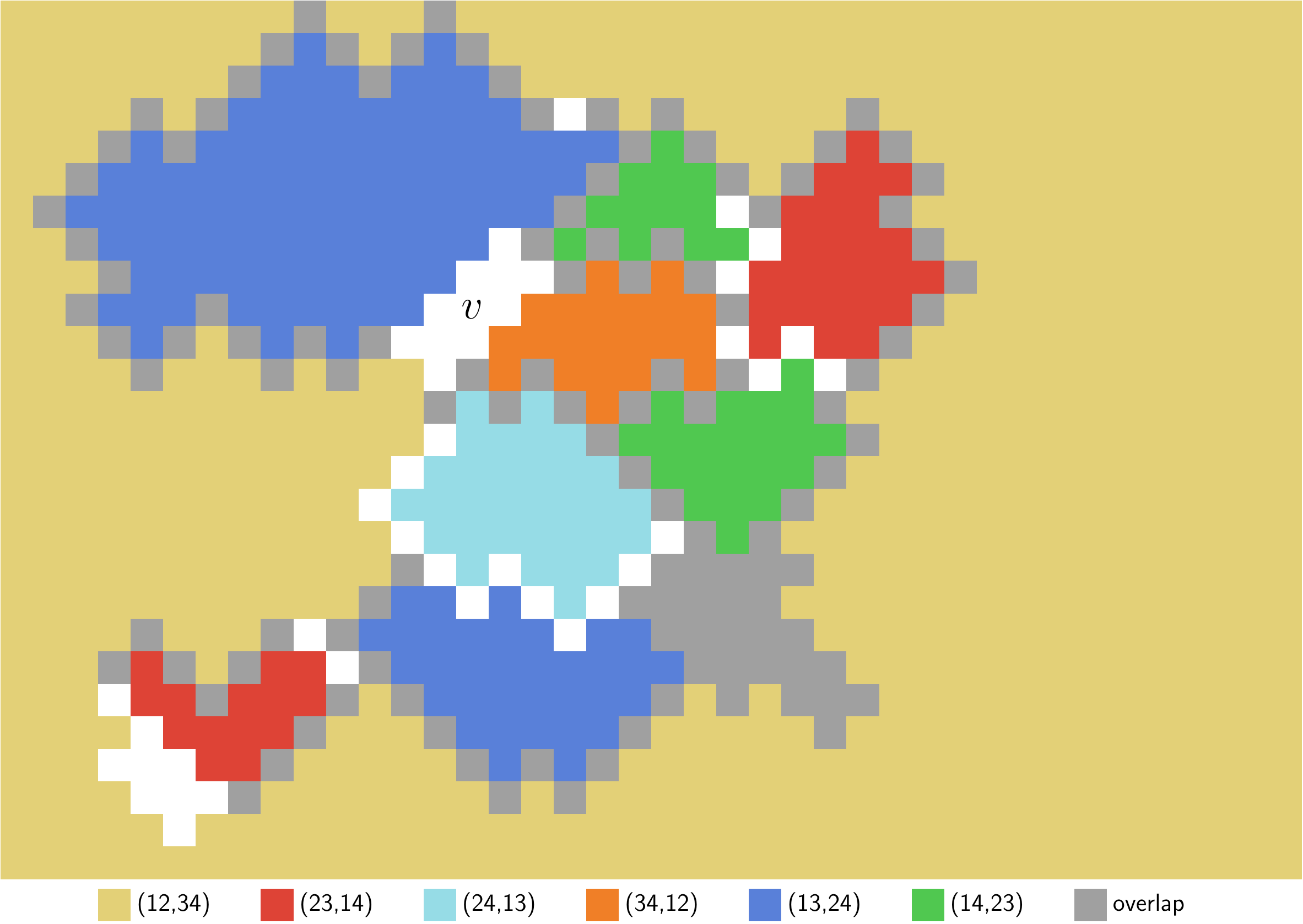}~
	\includegraphics[scale=0.204]{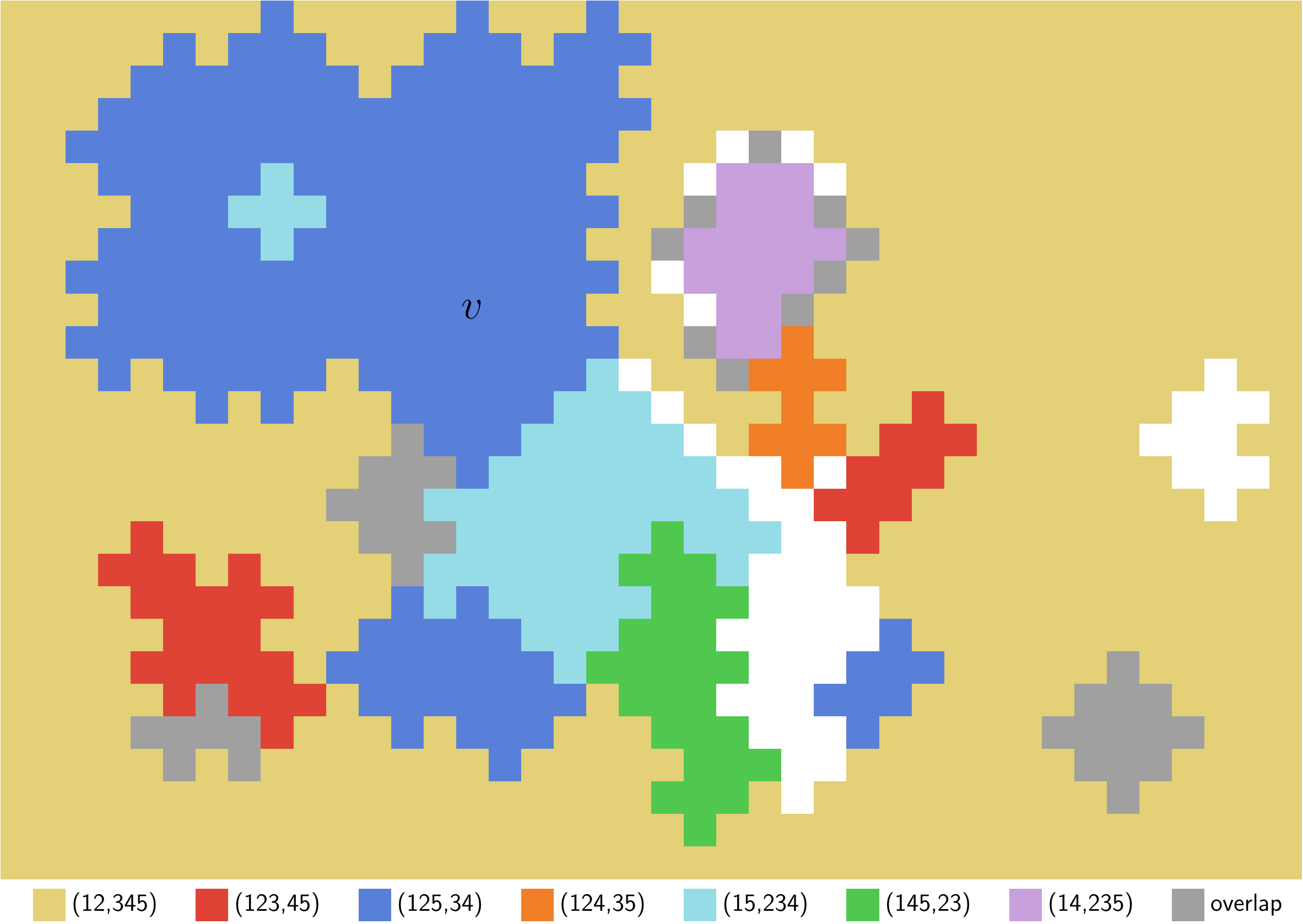}
	\captionsetup{width=0.95\textwidth, font=small}
	\caption{Breakups seen from $v$ of the colorings shown in Figure~\ref{fig:breakup} (left: $q=4$, right: $q=5$). Such breakups are not unique; for instance, the $q=5$ case may further include the small violation of the boundary pattern in the top-right corner of Figure~\ref{fig:breakup} (bottom).
Each $X_P$ has a different color, with gray indicating $X_\overlap$ and white indicating $X_\bad$ (the information of a breakup also includes the classification of the region $X_\overlap$ into various $X_P$, though this is not depicted in the figure).}
	\label{fig:breakup2}
\end{figure}

\subsection{Breakups}
\label{sec:proof-overview-breakup}
With Theorem~\ref{thm:long-range-order} in mind, let $f$ be sampled from $\Pr_{\Lambda,P_0}$ and fix a vertex $v\in\Lambda$. It is convenient to extend $f$ to a coloring of $\Z^d$ by coloring vertices of $\Lambda^c$ independently and uniformly from $A_0$ or $B_0$ according to their parity (so that they are in the $P_0$-pattern). The collection $(Z_P)$ then identifies ordered and disordered regions in $f$. Our goal is to show that $v$ is typically in the $P_0$-pattern.
One checks that $Z_{P} \setminus Z_{\overlap}$ is in the $P$-pattern, and therefore it suffices to show that, with high probability, $Z_{P_0}$ is the unique set among $(Z_P)$ to which $v$ belongs. This, in turn, follows by showing that there is a path from $v$ to infinity avoiding $Z_*$. If no such path exists, there needs to be a \emph{connected} component of $Z_*^+$ which disconnects~$v$ from infinity. Our focus is then on these connected components and this motivates the following notion of a breakup seen from $v$, which encodes partial information from $(Z_P)$ relevant to these components.

A breakup of $f$ is a collection $X = (X_P)_{P\in\phasedom}$ of regular $P$-even subsets of $\Z^d$, from which one defines $X_*$ in the same manner as $Z_*$ is defined from $(Z_P)$, with the following properties: (i) $\Lambda^c \subset X_{P_0}$, and (ii) For each $P\in\phasedom$, every $P$-odd vertex $u \in X_*^{+5}$ satisfies that $u \in X_P$ if and only if $u\in Z_P$. This definition allows $f$ to have multiple breakups. A trivial example of a breakup, for which $X_* = \emptyset$, is obtained when $X_{P_0}=\Z^d$ while $X_P = \emptyset$ for all $P\neq P_0$. A second example of a breakup is $X = (Z_P)$, for which $X_* = Z_*$. More generally, the idea behind the definition is that some subset of the connected components of $Z_*$ is selected (though not every choice is possible) and then $X$ is set up in such a way that $X_*$ is exactly the union of the selected components, and each $X_P$ coincides with $Z_P$ in a suitable neighborhood of $X_*$. A breakup is called non-trivial if $X_*\neq\emptyset$. A breakup is said to be seen from $v$ if every finite connected component of $X_*^{+5}$ disconnects $v$ from infinity. It will be shown that if $v$ is not in the $P_0$-pattern then there exists a non-trivial breakup seen from $v$ (see Section~\ref{sec:breakup-def}). Figure~\ref{fig:breakup2} shows possible breakups seen from $v$.

We remark that the use of the enlarged neighborhood $X_*^{+5}$ yields a wide region around $X_*$ where, for each $P\in\phasedom$, all vertices in $X_P$ are actually in the $P$-pattern. This will be convenient in the proof (though the specific number $5$ is not important and could just as well be taken larger).

\subsection{Approximations}
Suppose again that $f$ is sampled from $\Pr_{\Lambda,P_0}$ and $v\in\Lambda$. Following the previous discussion, in order to deduce Theorem~\ref{thm:long-range-order}, it suffices to bound the probability that $f$ has a non-trivial breakup seen from $v$. Our method of proof is, in essence, an involved variant of the Peierls argument. The standard argument consists of two parts: obtaining a bound on the probability that a given $X$ is a breakup of $f$ (this is discussed in the subsequent section), and concluding via a union bound that $f$ is unlikely to have any non-trivial breakup seen from~$v$.
However, the toy scenario considered in Section~\ref{sec:toy scenario} shows that the ``entropic loss per edge'' on the interfaces between different $X_P$ may be small. Indeed, the bound obtained on the probability that a given $X$ is a breakup of $f$ does not allow to conclude the proof (via the union bound) as the number of possible breakups seen from $v$ is too large in comparison. We thus vary the standard argument as follows. We employ a delicate coarse-graining scheme of the possible breakups according to their approximate structure, i.e., multiple breakups are grouped together according to a common ``approximation''.
The scheme is, on the one hand, coarse enough to ensure that a relatively small number of approximations suffices to cover all possible breakups seen from $v$, while it is, on the other hand, sufficiently fine to allow a useful bound on the probability that $f$ has a non-trivial breakup with a given approximation. We conclude via a union bound over the possible approximations.

The crucial property of breakups which allows their approximation is that each $X_P$ is either regular even or regular odd. Let us briefly discuss the theory of such sets: The number of odd sets $U\subset\Z^d$ which are connected, have connected complement, contain the origin and have $|\partial U|=L$ boundary plaquettes grows as $2^{(\frac{1+\varepsilon_d}{2d})L}$ for $L$ large~\cite{feldheim2016growth}, with $2^{-2d}\le\varepsilon_d\le\frac{C\log^{3/2} d}{\sqrt{d}}$. This contrasts with the same count when the set is not required to be odd, which grows faster, roughly as~$e^{\frac{c\log d}{d} L}$~\cite{lebowitz1998improved,balister2007counting}. The different growth rates are indicative of a deeper structural difference. Typical odd sets of the above type have a macroscopic shape (e.g., an axis-parallel box) from which they deviate on the microscopic scale, while sets of the above type without the parity restriction should scale to integrated super-Brownian excursion \cite{lubensky1979statistics, slade1999lattice}. The distinction between these very different behaviors is akin to the breathing transition undergone by random surfaces~\cite[Section~7.3]{fernandez2013random}. This phenomenon has been exploited in previous works, e.g., \cite{galvin2004phase,peled2010high,feldheim2015long}, to provide a natural coarse-graining scheme for odd sets, grouping them according to their macroscopic shape, and noting that this shape has significantly less entropy in high dimensions than the odd sets themselves (of order at most~$\big(\frac{\log d}{d}\big)^{3/2} L$). We proceed in the same manner here, extending the previous schemes from a single $X_P$ to breakups.

It is natural to approximate breakups by applying the previous coarse-graining techniques separately to each $X_P$. This can indeed be done, but due to the amount of dominant patterns it leads to a version of Theorem~\ref{thm:long-range-order} which requires the dimension $d$ to be larger than an exponential function of $q$, rather than the stated power-law dependence~\eqref{eq:dim-assump}. Instead, we use a more sophisticated scheme which takes into account the interplay between the different $X_P$.

\begin{figure}
	\centering
	\includegraphics[scale=0.204]{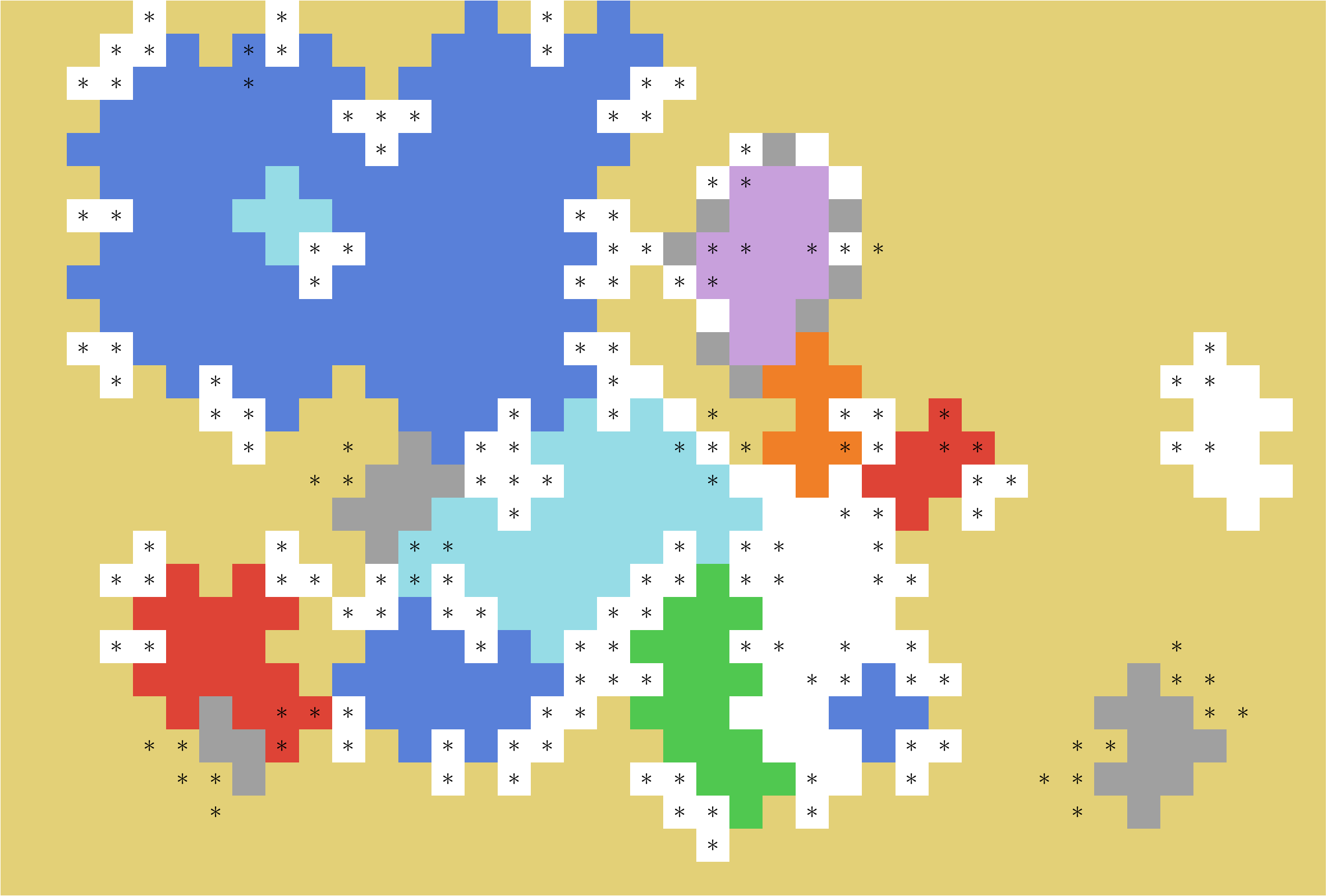}~
	\includegraphics[scale=0.204]{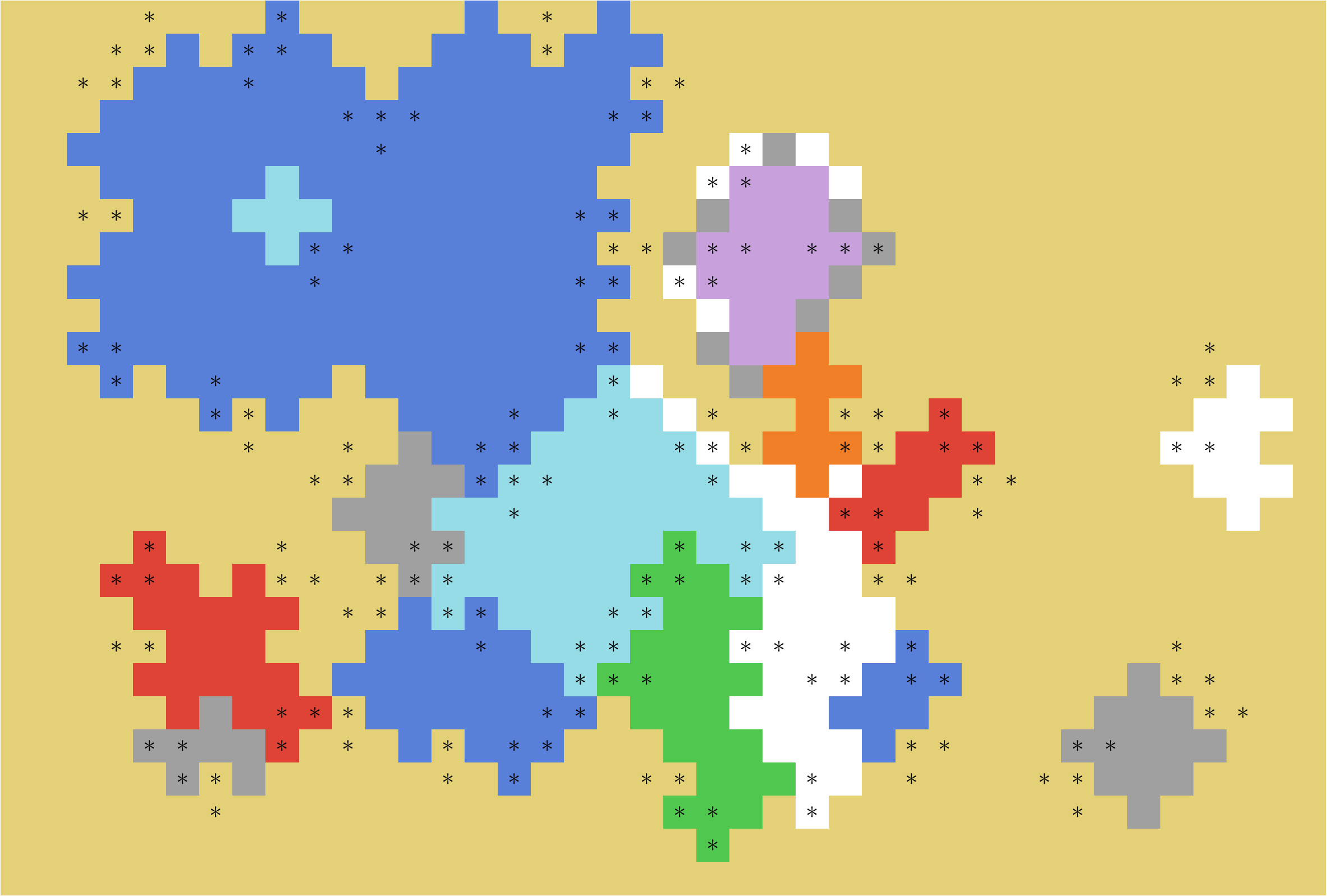}\vspace{-3pt}
    \includegraphics[scale=0.26]{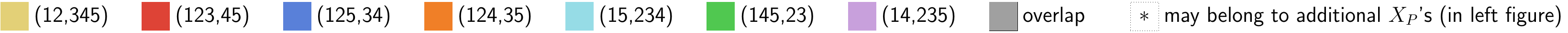}
	\captionsetup{width=0.95\textwidth, font=small}
	\caption{An illustration of an approximation (left) of a breakup (right). An approximation provides partial information on a breakup. On the left, the colors represent regions belonging to a single $A_P$ -- these regions are known to belong to $X_P$ and are not known to belong to any other $X_Q$. A gray background indicates regions belonging to two or more $A_P$ -- these regions are known to belong to $X_\overlap$. A white background indicates regions which do not belong to any $A_P$. A star ($*$) indicates the set $A^{**}$ -- these regions may belong to additional $X_P$'s. In particular, a white background with no star indicates regions which do not belong to $A^{**}$ or to any $A_P$ -- these regions are known to belong to $X_\bad$. The stars are also shown on the right to ease comparison. We note that the approximations used in the proof carry additional information; see Section~\ref{sec:approx-overview}.}
	\label{fig:approx}
\end{figure}

An approximation of a breakup $X=(X_P)$ is a collection $A = ((A_P)_{P \in \phasedom}, A^*,A^{**})$ of subsets of $\Z^d$ which provides partial information on $X$. Its precise definition is given in Section~\ref{sec:approx-overview} but we mention here that it satisfies that $A_P \subset X_P \subset A_P \cup A^{**}$ for all $P$. Thus $A_P$ is a region known to be in $X_P$ while $A^{**}$ is a region on which the classification into the various $(X_P)$ is not fully specified (so that a single $A$ may approximate many breakups). Further information is provided through the subset $A^* \subset A^{**}$ and additional properties ensure that $A^{**}$ is not large and that it is only present near $X_*$. See Figure~\ref{fig:approx} for an illustration.

\subsection{Repair transformation}\label{sec:repair transformation}
We proceed to explain, for a given $X=(X_P)$, how to bound the probability that $X$ is a breakup of $f$, when $f$ is sampled from $\Pr_{\Lambda,P_0}$. In the full proof the arguments need to be adapted to the case that only an approximation of $X$ is given rather than $X$ itself, but this adaptation is not the essence of the argument so our focus in the overview is on the case that $X$ is given.

Let $\Omega$ be the set of proper colorings having $X$ as a breakup.
To establish the desired bound on $\Pr_{\Lambda,P_0}(\Omega)$, we apply the following one-to-many transformation to every coloring $f \in \Omega$: (i) Erase the colors at all vertices of $X_*$. (ii) For each dominant pattern $P=(A,B)$,
apply a permutation of $[q]$ which takes $P$ to $P_0$ to the colors of $f$ on $X_P\setminus X_*$, and also, in the case that $|A|>|B|$, shift the coloring in $X_P\setminus X_*$ by a single lattice site in some fixed direction~(such a shift was first used by Dobrushin for the hard-core model~\cite{dobrushin1968problem}). (iii) Arbitrarily assign colors in the $P_0$-pattern at all remaining vertices (making the transformation multiple valued). See Figure~\ref{fig:repairmap} for an illustration.

Noting that the resulting configuration is always a proper coloring, and that no entropy is lost in step (ii), it remains to show that the entropy gain in step (iii) is much larger than the entropy loss in step (i). The gain in step (iii) is either $\log \lfloor \tfrac{q}{2} \rfloor$ or $\log \lceil \tfrac{q}{2} \rceil$ per vertex according to its parity, making the entropy gain an easily computable quantity. The main challenge is thus to bound the loss in step (i), and the method used for this purpose is described next.

\begin{figure}
	\centering
	\captionsetup{font=small}
	\begin{subfigure}[t]{0.5\textwidth}
		\includegraphics[scale=0.204]{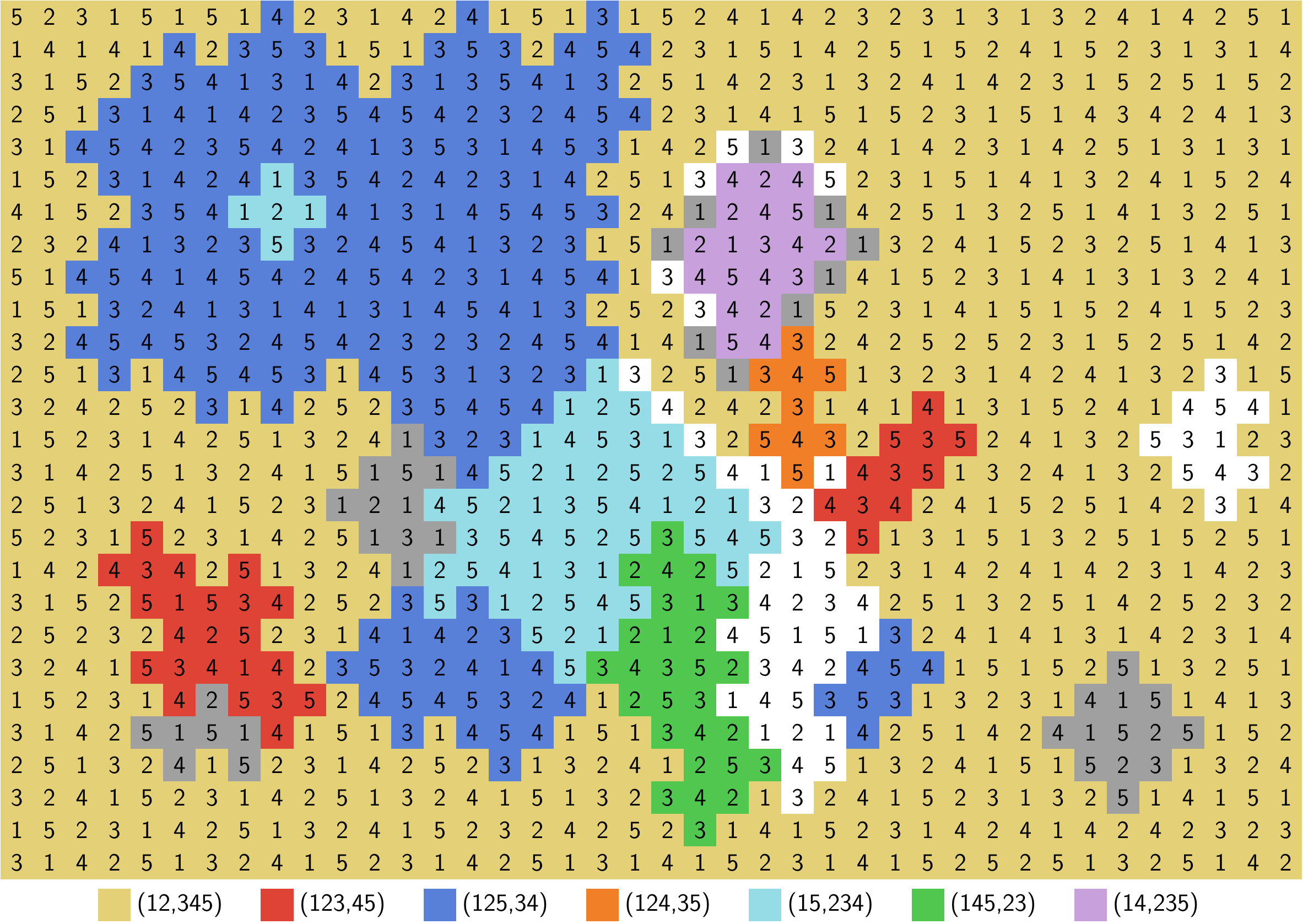}
		\caption{A coloring having a breakup $X$.}
	\end{subfigure}\,\,%
	\begin{subfigure}[t]{0.5\textwidth}
		\includegraphics[scale=0.204]{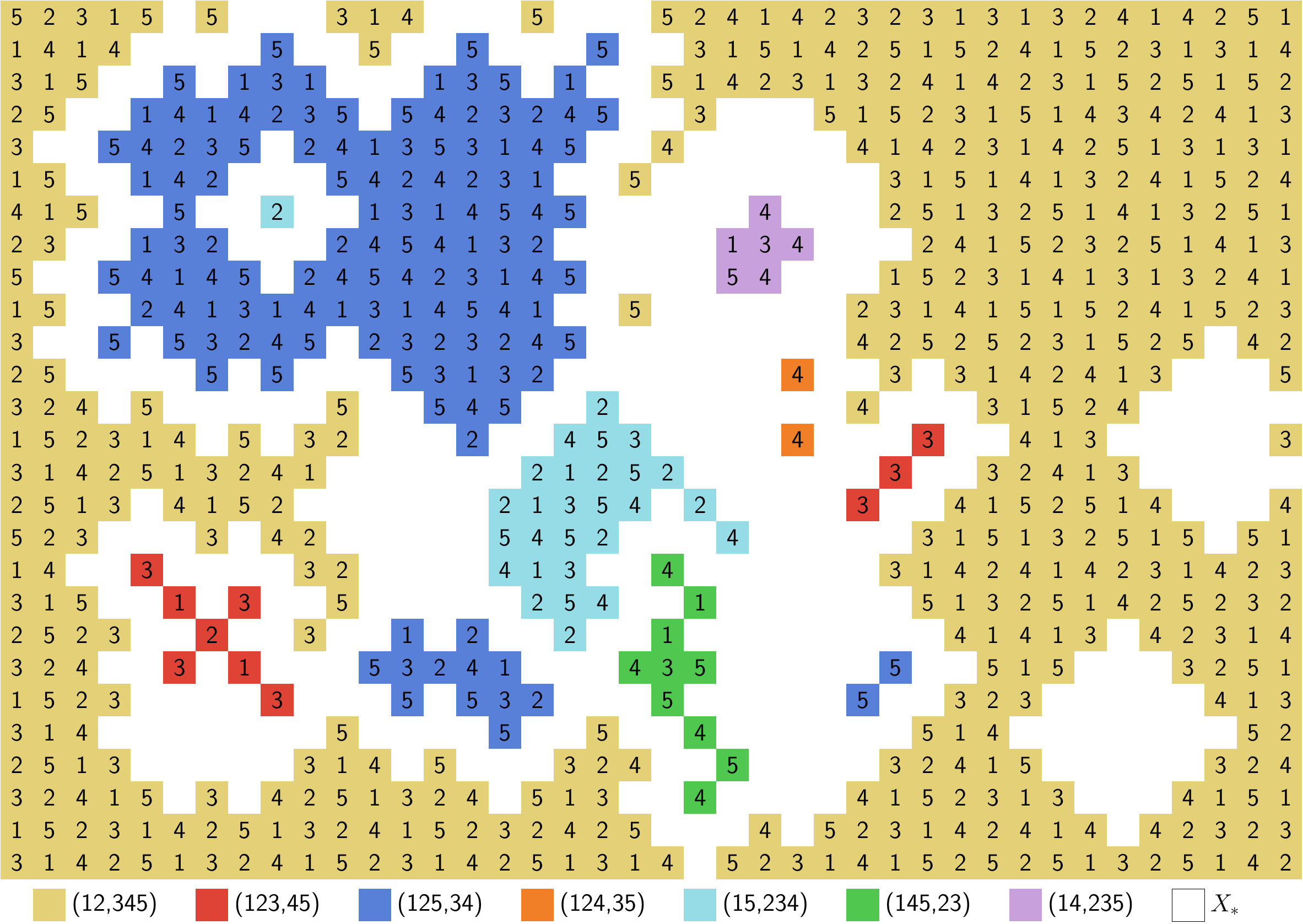}
		\caption{Step (i): colors in $X_*$ are erased.}
	\end{subfigure}
	\vspace{5pt}
	
	\begin{subfigure}[t]{0.5\textwidth}
		\includegraphics[scale=0.204]{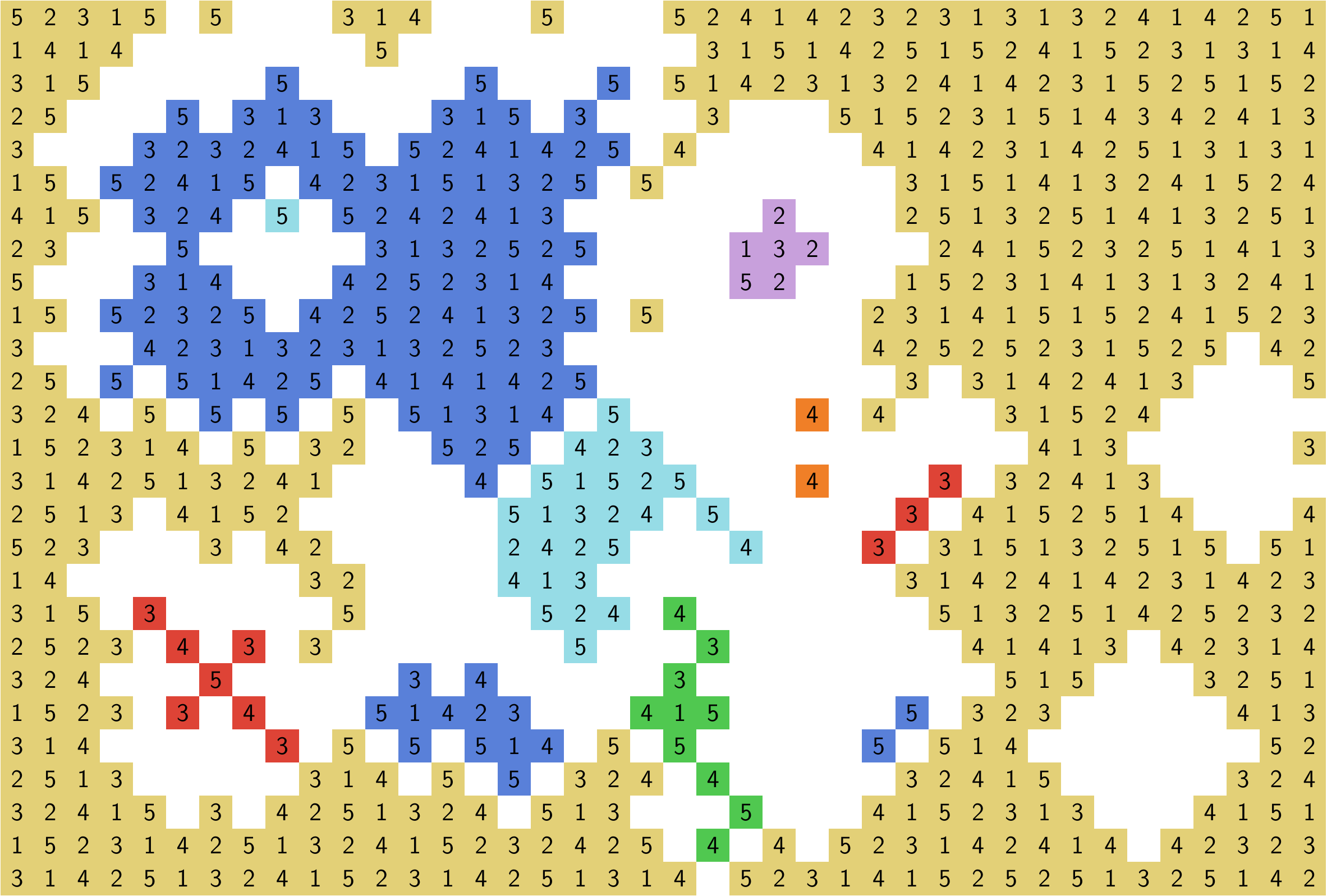}
		\caption{Step (ii): colors in $X_P$ are permuted and shifted.}
	\end{subfigure}\,\,%
	\begin{subfigure}[t]{0.5\textwidth}
		\includegraphics[scale=0.204]{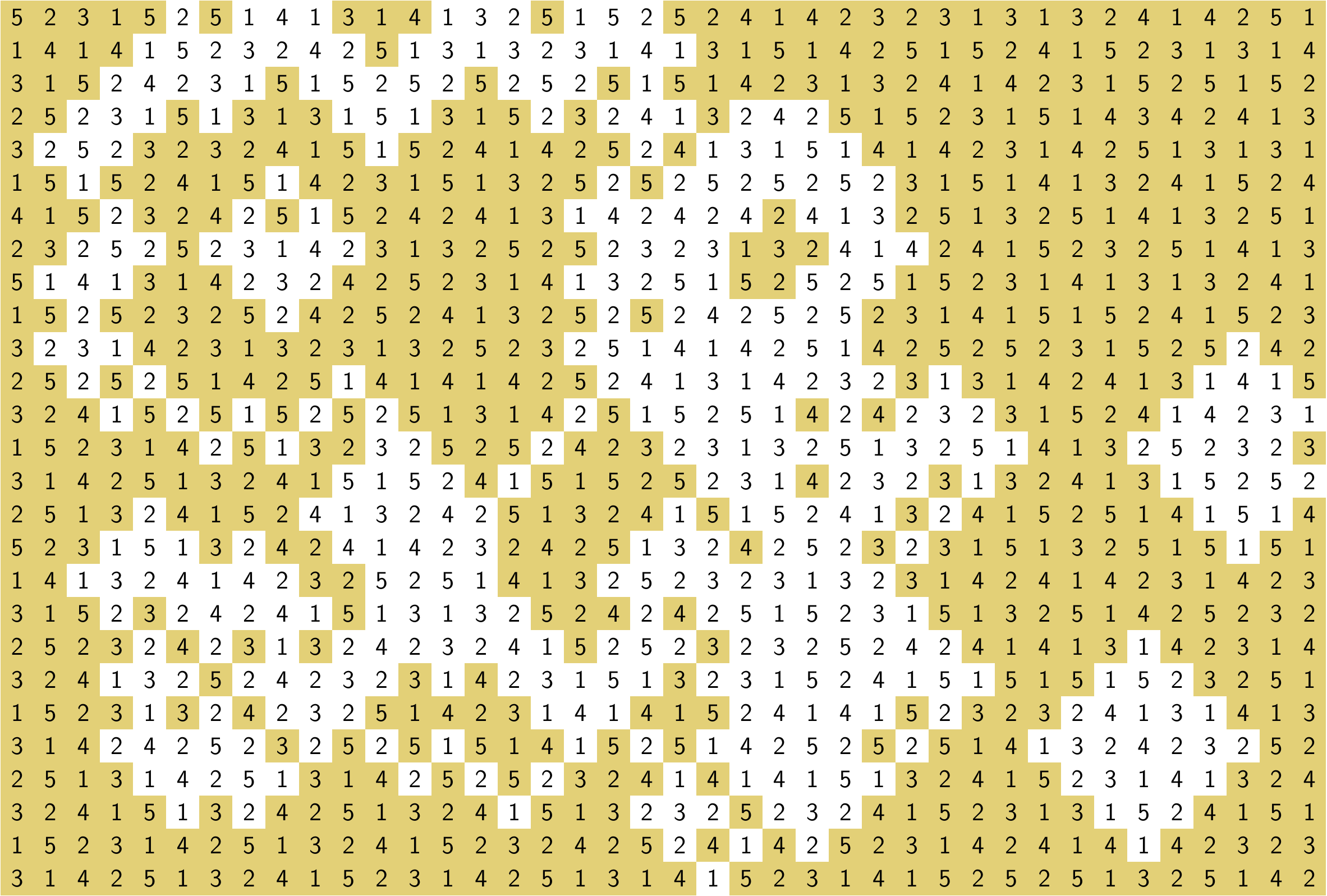}
		\caption{Step (iii): empty sites are colored in the $P_0$-pattern.}
	\end{subfigure}
	\captionsetup{font=normal}
	\caption{The repair transformation applied to the $5$-coloring of Figure~\ref{fig:breakup} with the breakup of Figure~\ref{fig:breakup2}.}
	\label{fig:repairmap}
\end{figure}

\subsection{Upper bounds on entropy loss}\label{sec:upper-bounds-on-entropy-loss}
We make use of the following extension of the subadditivity of entropy (see Section~\ref{sec:entropy} for basic definitions and properties), first used in a similar context by Kahn~\cite{kahn2001entropy}, followed by Galvin--Tetali~\cite{galvin2004weighted}.

\begin{lemma}[Shearer's inequality~\cite{chung1986some}]\label{lem:shearer}
Let $Z_1,\dots,Z_n$ be discrete random variables. Let $\cI$ be a collection of subsets of $\{1,\dots,n\}$ such that $|\{I \in \cI : i \in I\}| \ge k$ for every~$i$.
Then
\[ \Ent(Z_1,\dots,Z_n) \le \frac{1}{k} \sum_{I \in \cI} \Ent((Z_i)_{i\in I}) .\]
\end{lemma}

Recall that $X$ is fixed and that $\Omega$ is the set of proper colorings having $X$ as a breakup. Let $f$ be sampled from $\Pr_{\Lambda,P_0}$ conditioned on $f\in\Omega$. Let $F$ be the configuration coinciding with $f$ on $X_*$ and equaling a fixed symbol $\star$ on $X_*^c$. Applying Shearer's inequality to $(F_v)_{v \in \Even}$ with $\cI = \{ N(v) \}_{v \in \Odd}$ yields
\[ \Ent(F)
 = \Ent(F|_{\Even}) + \Ent(F|_{\Odd} \mid F|_{\Even})
 \le \sum_{v \in \Odd} \left[ \tfrac{\Ent(F|_{N(v)})}{2d}  + \Ent\big(F(v) \mid F|_{N(v)}\big) \right]. \]
Averaging this with the inequality obtained by reversing the roles of odd and even yields that
\begin{equation}\label{eq:entropy-bound-overview}
\Ent(f_{X_*}) = \Ent(F) \le \frac{1}{2}\sum_{v} \bigg[ \underbrace{\tfrac{\Ent\big(F(N(v))\big)}{2d}}_{\textup{I}} + \underbrace{\tfrac{\Ent\big(F|_{N(v)}~\mid~ F(N(v))\big)}{2d} + \Ent\big(F(v) \mid F(N(v))\big)}_{\textup{II}} \bigg].
\end{equation}
The advantage of this bound is that it is local, with each term involving only the values of $F$ on a vertex and its neighbors. The terms corresponding to vertices $v$ at distance $2$ or more from $X_*$ equal zero as $F$ is deterministic in their neighborhood.
The boundary terms corresponding to vertices $v$ in $\intextB X_*$ need to be handled with careful bookkeeping, which we do not elaborate on here.
Each of the remaining terms admits the simple bounds $\textup{I}\le \frac{q\log 2}{2d}$ and $\textup{II}\le\log (\lfloor \tfrac{q}{2} \rfloor  \lceil \tfrac{q}{2} \rceil)$, which only take into account the fact that $f$ is a proper coloring, i.e., that $F(v)\notin F(N(v))\subset[q]$. Equality in the second bound is achieved when $(F(v),F|_{N(v)})$ is uniformly distributed in $A \times B^{2d}$ for some dominant pattern $(A,B)$ (and in certain mixtures of such distributions). To obtain stronger bounds, we use additional information implied by the knowledge that $f \in \Omega$. This direction is developed in detail starting from Section~\ref{sec:Shearer_overview}. Let us here illustrate some ways in which one can proceed (though the actual proof differs in several ways from this illustration).

Recall that $X_*$ consists of $X_\overlap$, $X_\bad$ and the union of all $\intextB X_P$. Suppose, as a first example, that we are given the information that an even vertex $v$ is in both the $P$-pattern and the $Q$-pattern for some fixed distinct dominant patterns $P=(A,B)$ and $Q=(A',B')$ satisfying $|A|=|A'|=\lfloor \frac q2 \rfloor$. Thus $f(v)\in A \cap A'$ and $f(N(v))\subset [q] \setminus (A \cap A')$. Hence,
\[ \textup{II} \le \log (|A \cap A'| \cdot (q - |A \cap A'|)) \le \log ((\lfloor \tfrac q2 \rfloor-1)(\lceil \tfrac q2 \rceil+1)) \le \log (\lfloor \tfrac{q}{2} \rfloor  \lceil \tfrac{q}{2} \rceil - 1) .\]
In fact, similar techniques can be used to show that
\[ \textup{II} \le \log (\lfloor \tfrac{q}{2} \rfloor  \lceil \tfrac{q}{2} \rceil - 1) \qquad\text{for any }v \in X_\overlap .\]

As a second example, suppose that $v\in \extB X_P$. In particular, each neighbor $u\in X_P$ of $v$ is in the $P$-pattern but there necessarily exists a neighbor of $v$ which is not in the $P$ pattern. This information already suffices (as a calculation shows) to obtain that
\[ \textup{II} \le \log (\lfloor \tfrac{q}{2} \rfloor \lceil \tfrac{q}{2} \rceil) - \tfrac{|N(v) \cap X_P|}{2d} \log \left(\tfrac{\lfloor \frac{q}{2} \rfloor \lceil \frac{q}{2} \rceil}{\lfloor \frac{q}{2} \rfloor  \lceil \frac{q}{2} \rceil - 1}\right) \qquad\text{for any }v \in \extB X_P .\]
The gain in this bound thus depends on the size of the edge boundary $\partial X_P$ (rather than the size of $\intextB X_P$) which, for odd $q$, is in agreement with the bound in the toy scenario of Section~\ref{sec:toy scenario}.

As a third example, suppose that $v\in X_\bad$ is odd. When $q$ is even this necessarily implies that $|f(N(v))|>\frac{q}{2}$ (as in Figure~\ref{fig:breakup}) which leads to the bound $\textup{II}\le \log (\frac{q^2}{4} - 1)$. The case that $q$ is odd is more delicate and our proof introduces an additional idea to handle it (this is in fact done also in the even $q$ case in order to obtain a unified proof). We show that the set $\Omega$ may be divided into a relatively small number of subsets so that useful bounds are available for the vertices of $X_\bad$ when conditioning that $f$ belongs to any one of these subsets.

\section{Preliminaries}
\label{sec:preliminaries}

\subsection{Notation}\label{sec:notation}

Let $G=(V,E)$ be a graph.
For vertices $u,v \in V$, we denote the graph-distance between $u$ and $v$ by $\text{dist}(u,v)$.
For two non-empty sets $U,W \subset V$, we denote by $\dist(U,W)$ the minimum graph-distance between a vertex in $U$ and a vertex in $W$. We also write $\dist(u,W)$ as shorthand for $\dist(\{u\},W)$.
For vertices $u,v \in V$ such that $\{u,v\} \in E$, we say that $u$ and $v$ are {\em adjacent} and write $u \sim v$.
For a subset $U \subset V$, denote by $N(U)$ the {\em neighbors} of $U$, i.e., vertices in $V$ adjacent to some vertex in $U$, and define for $t>0$,
\[ N_t(U) := \{ v \in V : |N(v) \cap U| \ge t \} .\]
In particular, $N_1(U)=N(U)$.
Denote the {\em external boundary} and the \emph{internal boundary} of $U$ by
\[ \extB U := N(U) \setminus U \qquad\text{and}\qquad \intB U := \extB U^c ,\]
respectively.
Denote also
\[ \intextB U := \intB U \cup \extB U \qquad\text{and}\qquad U^+ := U \cup \extB U .\]
For a positive integer $r$, we denote
\[ U^{+r} := \{ v \in V : \dist(v,U) \le r \} .\]
In particular, $U^{+1}=U^+$.
The set of edges between two sets $U$ and $W$ is denoted by
\[ \partial(U,W):=\{ \{u,w\} \in E : u \in U,~ w \in W \} .\]
The {\em edge-boundary} of $U$ is denoted by $\partial U := \partial(U,U^c)$. We also define the set of out-directed boundary edges of $U$ to be
\[ \dpartial U := \{ (u,v) : u \in U,~v \in U^c,~u \sim v \} .\]
We write $\dpartialrev U := \dpartial (U^c)$ for the in-directed boundary edges of~$U$. We also use the shorthands $u^+ := \{u\}^+$, $\partial u := \partial \{u\}$ and $\dpartial u := \dpartial \{u\}$.
The \emph{diameter} of $U$, denoted by $\diam U$, is the maximum graph-distance between two vertices in $U$, where we follow the convention that the diameter of the empty set is $-\infty$.
For a positive integer $r$, we denote by $G^{\otimes r}$ the graph on $V$ in which two vertices are adjacent if their distance in $G$ is at most $r$.

We consider the graph $\Z^d$ with nearest-neighbor adjacency, i.e., the edge set $E(\Z^d)$ is the set of $\{u,v\}$ such that $u$ and $v$ differ by one in exactly one coordinate.
A vertex of $\Z^d$ is called {\em even (odd)} if it is at even (odd) graph-distance from the origin.
We denote the set of even and odd vertices of $\Z^d$ by $\Even$ and $\Odd$, respectively.
We say that a set $U \subset \Z^d$ \emph{disconnects a vertex $v \in \Z^d$ from infinity} if every infinite simple path starting from $v$ intersects $U$ (in particular, this occurs if $v \in U$).

As noted in the introduction, our main results hold also when $\Z^d$ is replaced by $\Z^{d_1}\times\T_{2m}^{d_2}$ with $\T_{2m}$ the cycle graph on $2m$ vertices (the path on $2$ vertices if $m=1$), $m\ge 1$, $d_1\ge 2$ and $d:=d_1+d_2$. The proofs require only minor adjustments. One such adjustment, relevant only for $m=1$, is to replace occurrences of the degree $2d$ of $\Z^d$ by the degree $2d - \1_{\{m=1\}}d_2$ of $\Z^{d_1}\times\T_{2m}^{d_2}$. Other adjustments (beyond obvious notational changes) are noted in the places where they are required.

For $t>0$ and an integer $n \ge 1$, we denote $\binom{n}{\le t}:=\sum_{k=0}^{\lfloor t \rfloor} \binom nk$ and note that $\binom{n}{\le t} \le (en/t)^t$.

\smallskip\noindent
{\bf Policy on constants:} In the rest of the paper, we employ the following policy on constants. We write $C,c,C',c'$ for positive absolute constants, whose values may change from line to line. Specifically, the values of $C,C'$ may increase and the values of $c,c'$ may decrease from line to line.

\subsection{Odd sets and regular odd sets}
\label{sec:odd-sets}

We say that a set $U \subset \Z^d$ is {\em odd} (even) if its internal boundary consists solely of odd (even) vertices, i.e., $U$ is odd if and only if $\intB U \subset \Odd$ and it is even if and only if $\intB U \subset \Even$. We say that an odd or even set $U$ is \emph{regular} if both it and its complement contain no isolated vertices. Observe that $U$ is odd if and only if $(\Even \cap U)^+ \subset U$ and that $U$ is regular odd if and only if $U=(\Even \cap U)^+$ and $U^c=(\Odd \cap U^c)^+$.

An important property of odd sets is that the size of their edge-boundary cannot be too small.
The following is by now rather well known (see, e.g., \cite[Corollary~1.4]{feldheim2016growth}.

\begin{lemma}\label{lem:boundary-size-of-odd-set}
	Let $A \subset \Z^d$ be finite and odd. If $A$ contains an even vertex then $|\partial A| \ge 2d(2d-1)$.
\end{lemma}
We note that the proof of \cite[Corollary~1.4]{feldheim2016growth} applies also in the setting of $\Z^{d_1}\times\T_{2m}^{d_2}$, provided the unit vector $s$ there is chosen in one of the $2d_1$ infinite directions (and with $2d - d_2$ replacing $2d$ in the statement of the lemma if $m=1$).

\subsection{Co-connected sets}
\label{sec:co-connected-sets}

In this section, we fix an arbitrary connected graph $G=(V,E)$.
A set $U \subset V$ is called {\em co-connected} if its complement $V \setminus U$ is connected.
For a set $U \subset V$ and a vertex $v \in V$, we define the {\em co-connected closure} of $U$ with respect to $v$ to be the complement of the connected component of $V \setminus U$ containing $v$, where it is understood that this results in $V$ when $v \in U$.
We say that a set $U' \subset V$ is a co-connected closure of a set $U \subset V$ if it is its co-connected closure with respect to some $v \in V$.
Evidently, every co-connected closure of a set $U$ is co-connected and contains $U$.
The following simple lemma summarizes some basic properties of the co-connected closure (see~\cite[Lemma~2.5]{feldheim2015long} for a proof).

\begin{lemma}\label{lem:co-connect-properties}
Let $A,B \subset V$ be disjoint and let $A'$ be a co-connected closure of $A$. Then
\begin{enumerate}[label=(\alph*)]
\item \label{it:co-connect-reduces-boundary} $\dpartial A' \subset \dpartial A$.
\item \label{it:co-connect-reduces-boundary2} $\dpartial (B \setminus A') \subset \dpartial B$.
\item \label{it:co-connected-minus-set-is-co-connected} If $B$ is co-connected then $B \setminus A'$ is also co-connected.
\item \label{it:co-connect-kills-components}
If $B$ is connected then either $B \subset A'$ or $B \cap A' = \emptyset$.
\end{enumerate}
\end{lemma}

The following lemma, taken from~\cite[Proposition~3.1]{feldheim2013rigidity} and based on ideas of Tim{\'a}r \cite{timar2013boundary}, establishes the connectivity of the boundary of subsets of $\Z^d$ which are both connected and co-connected.

\begin{lemma}\label{lem:int+ext-boundary-is-connected}
Let $A \subset \Z^d$ be connected and co-connected. Then $\intextB A$ is connected.
\end{lemma}
The following corollary is an extension of the lemma to the setting of $\Z^{d_1}\times\T_{2m}^{d_2}$.
\begin{cor}\label{cor:int+ext-boundary-is-connected torus}
Let $A \subset \Z^{d_1}\times\T_{2m}^{d_2}$, $d_1\ge 2$, be connected and co-connected. Then either $\intextB A$ is connected or each connected component of $\intextB A$ is infinite.
\end{cor}
\begin{proof}
  {\bf Case 1}: We first assume that $A$ is finite and prove that $\intextB A$ is connected. Let $d = d_1 + d_2$ and identify the vertex set of $\Z^{d_1}\times\T_{2m}^{d_2}$ as the subset of $\Z^d$ in which the last $d_2$ coordinates are restricted to take value in $\{0,1,\ldots, 2m-1\}$. For $v\in\Z^d$ let $P(v)$ be the vertex in $\Z^{d_1}\times\T_{2m}^{d_2}$ obtained from $v$ by performing modulo $2m$ in the last $d_2$ coordinates. Define a set $\bar{A}\subset\Z^d$ from $A$ by ``unwrapping'' the torus dimensions. Precisely, $v\in \bar{A}$ if and only if $P(v)\in A$.

  Let us check that $\bar{A}$ is co-connected: as $A$ is finite, there exists $v\in \Z^d$ such that any vertex agreeing with $v$ on the first $d_1$ coordinates lies outside $\bar{A}$. Let $w\in\Z^d\setminus \bar{A}$. As $w$ is arbitrary, co-connectedness of $\bar{A}$ is implied by the existence of a path in $\Z^d\setminus \bar{A}$ joining $w$ to $v$. To this end note that, as $A$ is co-connected, there is a path in $\Z^{d_1}\times\T_{2m}^{d_2}\setminus A$ joining $P(w)$ with $P(v)$. Thus there is a ``lift'' of this path to $\Z^d\setminus \bar{A}$ which joins $w$ with a vertex $\bar{v}$ having $P(\bar{v})=P(v)$. Lastly, this path may be continued in $\Z^d\setminus \bar{A}$ to connect $\bar{v}$ with $v$, by the definition of $v$.

  Let $\bar{A}_0$ be a connected component of $\bar{A}$. Then $\bar{A}_0$ is connected and co-connected in $\Z^d$ and thus Lemma~\ref{lem:int+ext-boundary-is-connected} implies that $\intextB \bar{A}_0$ is connected. This then implies that $\intextB A$ is connected in $\Z^{d_1}\times\T_{2m}^{d_2}$ as one may check that $P(\intextB \bar{A}_0) = \intextB A$.

  {\bf Case 2:} We now assume that $A$ is infinite. We may assume without loss of generality that $A^c$ is also infinite as otherwise we may replace $A$ by $A^c$ and deduce the result from the previous case. Fix $x\in A$ and $y\in A^c$. Let $(S_n)$ be an increasing sequence of finite subsets satisfying that $x,y\in S_n$ for all $n$ and $\cup_n S_n=\Z^{d_1}\times\T_{2m}^{d_2}$. For each $n$, let $B_n$ be the connected component of $x$ in $A\cap S_n$ and let $A_n$ be the co-connected closure of $B_n$ with respect to $y$.

  We claim first that each $A_n$ is finite. Indeed, $B_n$ is finite since $S_n$ is finite. In addition, the connected component of $y$ in $B_n^c$ contains $A^c$ and is thus infinite. Using the fact that $\Z^{d_1}\times\T_{2m}^{d_2}$ is one ended (since $d_1\ge 2$) we further deduce that this connected component contains the unique infinite connected component of $S_n^c$. This implies the finiteness of $A_n$.

  We next claim that the sequence $(A_n)$ increases to $A$. Indeed, if $z\in A$ then there is a finite path in $A$ from $x$ to $z$ and hence $z\in B_n\subset A_n$ for all large $n$. Similarly, if $z\in A^c$ then there is a finite path in $A^c$ from $y$ to $z$ and thus $z$ is in the connected component of $y$ in $B_n^c$ for all $n$, which implies that $z\in A_n^c$ for all $n$.

  We may thus apply the first case of the proof to $A_n$ and deduce that $\intextB A_n$ is connected. Observe also that $\intextB A_n$ converges to $\intextB A$ in the sense that for each $z$, $\1_{\intextB A_n}(z)\to \1_{\intextB A}(z)$ as $n\to\infty$. The last two facts imply that if $\intextB A$ has a finite connected component then this component must equal $\intextB A_n$ for all large $n$, whence it must be the unique connected component of $\intextB A$.
\end{proof}

\subsection{Graph properties}

In this section, we gather some elementary combinatorial facts about graphs.
Here, we fix an arbitrary graph $G=(V,E)$ of maximum degree $\Delta$.

\begin{lemma}\label{lem:sizes}
	Let $U \subset V$ be finite and let $t>0$. Then
	\[ |N_t(U)| \le \frac{\Delta}{t} \cdot |U| .\]
\end{lemma}
\begin{proof}
	This follows from a simple double counting argument.
	\[ t |N_t(U)|
		\le \sum_{v \in N_t(U)} |N(v) \cap U|
		= \sum_{u \in U} \sum_{v \in N_t(U)} \1_{N(u)}(v)
		= \sum_{u \in U} |N(u) \cap N_t(U)|
		\le \Delta |U| . \qedhere \]
\end{proof}

The next lemma follows from a classical result of Lov{\'a}sz~\cite[Corollary~2]{lovasz1975ratio} about fractional vertex covers,
applied to a weight function assigning a weight of $\frac1t$ to each vertex of $S$.

\begin{lemma}\label{lem:existence-of-covering2}
Let $S \subset V$ be finite and $t \ge 1$. Then there exists a set $T \subset S$ of size $|T|~\hspace{-4pt}\le~\hspace{-4pt}\frac{1+\log \Delta}{t} |S|$ such that $N_t(S) \subset N(T)$.
\end{lemma}

The following standard lemma gives a bound on the number of connected subsets of a graph.
\begin{lemma}[{\cite[Chapter~45]{Bol06}}]\label{lem:number-of-connected-graphs}
The number of connected subsets of $V$ of size $k+1$ which contain the origin is at most $(e(\Delta-1))^k$.
\end{lemma}

\subsection{Entropy}\label{sec:entropy}

In this section, we give a brief background on entropy (see, e.g., \cite{mceliece2002theory} for a more thorough discussion). Let $Z$ be a discrete random variable and denote its support by $\supp Z$.
The \emph{Shannon entropy} of $Z$ is
\[ \Ent(Z) := -\sum_z \Pr(Z=z) \log \Pr(Z=z) ,\]
where we use the convention that such sums are always over the support of the random variable in question.
Given another discrete random variable $Y$, the conditional entropy of $Z$ given $Y$ is
\[ \Ent(Z \mid Y) := \E\big[ \Ent(Z \mid Y=y) \big] = - \sum_y \Pr(Y=y) \sum_z \Pr(Z=z \mid Y=y) \log \Pr(Z=z \mid Y=y) .\]
This gives rise to the following chain rule:
\begin{equation}\label{eq:entropy-chain-rule}
\Ent(Y,Z) = \Ent(Y) + \Ent(Z \mid Y) ,
\end{equation}
where $\Ent(Y,Z)$ is shorthand for the entropy of $(Y,Z)$.
A simple application of Jensen's inequality gives the following two useful properties:
\begin{equation}\label{eq:entropy-support}
\Ent(Z) \le \log |\supp Z|
\end{equation}
and
\begin{equation}\label{eq:entropy-jensen}
\Ent(Z \mid Y) \le \Ent(Z \mid \phi(Y))\qquad\text{for any function }\phi .
\end{equation}
Equality holds in~\eqref{eq:entropy-support} if and only if $Z$ is a uniform random variable.
Together with the chain rule, \eqref{eq:entropy-jensen} implies that entropy is subadditive. That is, if $Z_1,\dots,Z_n$ are discrete random variables, then
\begin{equation}\label{eq:entropy-subadditivity}
\Ent(Z_1,\dots,Z_n) \le \Ent(Z_1) + \cdots + \Ent(Z_n) .
\end{equation}
As discussed in the overview, Shearer's inequality (Lemma~\ref{lem:shearer}) is an extension of this inequality.

\section{Main steps of proof}
\label{sec:high-level-proof}

In this section, we give the main steps of the proof of Theorem~\ref{thm:long-range-order}, providing definitions, stating lemmas and propositions, and concluding Theorem~\ref{thm:long-range-order} from them.
The proofs of the technical lemmas and propositions are given in subsequent sections. Theorem~\ref{thm:existence_Gibbs_states} and Theorem~\ref{thm:characterization_of_Gibbs_states} are proved in Section~\ref{sec:gibbs} and partly rely on the propositions given below.

\subsection{Notation}
\label{sec:overview-notation}
Throughout Section~\ref{sec:high-level-proof}, we fix a domain $\Lambda \subset \Z^d$ and a dominant pattern $P_0=(A_0,B_0)$ satisfying $|A_0|\le|B_0|$ as in~\eqref{eq:P_0_def}.
Recall from Theorem~\ref{thm:long-range-order} that
\begin{equation}\label{eq:finite_volume_measure}
\substack{\text{\normalsize $\Pr_{\Lambda,P_0}$ is the uniform measure on proper colorings}\\\text{\normalsize of $\Lambda$ satisfying that $\intB \Lambda$ is in the $P_0$-pattern}}.
\end{equation}
As mentioned in Section~\ref{sec:proof-overview-breakup}, in proving statements for this finite-volume measure, it will be technically convenient to work in an infinite-volume setting as follows. Sample $f$ from $\Pr_{\Lambda,P_0}$ and extend it to a proper coloring of $\Z^d$ by requiring that
\begin{equation}\label{eq:prob_outside_Lambda_def1}
\{ f(v) \}_{v \in \Lambda^c}\text{ are independent random variables, independent also from }f|_\Lambda,
\end{equation}
and
\begin{equation}\label{eq:prob_outside_Lambda_def2}
\begin{aligned}
f(v)\text{ is uniformly distributed in }A_0&\qquad\text{for all even $v \notin \Lambda$,}\\
f(v)\text{ is uniformly distributed in }B_0&\qquad\text{for all odd $v \notin \Lambda$.}
\end{aligned}
\end{equation}
With a slight abuse of notation, we continue to denote the distribution of the random coloring $f$ obtained as such by $\Pr_{\Lambda,P_0}$.

Denote the set of dominant patterns by $\phasedom$.
Let $\phase_0$ be the set of dominant patterns $P=(A,B)$ having $|A|\le|B|$ and set $\phase_1 := \phasedom \setminus \phase_0$. Note that $\phase_1$ is empty when $q$ is even and that $|\phase_0|=|\phase_1|$ when $q$ is odd.
The difference between dominant patterns in $\phase_0$ and $\phase_1$ plays an important role. For this reason, it will be convenient to use a notation distinguishing the two.
For $P=(A,B) \in \phasedom$, denote
\begin{equation}\label{eq:P_bdry_inner_def}
(P_\bdry,P_\inner) := \begin{cases}(A,B) &\text{if }P \in \phase_0\\(B,A) &\text{if }P \in \phase_1\end{cases} ,
\end{equation}
so that, for any $P \in \phasedom$,
\begin{equation}\label{eq:P_bdry_P_inner}
|P_\bdry| = \lfloor\tfrac{q}{2}\rfloor \qquad\text{and}\qquad |P_\inner| =\lceil\tfrac{q}{2}\rceil .
\end{equation}
Recall also the convention~\eqref{eq:P-even-odd}.
With this terminology, for any $P \in \phasedom$ and $v \in \Z^d$,
\begin{equation}\label{eq:P-even-odd-in-P-phase}
\begin{aligned}
	&\text{$v$ is in the $P$-pattern}~\iff~ f(v) \in P_\bdry &&\qquad\text{when $v$ is $P$-even,}\\
	&\text{$v$ is in the $P$-pattern}~\iff~ f(v) \in P_\inner &&\qquad\text{when $v$ is $P$-odd.}
\end{aligned}
\end{equation}
Note that $P_0$-even is even and $P_0$-odd is odd.
We denote by $\Even_P$ and $\Odd_P$ the set of $P$-even and $P$-odd vertices of $\Z^d$, respectively.

\subsection{Breakups -- definition and existence}
\label{sec:breakup-def}
We make use of the definitions of $Z_P(f)$ and $Z_*(f)$ from~\eqref{eq:Z-def} and~\eqref{eq:Z_*-def}.
As explained in Section~\ref{sec:proof-overview-Z}, $Z_P(f)$ indicates the regions that are ordered according to the $P$-pattern. As explained in Section~\ref{sec:proof-overview-breakup}, in order to bound the probability that a given vertex $v$ is not in the $P_0$-pattern, we introduce the notions of a breakup and a breakup seen from $v$.

The geometric structure of a breakup is captured by the following notion of an atlas.
An \emph{atlas} is a collection $X = (X_P)_{P \in \phasedom}$ of subsets of $\Z^d$ such that, for every $P$,
\begin{equation}\label{eq:def-atlas}
\text{$X_P$ is a regular $P$-even set}.
\end{equation}
For an atlas $X$, we define
\[ X_\overlap := \bigcup_{P \neq Q} (X_P \cap X_Q) , \qquad X_\bad := \bigcap_P (X_P)^c,\qquad X_* := \bigcup_P \intextB X_P \cup X_\overlap \cup X_\bad .\]
We say that an atlas $X$ is \emph{non-trivial} if $X_*$ is non-empty and that it is \emph{finite} if $X_*$ is finite.
For a set $V \subset \Z^d$, we also say that an atlas $X$ is \emph{seen from $V$} if every finite connected component of $X_*^{+5}$ disconnects some vertex $v \in V$ from infinity.

Let $f$ be a proper coloring of $\Z^d$.
An atlas $X$ is called a \emph{breakup of $f$} (with respect to the fixed domain $\Lambda$ and the fixed boundary pattern $P_0$) if it satisfies that
\begin{equation}\label{eq:breakup-0}
 \Lambda^c \subset X_{P_0}
\end{equation}
and that for every dominant pattern $P$ and every vertex $v$:
\begin{align}
&\text{If $v \in X_*^{+5}$ is $P$-odd then}&v \in X_P &~\iff~ N(v)\text{ is in the $P$-pattern} \label{eq:breakup-1}\\
&&&~\iff~ v \in Z_P(f). \nonumber
\end{align}

It is instructive to note that $(Z_P(f))_P$ is a breakup of $f$ whenever $\Lambda^c \cup \intB \Lambda$ is in the $P_0$-pattern.
The above property~\eqref{eq:breakup-1} is formulated via the values of $f$ on the neighbors of a vertex $v$. It is convenient to note its implication on the value of $f$ at $v$ itself.
Suppose that $X$ is a breakup of $f$ and let $P$ be a dominant pattern.
Then, by~\eqref{eq:P-even-odd-in-P-phase}, \eqref{eq:def-atlas} and~\eqref{eq:breakup-1},
\begin{align}
&f(v) \in P_\bdry &&\text{for any $P$-even }v \in X_*^{+5} \cap X_P, \label{eq:breakup-prop-even}\\
&f(v) \in P_\inner &&\text{for any $P$-odd }v \in X_*^{+5} \cap X_P \setminus X_\overlap. \label{eq:breakup-prop-odd}
\end{align}
Thus, $P$-even vertices in $X_*^{+5} \cap X_P$ are always in the $P$-pattern, while in regions of $X_*^{+5} \cap X_P$ which do not overlap with any other $X_{P'}$, all vertices are in the $P$-pattern. This property of $(X_P)_P$ is analogous to that of $(Z_P(f))_P$, except that here we do not have information on vertices of $X_P$ that are not near $X_*$.
Observe also that, by~\eqref{eq:breakup-1} and~\eqref{eq:breakup-prop-even},
\begin{align}
 &f(N(v)) \not\subset P_\bdry &&\qquad\text{for any $P$-odd }v \in X_\bad , \label{eq:breakup-prop-bad}\\
 f(u) \in P_\bdry\quad\text{and}\quad &f(N(v)) \not\subset P_\bdry &&\qquad\text{for any }(u,v) \in \dpartial X_P . \label{eq:breakup-prop-bdry-X_P}
 \end{align}
See Figure~\ref{fig:breakup} and Figure~\ref{fig:breakup2} for illustrations of breakups.

The following lemma, whose proof is given in Section~\ref{sec:breakup-construction}, shows that whenever there is a violation of the boundary pattern, there exists a breakup that ``captures'' that violation.

\begin{lemma}[existence of breakups seen from a vertex/set]\label{lem:existence-of-breakup}
	Let $f$ be a proper coloring of $\Z^d$ such that $\Int(\Lambda)^c$ is in the $P_0$-pattern and let $V \subset \Lambda$. Then there exists a breakup $X$ of $f$ satisfying that $X_*^{+5}$ is the union of those connected components of $Z_*(f)^{+5}$ that are either infinite or disconnect some vertex in $V$ from infinity. In particular,
	\begin{itemize}
	 \item $X$ is seen from $V$.
	 \item $X$ is non-trivial if $V^{+5}$ either intersects $Z_*(f)$ or is not in the $P_0$-pattern.
	 \item $V^{+5} \cap X_{P_0} \setminus X_\overlap$ is in the $P_0$-pattern.
	\end{itemize}
\end{lemma}

\subsection{Unlikeliness of breakups}\label{sec:unlikeliness-of-breakups}
Now that we have a definition of breakup and we know that any violation of the boundary pattern creates a non-trivial breakup, it remains to show that breakups are unlikely.

The main part of the proof consists of obtaining a quantitative bound on the probability of a large breakup.
Nevertheless, formally one also needs to rule out the existence of an infinite breakup. As this does not require a quantitative bound, it is actually rather simple to do so. The following lemma is proved in Section~\ref{sec:no-infinite-breakups}.

\begin{lemma}\label{lem:no-infinite-breakups}
	$\Pr_{\Lambda,P_0}$-almost surely, every breakup seen from a finite set is finite.
\end{lemma}

We now discuss the quantitative bound on finite breakups.
To this end, denote by $\breakups$ the collection of atlases which have a positive probability of being a breakup and, for integers $L,M,N \ge 0$, denote
\[ \breakups_{L,M,N} := \left\{ X \in \breakups ~:~ \Big|\bigcup_P \partial X_P\Big|=L,~|X_\overlap|=M,~|X_\bad|=N \right\} .\]

\begin{prop}\label{prop:prob-of-breakup-associated-to-V}
	For any finite $V \subset \Z^d$ and any integers $L,M,N \ge 0$, we have
	\[ \Pr_{\Lambda,P_0}(\text{there exists a breakup in }\breakups_{L,M,N}\text{ seen from }V) \le 2^{|V|} \cdot \exp\left(- \tfrac{c}{q^3(q+\log d)} \big( \tfrac{L}{d}+\tfrac{M}{q}+\tfrac{N}{q^2} \big) \right) .\]
\end{prop}
This is the main technical proposition of this paper. An overview of the tools to prove the proposition is given in the rest of Section~\ref{sec:high-level-proof}, with the detailed proofs appearing in Section~\ref{sec:breakup}, Section~\ref{sec:shift-trans} and Section~\ref{sec:approx}.

It is now a simple matter to deduce Theorem~\ref{thm:long-range-order}.

\begin{proof}[Proof of Theorem~\ref{thm:long-range-order}]
	Suppose that $v$ is not in the $P_0$-pattern.
	Lemma~\ref{lem:existence-of-breakup} implies the existence of a non-trivial breakup $X$ seen from $v$.
	By Lemma~\ref{lem:no-infinite-breakups}, we may assume that $X$ is finite so that $X \in \breakups_{L,M,N}$ for some $L,M,N \ge 0$. Since $X$ is also non-trivial, some set in $\{ X_P,X_P^c \}_P$ is both non-empty and not $\Z^d$. Recalling~\eqref{eq:def-atlas} and applying Lemma~\ref{lem:boundary-size-of-odd-set} (or its analogue for even sets) to any such set shows that $L\ge d^2$.
	Therefore, by Proposition~\ref{prop:prob-of-breakup-associated-to-V},
	\[ \Pr_{\Lambda,P_0}\big(v\text{ is not in the $P_0$-pattern}\big) \le 2 \sum_{\substack{L \ge d^2,\,M,N \ge 0}} \exp\left(- \tfrac{c}{q^3(q+\log d)} \big( \tfrac{L}{d}+\tfrac{M}{q}+\tfrac{N}{q^2} \big) \right) .\]
	Using~\eqref{eq:dim-assump}, the desired inequality follows (perhaps with a larger constant $C$ in~\eqref{eq:dim-assump}).
\end{proof}

\subsection{Unlikeliness of specific breakups}

In light of the bound in Proposition~\ref{prop:prob-of-breakup-associated-to-V}, it is natural to first prove that a specific atlas is unlikely to be a breakup. Precisely, we would like to show the following.

\begin{prop}\label{prop:prob-of-given-breakup}
	For any $X \in \breakups_{L,M,N}$, we have
	\[ \Pr_{\Lambda,P_0}(X\text{ is a breakup}) \le \exp\left(- \tfrac{c}{q} \big( \tfrac{L}{d}+\tfrac{M}{q}+\tfrac{N}{q^2} \big)\right) .\]
\end{prop}

\subsection{Approximations}\label{sec:approx-overview}

It is temping to conclude that breakups seen from $V$ are unlikely (as stated in Proposition~\ref{prop:prob-of-breakup-associated-to-V}) by summing the bound of Proposition~\ref{prop:prob-of-given-breakup} over all atlases in $\breakups_{L,M,N}$ that are seen from $V$. Unfortunately, this approach fails as the size of the latter collection exceeds the reciprocal of the bound of Proposition~\ref{prop:prob-of-given-breakup}.
To overcome this obstacle we employ a delicate coarse-graining scheme of the possible breakups according to their rough features. For this we crucially rely on the geometric restriction~\eqref{eq:def-atlas}.

Let $A = ((A_P)_{P \in \phasedom}, A^*,A^{**})$ be a collection of subsets of $\Z^d$ such that each $A_P$ is $P$-even and $A^* \subset A^{**}$. For notational convenience, we write $Q \simeq P$ if $Q,P \in \phase_i$ for some $i \in \{0,1\}$.
We say that $A$ is an \emph{approximation of an atlas $X$} if the following conditions hold for all $P \in \phasedom$:
\begin{enumerate}[label=\text{(A\arabic*)},ref=\text{(A\arabic*)}]
  \item \label{it:approx-A_P} $A_P \subset X_P \subset A_P \cup (\Odd_P \cap A^*) \cup (\Even_P \cap A^{**})$.
 \item \label{it:approx-A*_P} $\Odd_P \cap A^* \subset N_d(\bigcup_{Q \simeq P} A_Q)$.
 \item \label{it:approx-unknown-size} $|A^{**}| \le \tfrac{C\log d}{\sqrt{d}}\cdot \big|\bigcup_Q \partial X_Q\big|$.
 \item \label{it:approx-unknown-location} $A^{**} \subset \bigcup_Q (\intextB X_Q)^{+3}$.
\end{enumerate}

Since $A^*\subset A^{**}$, property \ref{it:approx-A_P} implies that $A_P\subset X_P\subset A_P\cup A^{**}$ for all $P$ (See Figure~\ref{fig:approx} for an illustration of these containments). In words, the sets $A_P$ indicate vertices which are guaranteed to be in $X_P$ while the set $A^{**}$ indicates vertices whose classification into the various $X_P$ is not fully specified by the approximation. The distinguished subset $A^*\subset A^{**}$ conveys additional information through \ref{it:approx-A_P} and \ref{it:approx-A*_P}: A $P$-odd vertex is either guaranteed to belong to $X_P$ (if it belongs to $A_P$), is guaranteed not to belong to $X_P$ (if it does not belong to $A_P \cup A^*$), or at least half of its neighbors belong to $\bigcup_{Q \simeq P} A_Q$. The other two properties further restrict the ``missing information'', with~\ref{it:approx-unknown-size} ensuring that $A^{**}$ is not too large and~\ref{it:approx-unknown-location} ensuring that $A^{**}$ is only present near the boundaries of the~$X_P$'s.

The following proposition shows that one may find a small family which contains an approximation of every atlas seen from a given set.

\begin{prop}\label{prop:family-of-odd-approx}
	For any integers $L,M,N \ge 0$ and any finite set $V \subset \Z^d$, there exists a family $\cA$ of approximations of size
	\[ |\cA| \le 2^{|V|} \cdot \exp\left(CL \tfrac{(q+\log d)\log d}{d^{3/2}} + C(M+N) \tfrac{ \log^2 d}{d} \right) \]
	such that any $X \in \breakups_{L,M,N}$ seen from $V$ is approximated by some element in $\cA$.
\end{prop}

Of course, working with approximations, finding a suitable modification of Proposition~\ref{prop:prob-of-given-breakup} becomes a more complicated task.
The following proposition provides a similar bound on the probability of having a breakup which is approximated by a given approximation (its proof actually uses Proposition~\ref{prop:prob-of-given-breakup} as an ingredient).

\begin{prop}\label{prop:prob-of-odd-approx}
	For any approximation $A$ and any integers $L,M,N \ge 0$, we have
	\[ \Pr_{\Lambda,P_0}(A\text{ approximates some breakup in }\breakups_{L,M,N}) \le \exp\left(- \tfrac{c}{q^3(q+\log d)} \big( \tfrac{L}{d}+\tfrac{M}{q}+\tfrac{N}{q^2} \big) \right) .\]
\end{prop}

We are now ready to complete the proof of Proposition~\ref{prop:prob-of-breakup-associated-to-V}.

\begin{proof}[Proof of Proposition~\ref{prop:prob-of-breakup-associated-to-V}]
Let $\cA$ be a family of approximations as guaranteed by Proposition~\ref{prop:family-of-odd-approx}. Let $\Omega$ be the event that there exists a breakup in $\breakups_{L,M,N}$ seen from $V$ and let $\Omega(A)$ be the event that there exists a breakup in $\breakups_{L,M,N}$ seen from $V$ and approximated by $A$. Then, by Proposition~\ref{prop:family-of-odd-approx} and Proposition~\ref{prop:prob-of-odd-approx},
\[ \Pr(\Omega) \le \sum_{A \in \cA} \Pr(\Omega(A)) \le 2^{|V|} \cdot \exp\left(CL \tfrac{(q+\log d)\log d}{d^{3/2}} + C(M+N) \tfrac{ \log^2 d}{d} - \tfrac{c}{q^3(q+\log d)} \big( \tfrac{L}{d}+\tfrac{M}{q}+\tfrac{N}{q^2} \big) \right) .\]
The proposition now follows using that $q \le c d^{1/10} / \log^{1/5} d$ by~\eqref{eq:dim-assump}.
\end{proof}

The proofs of Proposition~\ref{prop:prob-of-given-breakup}, Proposition~\ref{prop:family-of-odd-approx} and Proposition~\ref{prop:prob-of-odd-approx} constitute the main technical parts of the paper.
Proposition~\ref{prop:family-of-odd-approx}, showing the existence of approximations, is proved in Section~\ref{sec:approx}.
The proofs of Proposition~\ref{prop:prob-of-given-breakup} and Proposition~\ref{prop:prob-of-odd-approx} make use of a repair transformation and entropy  methods, along the lines discussed in Section~\ref{sec:proof-overview}.
Key points of the analysis are introduced in the next section, while the detailed proofs are given in Section~\ref{sec:prob-of-given-breakup}, Section~\ref{sec:prob-of-approx} and Section~\ref{sec:shift-trans}.

\subsection{Bounding the probability of breakups and approximations}\label{sec:Shearer_overview}

In proving Proposition~\ref{prop:prob-of-given-breakup}, we roughly follow the plan discussed in Section~\ref{sec:repair transformation} and Section~\ref{sec:upper-bounds-on-entropy-loss}. The same approach is also used for the proof of Proposition~\ref{prop:prob-of-odd-approx}, but is more involved as less information is provided (only an approximation of $X$ is given). In following this approach, we are led to estimate entropic terms similar to the terms \textup{I} and \textup{II} appearing in~\eqref{eq:entropy-bound-overview}. The type of additional information we shall use in order to improve the naive bounds on such entropic terms is based on four notions --- \emph{non-dominant vertices}, vertices having \emph{unbalanced neighborhoods}, \emph{restricted edges} and vertices having a \emph{unique pattern} --- all of which we now define. These notions are somewhat abstract (and not directly related to a specific breakup) in order to allow sufficient flexibility for the proof of both propositions.

Let $f \colon \Z^d \to [q]$ be a proper coloring and let $\Omega$ be a collection of proper colorings of $\Z^d$.
The four notions implicitly depend on $f$ and~$\Omega$. Let $v \in \Z^d$ be a vertex and let $u$ be adjacent to $v$. Recall that $(v,u) \in \dpartial v$ is the directed edge from $v$ to $u$. We say that

\smallskip
\begin{itemize}[leftmargin=15pt]
	\item
	$v$ is \emph{non-dominant} (in $f$) if
	\begin{equation}
	|f(N(v))| \notin \big\{\lfloor \tfrac{q}{2} \rfloor, \lceil \tfrac{q}{2} \rceil \big\} \label{eq:restricted_not_dom}.
	\end{equation}
Thus, a vertex is non-dominant if the set of colors which appear on its neighbors does not determine a dominant pattern. See Figure~\ref{fig:breakup} for an illustration of this notion.

\smallskip
	\item
	$(v,u)$ is \emph{restricted} (in $(f,\Omega)$) if
\begin{equation}\label{eq:restricted_def}
\big\{g(u) : g \in \Omega,~ g(N(v)) = f(N(v)) \big\} \cup \big\{g(v) : g \in \Omega,~ g(N(v)) = f(N(v)) \big\} \neq [q] .
\end{equation}
Observe that $(v,u)$ is restricted if and only if
	\begin{align}
	\text{either}\qquad&\big\{g(u) : g \in \Omega,~ g(N(v)) = f(N(v)) \big\} ~\neq~ f(N(v)),\label{eq:restricted_A_def}\\
	\text{or}\qquad&\big\{g(v) : g \in \Omega,~ g(N(v)) = f(N(v)) \big\} ~\neq~ f(N(v))^c .\label{eq:restricted_B_def}
	\end{align}
Thus, roughly speaking, $(v,u)$ is restricted if upon inspection of the set of values which appear on the neighbors of $v$, one is guaranteed that either $u$ or $v$ cannot take all possible values which they should typically take, i.e., either $u$ cannot take some value in $f(N(v))$, or $v$ cannot take some value in $f(N(v))^c$.
Note that~\eqref{eq:restricted_B_def} actually implies that all edges in $\dpartial v$ are restricted as it does not involve $u$.

\smallskip
	\item
	$v$ has an \emph{unbalanced neighborhood} (in $f$) if
	\[ |\{ u\in N(v) : f(u) = i \}| \le \tfrac{d}{q} \qquad\text{for some }i \in f(N(v)) .\]
This condition states that at least one of the colors which appear on the neighbors of $v$ is significantly underrepresented in the sense that it appears substantially less than $\frac{2d}{|f(N(v))|}$ times.

\smallskip
	\item
    $v$ has a \emph{unique pattern} (in $\Omega$) if there exists $A \subset [q]$ such that the following holds for every $g \in \Omega$: if $g(N(v)) \neq A$ then either $v$ is non-dominant in $g$ or all edges in $\dpartial v$ are restricted in~$(g,\Omega)$.



\noindent Thus, $A$ is the unique choice for $f(N(v))$ which does not lead to a significant loss of entropy.

\end{itemize}

\smallskip

With Proposition~\ref{prop:prob-of-given-breakup} in mind, suppose that $\Omega$ is the set of all proper colorings having a given atlas $X$ as a breakup. To illustrate the above notions one may check, for instance, that (i) every edge in $\dpartialrev X_P$ is necessarily restricted, (ii) every edge incident to $X_\overlap$ is either restricted (in one of its two orientations) or else it must be incident to a non-dominant vertex, (iii) for even~$q$, every odd vertex in $X_\overlap\cup X_\bad$ must be non-dominant and (iv) every vertex in every $X_P$ has a unique pattern (including those in $X_\overlap$). Thus the above notions feature significantly on $\intextB X_P$, $X_\overlap$ and also on $X_\bad$ when $q$ is even, and are helpful in controlling the probability that $X$ is a breakup. In contrast, when $q$ is odd, difficulties arise in controlling the prevalence of the above notions in $X_\bad$ (this is due to the possibility of having vertices $v$ in $X_\bad$ satisfying that $|f(N(v))|=\lceil\frac{q}{2}\rceil$). To overcome this, we shall, in the course of proving Proposition~\ref{prop:prob-of-given-breakup}, divide $\Omega$ into a relatively small number of subevents on which we are ensured that the above notions also feature sufficiently in $X_\bad$ (the notion of unbalanced neighborhood has an important role in this part).

The following lemma, which is proved in Section~\ref{sec:shift-trans}, provides a general upper bound on the probability of certain events in terms of the above notions. Given a proper coloring $f$ of $\Z^d$, a collection $\Omega$ of proper colorings of $\Z^d$ and a subset $S \subset \Z^d$, let $S^f_\nondom$ be the set of vertices in $S$ which are non-dominant in $f$, let $S^f_\unbal$ be the set of vertices in $S$ which have unbalanced neighborhoods in $f$, let $S^{\Omega,f}_\rest$ be the set of directed edges $(v,u)$ with $v \in S$ which are restricted in $(f,\Omega)$, and let $S^\Omega_\unique$ be the set of vertices in $S$ which have a unique pattern in $\Omega$.

\begin{lemma}\label{lem:bound-on-pseudo-breakup}
	Let $S \subset \Z^d$ be finite and let $\{ S_P \}_{P \in \phasedom}$ be a partition of $S^c$ such that $\extB S_P \subset S$ for all~$P$.
	Suppose that $S \cup S_{P_0}$ contains $(\Lambda^c)^+$.
	Let $\Omega$ be an event on which $(\intB S_P)^+$ is in the $P$-pattern for every $P$ and denote
	\[ k(\Omega) := \min_{f \in \Omega} \left( \tfrac{1}{d}\big|S^{\Omega,f}_\rest\big| + \tfrac{1}{q} \big|S^f_\nondom\big| + \big|S^f_\unbal\big| \right) .\]
	Then
	\[ \Pr_{\Lambda,P_0}(\Omega) \le \exp\left[-\tfrac{k(\Omega)}{128q} + \tfrac{q}{d}\big|S \setminus S^\Omega_\unique\big| +e^{-d/65q^2}|S| \right] .\]
\end{lemma}

We conclude with a short outline as to how Lemma~\ref{lem:bound-on-pseudo-breakup} is used to prove Proposition~\ref{prop:prob-of-given-breakup}.
To this end, we take $S$ to be $X_*$ and $S_P$ to be $X_P \setminus X_*$, and, as a first attempt, we take $\Omega$ to be the event that $X$ is a breakup.
Concluding Proposition~\ref{prop:prob-of-given-breakup} from Lemma~\ref{lem:bound-on-pseudo-breakup} is still not straightforward, as the latter, when applied directly to $\Omega$, gives an insufficient bound on its probability.
The difficulty here is that, while $k(\Omega)$ is large in comparison to $L$ and $M$, it is not necessarily large in comparison to~$N$.
Indeed, the observations above will allow us to deduce (see Lemma~\ref{lem:lower-bound-on-size-of-restricted} and Lemma~\ref{lem:vertices-with-unique-pattern}) that
\[ \tfrac{1}{d}\big|S^{\Omega,f}_\rest\big| + \tfrac{1}{q}\big|S^f_\nondom\big| \ge \tfrac{L}{2d} + \tfrac{M}{4q} \qquad\text{and}\qquad |S \setminus S^\Omega_\unique| \le N ,\]
for every $f \in \Omega$.
Lemma~\ref{lem:bound-on-pseudo-breakup} thus gives that $\Pr_{\Lambda,P_0}(\Omega)\le \exp(-\frac {c}{dq} L - \frac {c}{q^2} M + \frac {2q}{d} N)$ (using also~\eqref{eq:dim-assump}), which is not small as a function of $N$.
Instead, to obtain a good bound, we shall apply Lemma~\ref{lem:bound-on-pseudo-breakup} to subevents $\Omega' \subset \Omega$ on which we have additional information about the coloring on the set $X_\bad$.
For suitably chosen subevents (see Lemma~\ref{lem:lower-bound-on-size-of-restricted-with-V}), the number of restricted edges in $X_\bad$ increases enough to ensure that
\[ k(\Omega') \ge \tfrac{L}{3d} + \tfrac{M}{6q} + \tfrac{N}{18q^2} .\]
As the entropy of this additional information is negligible with our assumptions (see Lemma~\ref{lem:family-of-strong-odd-approx}), this will allow us to conclude Proposition~\ref{prop:prob-of-given-breakup} by taking a union bound over the subevents~$\Omega'$.
This is carried out in detail in Section~\ref{sec:prob-of-given-breakup}.
The proof of Proposition~\ref{prop:prob-of-odd-approx} is given in Section~\ref{sec:prob-of-approx}.

\section{Breakups}
\label{sec:breakup}

In this section, we prove Lemma~\ref{lem:existence-of-breakup} about the existence of a non-trivial breakup, we prove Lemma~\ref{lem:no-infinite-breakups} about the absence of infinite breakups, we prove Proposition~\ref{prop:prob-of-given-breakup} about the probability of a given breakup, and we prove Proposition~\ref{prop:prob-of-odd-approx} about the probability of an approximation.

\subsection{Constructing a breakup seen from a vertex/set}\label{sec:breakup-construction}

	Here we prove Lemma~\ref{lem:existence-of-breakup}.
	As we have mentioned, the collection $Z=(Z_P(f))_P$ defined in~\eqref{eq:Z-def} is always a breakup as long as $\Int(\Lambda)^c$ is in the $P_0$-pattern. The main difficulty is therefore to construct a breakup that is seen from a given set. For this, we require the following lemma which allows to ``close holes''. The proof is accompanied by Figure~\ref{fig:closing-holes}.

\begin{lemma}\label{lem:closing-holes}
	Let $V,W \subset \Z^d$ and let $B$ be the union of connected components of $W$ that are either infinite or disconnect some vertex in $V$ from infinity. Let $A$ be a connected component of $B^c$. Then $\extB A$ is contained in a connected component of $(W^c)^+$.
\end{lemma}

\begin{proof}
	Let $a,a' \in \extB A$. It suffices to show that $a$ and $a'$ are connected by a path in $(W^c)^+$.
	Assume towards a contradiction that this is not the case.
	
	Let $S$ be the connected component of $a$ in $(W^c)^+$ and note that $\intextB S \subset W$. Let $\bar{S}$ be the co-connected closure of $S$ with respect to $a'$. Since $a' \notin S$ by assumption, we have $a' \notin \bar{S}$. Since Lemma~\ref{lem:int+ext-boundary-is-connected} implies that $\intextB \bar{S}$ is connected and since $\intextB \bar{S} \subset \intextB S \subset W$, we see that $\intextB \bar{S}$ is contained in a connected component $D$ of $W$.
	
	Since $\extB A \subset \intB B \subset \intB W$, the connected components $D_a$ and $D_{a'}$ of $a$ and $a'$ in $W$ are contained in $B$.
	Since any path between $a$ and $a'$ must intersect $\intextB \bar{S} \setminus \{a,a'\}$ and since there is a path in $B^c \cup \{a,a'\}$ between $a$ and $a'$, it follows that $\intextB \bar{S} \not\subset B$. In particular, $D \not\subset B$ so that $D \neq D_a,D_{a'}$. Hence, $D$ is disjoint from both $D_a$ and $D_{a'}$.
	
	We now show that $D \subset B$, which leads to a contradiction, and thus concludes the proof.
	If $D$ is infinite then this follows from the definition of $B$.
	Otherwise, $\intextB \bar{S} \subset D$ is finite, so that either $\bar{S}$ or $\bar{S}^c$ is finite (since $\Z^d$ with $d\ge 2$ is one ended). Thus, $\intextB \bar{S}$ disconnects either $a$ or $a'$ from infinity. Therefore, $D$ disconnects either $D_a$ or $D_{a'}$ from infinity. In particular, $D$ disconnects some vertex in $V$ from infinity, so that $D \subset B$ by the definition of $B$.
\end{proof}

The proof of Lemma~\ref{lem:closing-holes} requires a slight modification to apply in the setting of $\Z^{d_1}\times\T_{2m}^{d_2}$, $d_1\ge 2$. Lemma~\ref{lem:int+ext-boundary-is-connected} needs to be replaced by Corollary~\ref{cor:int+ext-boundary-is-connected torus} and thus the case that all connected components of $\intextB \bar{S}$ are infinite needs to be addressed. In fact, this case cannot occur. The arguments in the proof still imply that every connected component of $\intextB \bar{S}$ is contained in a connected component of $W$, and also that $\intextB \bar{S} \not\subset B$. The definition of $B$ thus implies that $\intextB \bar{S}$ has a finite connected component (whence $\intextB \bar{S}$ is connected, by Corollary~\ref{cor:int+ext-boundary-is-connected torus}).

\begin{figure}
	\includegraphics[scale=0.55]{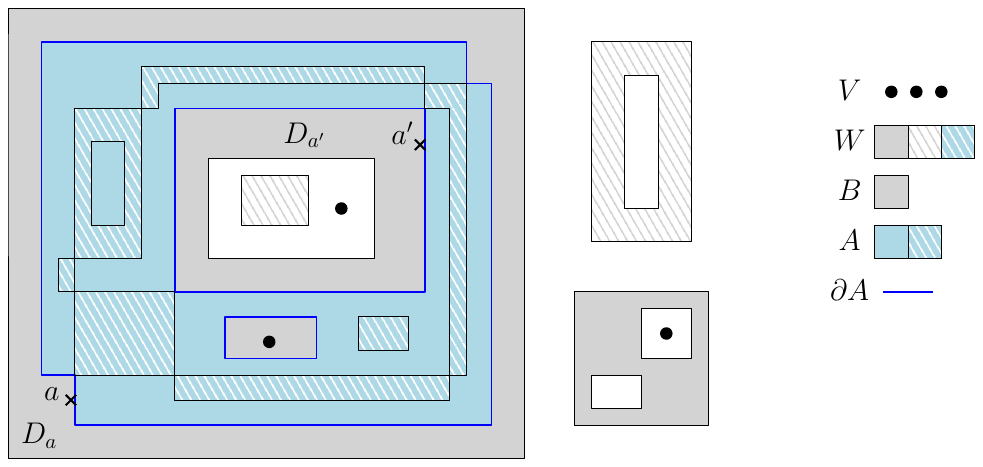}
	\caption{An illustration of the objects in Lemma~\ref{lem:closing-holes} and its proof.}
	\label{fig:closing-holes}
\end{figure}

The next lemma shows that an atlas can be ``localized'' into an atlas which is seen from $V$.

\begin{lemma}\label{lem:existence-of-adapted-atlas}
	Let $\Lambda$ be a domain, let $V \subset \Lambda$, let $P_0$ be a dominant pattern and let $Y$ be an atlas such that $\Lambda^c \subset Y_{P_0}$. Then there exists an atlas $X$ which is seen from $V$ and satisfies that
\begin{equation}\label{eq:adapted}
X_*^{+5} \cap X_P = X_*^{+5} \cap Y_P \qquad\text{for every dominant pattern }P.
\end{equation}
Moreover, $\Lambda^c \subset X_{P_0}$ and $X_*^{+5}$ is the union of connected components of $Y_*^{+5}$ that are either infinite or disconnect some vertex in $V$ from infinity.
\end{lemma}
\begin{proof}
	Let $B$ be the union of connected components of $Y_*^{+5}$ that are infinite or disconnect some vertex in $V$ from infinity. Let $\cA$ be the set of connected components of $B^c$. We claim that
	\[ \text{for every $A \in \cA$, there exists a unique $P_A \in \phasedom$ such that $A^{+5} \setminus A \subset Y_{P_A} \setminus Y_*$} .\]
	Indeed, it follows from the definition of $Y_*$ that for every $a \in A^{+5} \setminus A \subset Y_*^c$, there exists a unique dominant pattern $P_a$ such that $a \in Y_{P_a}$. Since Lemma~\ref{lem:closing-holes} applied with $W:=Y_*^{+5}$ yields that $\extB A$ is contained in a connected component of $(W^c)^+ \subset (Y_*^{+4})^c$, we see that $P_a=P_{a'}$ for all $a,a' \in \extB A$.
	The claim follows. Note also that, since $\Lambda^c \subset Y_{P_0}$, we have $P_A=P_0$ for all $A \in \cA$ such that $A \not\subset \Lambda$.
	
	We now define $X=(X_P)_P$ by
	\[ X_P := (Y_P \cap B) \cup \bigcup \{A \in \cA : P_A=P \} ,\qquad P \in \phasedom .\]
	Let us show that $X$ satisfies the conclusion of the lemma.
	Note first that $X_P \cap B = Y_P \cap B$ and $X_* \subset B$, so that $X_* = Y_* \cap B$ and $X_*^{+5} = B$. It easily follows that $X$ is an atlas satisfying~\eqref{eq:adapted}.
	Let us check that $X$ is seen from $V$. Indeed, every finite connected component of $X_*^{+5}=B$ is by definition a connected component of $Y_*^{+5}$ that disconnects some vertex in $V$ from infinity.
	Finally, $\Lambda^c \subset X_{P_0}$, since $\Lambda^c \subset Y_{P_0}$ and $P_A=P_0$ for all $A \in \cA$ such that $A \not\subset \Lambda$.
\end{proof}

\begin{proof}[Proof of Lemma~\ref{lem:existence-of-breakup}]
	Recall the definition of $Z_P(f)$ from~\eqref{eq:Z-def}.
	It is straightforward to check that $Y=Z=(Z_P(f))_P$ is an atlas and, using the assumption that $\Int(\Lambda)^c$ is in the $P_0$-pattern, that $\Lambda^c \subset Z_{P_0}$. Thus, the first part of the lemma follows from Lemma~\ref{lem:existence-of-adapted-atlas} (since \eqref{eq:adapted} implies~\eqref{eq:breakup-1}). The three items stated in the second part now follow from the definitions.
\end{proof}

\subsection{No infinite breakups}\label{sec:no-infinite-breakups}
Here we prove Lemma~\ref{lem:no-infinite-breakups}.
As mentioned above, our main argument (namely, Proposition~\ref{prop:prob-of-given-breakup} and Proposition~\ref{prop:prob-of-odd-approx}) is concerned only with finite breakups. However, it is easy to rule out the existence of an infinite breakup in a random coloring. In doing so, there are two possibilities to have in mind: either there exists an infinite component of $Z_*^{+5}$ or infinitely many finite components surrounding a vertex.

\begin{proof}[Proof of Lemma~\ref{lem:no-infinite-breakups}]
	By~\eqref{eq:prob_outside_Lambda_def1} and~\eqref{eq:prob_outside_Lambda_def2}, for any $u \notin \Lambda^+$ and $P \neq P_0$ for which $u$ is $P$-even,
	\[ \Pr\big(u\text{ is in the $P$-pattern} \mid (f(v))_{v \neq u}\big) \le \tfrac{\lceil \tfrac{q}{2} \rceil-1}{\lceil \tfrac{q}{2} \rceil} \le \tfrac{q-1}{q+1} .\]
	Say that $u$ is in a \emph{double pattern} if $u$ is $P$-odd and $N(u)$ is in the $P$-pattern for some $P \neq P_0$. Then
	\[ \Pr\big(u\text{ is in a double pattern} \mid (f(v))_{v \notin N(u)}\big) \le 2^q \left(\tfrac{q-1}{q+1}\right)^{2d} .\]
	Note that if a vertex $u \in \Z^d \setminus \Lambda^+$ belongs to $Z_*$, then some vertex in $u^+$ is in a double pattern.
	
	We wish to show that, almost surely, every breakup seen from $V$ is finite.
	For $v \in \Z^d$, let $E_v$ be the event that $v$ is in an infinite connected component of $Z_*^{+5}$. Let $E'_v$ be the event that $v$ is disconnected from infinity by infinitely many connected components of $Z_*^{+5}$. It suffices to show that $\Pr(E_v)=\Pr(E'_v)=0$ for any $v \in \Z^d$. Let us show that $\Pr(E'_v)=0$; the proof that $\Pr(E_v)=0$ is very similar.
	On the event $E'_v$, for any $m$, there exists a set $B \subset \Z^d \setminus \Lambda^+$ of size at least $m$ such that $B^{+5}$ is connected and disconnects $v$ from infinity and such that for every vertex $u \in B$ there exists a vertex in $u^+$ which is in a double pattern. In particular, for any $m$, there exists a path $\gamma$ in $(\Z^d \setminus \Lambda^+)^{\otimes 50}$ of length $n \ge m$ such that $\{ \gamma_i^+ \}_{i=0}^n$ are pairwise disjoint, $\dist(v,\gamma_0) \le Cn$ and all vertices $\{ \gamma_i \}_{i=0}^n$ are in a double pattern. Since $\Pr(\gamma) \le 2^{qn} (\tfrac{q-1}{q+1})^{2dn}$ for any such fixed $\gamma$, and since the number of simple paths $\gamma$ in $(\Z^d)^{\otimes 50}$ of length $n$ with $\dist(v,\gamma_0) \le Cn$ is at most $d^{Cn}$, the lemma follows using~\eqref{eq:dim-assump}.
\end{proof}

\subsection{The probability of a given breakup}
\label{sec:prob-of-given-breakup}

In this section, we prove Proposition~\ref{prop:prob-of-given-breakup}.
Fix $X \in \breakups_{L,M,N}$ and let $\Omega$ be the set of proper colorings $f$ having $X$ as a breakup.
In order to bound the probability of $\Omega$, we aim to apply Lemma~\ref{lem:bound-on-pseudo-breakup} with
\[ S:=X_* \qquad\text{and}\qquad S_P := X_P \setminus X_* .\]
The definition of $X_*$ implies that $\{ S_P^+ \}_P$ are pairwise disjoint so that, in particular, $\{ S_P \}_P$ is a partition of $S^c$.
By~\eqref{eq:breakup-0}, $S \cup S_{P_0}$ contains $(\Lambda^c)^+$.
By~\eqref{eq:breakup-prop-even}, \eqref{eq:breakup-prop-odd} and~\eqref{eq:def-atlas}, $S_P^+ \cap S^{+2}$ is in the $P$-pattern on the event $\Omega$.
Thus, the assumptions of Lemma~\ref{lem:bound-on-pseudo-breakup} are satisfied.

The following lemma guarantees that there are many restricted edges in $(f,\Omega)$.
Recall the definitions of $S^f_\unbal$, $S^f_\nondom$, $S^{\Omega,f}_\rest$ and $S^\Omega_\unique$ from Section~\ref{sec:Shearer_overview}.

\begin{lemma}\label{lem:lower-bound-on-size-of-restricted}
	For any $f \in \Omega$, we have
	\[ \big|S^{\Omega,f}_\rest\big| \ge L \qquad\text{and}\qquad \tfrac{1}{d}\big|S^{\Omega,f}_\rest\big|+2\big|S^f_\nondom\big| \ge M .\]
\end{lemma}
\begin{proof}
	Fix $f \in \Omega$ and write $S_\rest$ for $S^{\Omega,f}_\rest$ and $S_\nondom$ for $S^{\Omega,f}_\nondom$.
	
	To show that $|S_\rest| \ge L$, it suffices to show that
\begin{equation}\label{eq:partial_X_P_restricted}
\dpartialrev X_P \subset S_\rest \qquad\text{for any }P .
\end{equation}
To this end, let $(v,u) \in \dpartialrev X_P$. Then $g(u) \in P_\bdry$ and $g(N(v)) \not\subset P_\bdry$ for any $g \in \Omega$ by~\eqref{eq:breakup-prop-bdry-X_P}, from which it follows that $(v,u)$ is restricted by~\eqref{eq:restricted_A_def}.

	We now show that $\tfrac{1}{d}|S_\rest|+2|S_\nondom| \ge M$.
	Letting $S^*_\nondom$ denote the set of edges having an endpoint in $S_\nondom$ and noting that $2|S_\nondom| \ge \tfrac{1}{d}|S^*_\nondom|$, we see that it suffices to show that
	\[ (v,u) \in S_\rest \quad\text{or}\quad u \in S_\nondom \quad\text{or}\quad v \in S_\nondom \qquad\text{for any }u \in X_\overlap\text{ and }v \sim u .\]
	Let $u \in X_\overlap$ and let $P \neq Q$ be such that $u \in X_P \cap X_Q$. If $v \notin X_P \cap X_Q$ then $(v,u)$ is restricted by~\eqref{eq:partial_X_P_restricted}. Otherwise, $v \in X_P \cap X_Q$.
Recall the definitions of $\phase_0$ and $\phase_1$ from Section~\ref{sec:overview-notation}	and that $P \simeq Q$ means that $P,Q \in \phase_0$ or $P,Q \in \phase_1$. If $P \simeq Q$, letting $w \in \{u,v\}$ be $P$-odd, we have $w^+ \subset X_P \cap X_Q$ by~\eqref{eq:def-atlas}. Thus, $g(N(w)) \subset P_\bdry \cap Q_\bdry$ for any $g \in \Omega$ by~\eqref{eq:breakup-prop-even}, and it follows that $w \in S_\nondom$.
	Otherwise, $P \not\simeq Q$ and we may assume without loss of generality that $v$ is $P$-even and $u$ is $Q$-even, in which case $g(v) \in P_\bdry$ and $g(u) \in Q_\bdry$ for any $g \in \Omega$ by~\eqref{eq:breakup-prop-even}, so that it follows from~\eqref{eq:restricted_def} that $(v,u)$ is restricted (note that $P \not\simeq Q$ can only occur when $q$ is odd, since $\phasedom_1$ is empty when $q$ is even).
\end{proof}

As explained in Section~\ref{sec:Shearer_overview}, applying Lemma~\ref{lem:bound-on-pseudo-breakup} directly for $\Omega$ does not produce the bound stated in Proposition~\ref{prop:prob-of-given-breakup}. This bound will instead follow by applying Lemma~\ref{lem:bound-on-pseudo-breakup} to subevents of $\Omega$ on which we have additional information about the coloring on the set $X_\bad$ and then summing the resulting bounds. To explain the reason for this and to motivate the definitions below, we note that, although~\eqref{eq:breakup-prop-bad} prohibits the possibility that the neighborhood $N(v)$ of a $P$-odd vertex $v \in X_\bad$ is in the $P$-pattern, this is possible for a $P$-even vertex. That is, when $q$ is even, it cannot happen that $|f(N(v))|=\tfrac{q}{2}$ for an odd vertex, but it may happen that $|f(N(v))|=\tfrac{q}{2}$ for an even vertex, and when $q$ is odd, it cannot happen that $|f(N(v))|=\lfloor \tfrac{q}{2} \rfloor$, but it may happen that $|f(N(v))|=\lceil \tfrac{q}{2} \rceil$. A vertex for which the latter occurs is problematic as it does not immediately reduce the entropy of the configuration (since it may also have a balanced neighborhood and no or few restricted edges incident to it). For even $q$, this issue is not important as $X_\bad$ is an odd set, so that at least half of its vertices are odd. For odd $q$, however, it may happen that many (perhaps even all or almost all) of the vertices in $X_\bad$ are of this type (see Figure~\ref{fig:breakup}). By recording the location of a small subset of these vertices and the dominant patterns in their neighborhoods, we may ensure that most vertices in $X_\bad$ become restricted in some manner (unbalanced neighborhood, non-dominant vertex, or many incident restricted edges). We now describe the structure of this additional information.

For $f \in \Omega$ and a dominant pattern $P$, define
\begin{equation}\label{eq:U_P_def}
U_P(f) := \big\{ u \in X_\bad : \text{$u$ is $P$-even, }f(N(u)) = P_\inner \big\} .
\end{equation}
Note that the sets $\{ U_P(f) \}_P$ are pairwise disjoint.
Note also that $u \in U_P(f)$ implies that $u^+$ is in the $P$-pattern and that $N(u)$ is not in the $Q$-pattern for any $Q \neq P$. In particular,
\begin{equation}\label{eq:U_P_inpattern}
f(U_P(f)) \subset P_\bdry .
\end{equation}

The collection $(U_P(f))_P$ contains the relevant information on $f$ beyond that which is given by the breakup $X$. However, it contains more information than is necessary and this comes at a large enumeration cost. Instead, we wish to specify only a certain approximation of this information.
Given a collection $V=(V_P)_P$ of subsets of $\Z^d$, let $\Omega(V)$ denote the set of $f \in \Omega$ satisfying that, for every dominant pattern $P$,
\begin{align}
V_P &\subset U_P(f) &\text{and}&& N_{d/3q}\Bigg(\bigcup_{Q \neq P} U_Q(f)\Bigg) &\subset N\Bigg(\bigcup_{Q \neq P} V_Q\Bigg). & \label{eq:iso-approx}
\end{align}
Thus, $V$ is a kind of approximation of $(U_P(f))_P$.
With this definition at hand, there are now two goals. The first is to show that the additional information given by $V$ is enough to improve the bound given in Lemma~\ref{lem:lower-bound-on-size-of-restricted}.
The second is to show that the cost of enumerating $V$ is not too large.

\begin{lemma}\label{lem:lower-bound-on-size-of-restricted-with-V}
	For any $V$ and any $f \in \Omega(V)$, we have
	\[ \tfrac{1}{d} \big|S^{\Omega(V),f}_\rest\big| + \tfrac{1}{q} |S^f_\nondom| + \big|S^f_\unbal\big| \ge \tfrac{L}{3d} + \tfrac{M}{6q} + \tfrac{N}{18q^2} .\]
\end{lemma}
\begin{proof}
	We fix $V$ and $f \in \Omega(V)$ and suppress them in the notation of $S^f_\unbal$, $S^f_\nondom$, $S^{\Omega(V),f}_\rest$, $U_P(f)$.
	It suffices to show that
	\[ \tfrac{1}{d} |S_\rest| + \tfrac{1}{q} |S_\nondom| + |S_\unbal| \ge \tfrac{N}{6q^2} ,\]
	as the lemma then follows by averaging this bound with the ones given by Lemma~\ref{lem:lower-bound-on-size-of-restricted}. In fact, we will show the slightly stronger inequality
	\[ N \le \tfrac{6q}d |S_\rest| + 6q|S_\nondom| + |S_\unbal| .\]	
	Let $S^*_\rest$ denote the set of vertices which are incident to at least $d/3q$ edges in $S_\rest$, i.e.,
	\[ S^*_\rest := \Big\{ v : \big|(\dpartial v \cup \dpartialrev v) \cap S_\rest\big| \ge \tfrac{d}{3q} \Big\} .\]
	Note that $|S_\rest| \ge \frac{d}{6q}|S^*_\rest|$ and $N_{d/3q}(S_\nondom) \le 6q |S_\nondom|$ by Lemma~\ref{lem:sizes}. It thus suffices to show that
	\begin{equation}\label{eq:X_bad_restricted}
	X_\bad \subset S^*_\rest \cup N_{d/3q}(S_\nondom) \cup S_\unbal .
	\end{equation}

	Let us first show that
	\begin{equation}\label{eq:X_bad_restricted1}
	X_\bad \setminus U \subset S_\nondom , \qquad\text{where }U := \bigcup_P U_P .
	\end{equation}
	To this end, let $u \in X_\bad \setminus U$ and note that, by the definition of a non-dominant vertex, we must show that $|f(N(u))| \notin \{\lfloor \tfrac{q}{2} \rfloor, \lceil \tfrac{q}{2} \rceil\}$.
	Let us consider separately the cases of even and odd $q$. Assume first that $q$ is even. Note that $|f(N(u))| \neq \tfrac{q}{2}$ by~\eqref{eq:breakup-prop-bad} if $u$ is odd and that $|f(N(u))| \neq \tfrac{q}{2}$ by~\eqref{eq:U_P_def} if $u$ is even. Assume now that $q$ is odd. Note that $|f(N(u))| \neq \lfloor\tfrac{q}{2}\rfloor$ by~\eqref{eq:breakup-prop-bad} and that $|f(N(u))| \neq \lceil\tfrac{q}{2}\rceil$ by~\eqref{eq:U_P_def}. This establishes~\eqref{eq:X_bad_restricted1}.
		
	Next, we show that
	\begin{equation}\label{eq:X_bad_restricted2}
	\bigcap_P N_{d/3q}(U \setminus U_P) \subset S^*_\rest .
	\end{equation}
	To see this, let $u \in \bigcap_P N_{d/3q}(U \setminus U_P)$ and note that, by~\eqref{eq:iso-approx}, $u \in N(V_P)$ for some $P$. Since $u \in N_{d/3q}(U \setminus U_P)$, another application of~\eqref{eq:iso-approx} yields that $u \in N(V_Q)$ for some $Q \neq P$.
	Since $V_P \subset U_P$ and $V_Q \subset U_Q$ by~\eqref{eq:iso-approx}, it follows from~\eqref{eq:U_P_def} that $g(u) \in P_\inner \cap Q_\inner$ for any $g \in \Omega(V)$.
	Since $u \in N_{d/3q}(U)$, in order to show that $u \in S^*_\rest$, it suffices to show that if $v \in N(u) \cap U_T$ for some $T$, then $(v,u)$ is restricted. Indeed, this follows since $f(N(v)) = T_\inner$ by~\eqref{eq:U_P_def}, which implies that $(v,u)$ is restricted by~\eqref{eq:restricted_A_def}. This establishes~\eqref{eq:X_bad_restricted2}.

	Finally, towards showing~\eqref{eq:X_bad_restricted}, let $u \in X_\bad$ and assume that
	\[ u \notin S^*_\rest \cup N_{d/3q}(S_\nondom) .\]
	We show that $u \in S_\unbal$. By~\eqref{eq:X_bad_restricted1}, we have $u \notin N_{d/3q}(X_\bad \setminus U)$ so that $u \in N_{2d-d/3q}(\bigcup_P X_P \cup U)$.
	Since $X_\bad \cap N_{d/3q}(\bigcup_P X_P) \subset S^*_\rest$ by~\eqref{eq:partial_X_P_restricted}, it follows that $u \in N_{2d-2d/3q}(U)$.
	Hence, by~\eqref{eq:X_bad_restricted2}, we have that $u \in N_{2d-d/q}(U_P)$ for some~$P$. In particular, $|N(u) \cap f^{-1}(P_\bdry)| \ge 2d-d/q$ by~\eqref{eq:U_P_inpattern}. Since $f(N(u)) \not\subset P_\bdry$ by~\eqref{eq:breakup-prop-bad} (note that $u$ is $P$-odd as it is adjacent to $U_P$), it follows that $u \in S_\unbal$.
\end{proof}

\begin{lemma}\label{lem:family-of-strong-odd-approx}
	There exists a family $\cV$ satisfying that
	\[ |\cV| \le \exp\left(\tfrac{CNq(q+\log d)\log d}{d}\right) \qquad\text{and}\qquad \Omega \subset \bigcup_{V \in \cV} \Omega(V) .\]
\end{lemma}
\begin{proof}
	Let $\cV$ be the collection of all $(V_P)_P$ such that $\{V_P\}_P$ are disjoint subsets of $X_\bad$ having $\sum_P |V_P| \le 3rN$, where $r := 3q(1+\log 2d)/d$. Let us check that $\cV$ satisfies the requirements of the lemma. Since $|X_\bad|=N$, we have
	\[ |\cV| \le \binom{N}{\le 3rN} \cdot (2^q)^{3rN} \le \left(\frac{e2^q}{3r}\right)^{3rN} \le e^{CNq(q+\log d)(\log d)/d} .\]
	
	Fix $f \in \Omega$. We must find a collection $(V_P)_P \in \cV$ for which~\eqref{eq:iso-approx} holds. We write $U_P$ for $U_P(f)$, and we denote $U_I := \bigcup_{P \in I} U_P$ for $I \subset \phasedom$ and $U := U_{\phasedom}$.
	Define a bipartite graph $G$ with vertex set $(\Z^d \times \{0,1\}) \cup U$ as follows. For each $v \in \Z^d$, let $I_v$ be a minimal set of dominant patterns for which $|N(v) \cap U_{I_v}| \ge \tfrac13 |N(v) \cap U|$, and place an edge between $(v,i) \in \Z^d \times \{0,1\}$ and $u \in U$ if and only if $v \sim u$ and $\1(u \in U_{I_v})=i$.
	Note that $G$ has maximum degree at most $2d$.
	
	By Lemma~\ref{lem:existence-of-covering2} applied to $G$ with $t=d/9q$, we obtain a set $W \subset U$ of size $|W| \le 3rN$ such that
	\[ v \in N_{d/9q}(U_I) \implies v \in N(W \cap U_I) \qquad\text{for any }v \in \Z^d\text{ and }I \in \{I_v, \phasedom \setminus I_v \}.\]
	Set $V_P := W \cap U_P$ for all $P$ and note that $W = \bigcup_P V_P$. Towards showing~\eqref{eq:iso-approx}, let $P \in \phasedom$ and $v \in N_{d/3q}(U \setminus U_P)$.	
	Suppose first that $P \notin I_v$.
	Then
	\[ v \in N_{d/3q}(U) \subset N_{d/9q}(U_{I_v}) \subset N(W \cap U_{I_v}) \subset N(W \setminus U_P) = N(W \setminus V_P) .\]
    Suppose next that $P \in I_v$. By the minimality of $I_v$, either $I_v=\{P\}$ or $|N(v) \cap U_{I_v}| < \tfrac23 |N(v) \cap U|$. In either case, we have $|N(v) \cap U_{\phasedom \setminus I_v}| \ge d/9q$ so that
    \[ v \in N_{d/9q}(U_{\phasedom \setminus I_v}) \subset N(W \cap U_{\phasedom \setminus I_v}) \subset N(W \setminus U_P) = N(W \setminus V_P) . \qedhere \]
\end{proof}

\begin{lemma}\label{lem:vertices-with-unique-pattern}
	$S \setminus X_\bad \subset S^\Omega_\unique$.
\end{lemma}
\begin{proof}
	Let $v \in S \setminus X_\bad$ and note that there exists $P$ such that $v \in X_P$.
	Assume first that $v$ is $P$-even. Then, by~\eqref{eq:breakup-prop-even}, $g(v) \in P_\bdry$ for all $g \in \Omega$, so that if $g(N(v)) \neq P_\inner$ then either $|g(N(v))| \notin \{\lfloor \tfrac{q}{2} \rfloor, \lceil \tfrac{q}{2} \rceil\}$ or all edges in $\dpartial v$ are restricted in $g$ by~\eqref{eq:restricted_B_def}. Hence, $v$ has a unique pattern.
	Assume next that $v$ is $P$-odd. Then $v^+ \subset X_P$ by~\eqref{eq:def-atlas} so that, by~\eqref{eq:breakup-prop-even}, $g(N(v)) \subset P_\bdry$ for all $g \in \Omega$. Thus, either $g(N(v)) = P_\bdry$ or $|g(N(v))|<\lfloor \tfrac{q}{2} \rfloor$. In particular, $v$ has a unique pattern.
\end{proof}

\begin{proof}[Proof of Proposition~\ref{prop:prob-of-given-breakup}]
Note that $|S| \le 2L+M+N$.
Thus, Lemma~\ref{lem:bound-on-pseudo-breakup}, Lemma~\ref{lem:lower-bound-on-size-of-restricted-with-V} and Lemma~\ref{lem:vertices-with-unique-pattern} imply that, for any $V$,
	\[ \Pr(\Omega(V)) \le \exp\Big(-\tfrac{1}{32q} \left(\tfrac{L}{3d} + \tfrac{M}{6q} + \tfrac{N}{18q^2}\right)+ \tfrac{qN}{d}+e^{-d/65q^2} (2L+M+N)\Big) .\]
Therefore, by Lemma~\ref{lem:family-of-strong-odd-approx} and~\eqref{eq:dim-assump},
\begin{align*}
\Pr(\Omega) &\le \exp\left(\tfrac{CNq(q+\log d)\log d}{d} + e^{-cd/q^2} (2L+M+N) - \tfrac{c}{q} \big(\tfrac{L}{d}+\tfrac{M}{q}+\tfrac{N}{q^2}\big) \right) \le e^{- \frac{c}{q} \big( \frac{L}{d}+\frac{M}{q}+\frac{N}{q^2} \big)} . \qedhere
\end{align*}
\end{proof}

\subsection{The probability of an approximated breakup}
\label{sec:prob-of-approx}

In this section, we prove Proposition~\ref{prop:prob-of-odd-approx}.
Fix integers $L,M,N \ge 0$ and an approximation $A$.
Denote
\[ A_\bad := \bigcap_P (A_P \cup A^{**})^c ,\qquad A_\overlap := \bigcup_{P \neq Q} (A_P \cap A_Q), \qquad U := A^{**} \cup A_\bad \cup A_\overlap. \]
Further define
\[  S_P := \Int(A_P \setminus U) \qquad\text{and}\qquad S := \bigcap_P (S_P)^c .\]
Note that $U^+ \subset S$, that $\{S_P\}_P$ is a partition of $S^c$ and that $\{ S_P^+ \}_P$ are pairwise disjoint.
Let $X$ be an atlas which is approximated by $A$. Note that, by~\ref{it:approx-A_P},
\[ A_\bad \subset X_\bad \subset A_\bad \cup A^{**}, \quad A_\overlap \subset X_\overlap \subset A_\overlap \cup A^{**}, \quad U = A^{**} \cup A_\bad \cup A_\overlap .\]

\begin{claim}\label{claim:S}
\[  S = X_* \cup (A^{**})^+ .\]
\end{claim}
\begin{proof}
Let us first show that $S \subset X_* \cup (A^{**})^+$.
Let $v \in S$ and note that $v \notin \Int(A_P \setminus U)$ for all $P$. Thus, for any $P$, there exists $u \in v^+$ such that $u \notin A_P$ or $u \in U$. If the latter occurs for some $P$, then $u \in U \subset A^{**} \cup X_*$ and we are done. Otherwise, for every $P$, there exists $u \in v^+$ such that $u \notin A_P$. That is, $u \in \bigcap_P \Int(A_P)^c$. Suppose that $u \notin X_*$ so that $u \in \Int(X_P)$ for some $P$. By~\ref{it:approx-A_P}, $u \in \Int(A_P \cup A^{**})$. Since $u \notin \Int(A_P)$, it must be that $u \in (A^{**})^+$.

Let us now show that $X_* \cup (A^{**})^+ \subset S$.
Since $A^{**} \subset U$ and $U^+ \subset S$, we see that $(A^{**})^+ \subset S$. Similarly, $X_\bad \cup X_\overlap \subset U \subset S$. It remains to show that $\bigcup_P \intextB X_P \subset S$.
Let $v \in \intextB X_P$ for some $P$ and suppose towards a contradiction that $v \in S_Q$ for some $Q$. Then~\ref{it:approx-A_P} implies that $v \in \Int(X_Q \setminus X_\overlap)$, which contradicts the fact that $v \in \intextB X_P$.
\end{proof}

Thus, using~\ref{it:approx-unknown-location}, we see that $S \subset X_*^{+4}$.
Recall that $A_P \subset X_P$ by~\ref{it:approx-A_P} and note that $\intextB S_P \subset \intextB S \cap X_P \setminus X_\overlap$. Thus, \eqref{eq:breakup-prop-even} and~\eqref{eq:breakup-prop-odd} imply that, for any coloring~$f$ having $X$ as a breakup,
\begin{equation}\label{eq:prob-of-approx-1}
\Even_P \cap A_P \cap S^+\text{ and }\intextB S_P\text{ are in the $P$-pattern}.
\end{equation}
Finally, by~\eqref{eq:breakup-0}, \ref{it:approx-A_P} and the fact that $X_* \subset S$, we have that $S \cup S_{P_0}$ contains $(\Lambda^c)^+$.
We have thus established that the assumptions of Lemma~\ref{lem:bound-on-pseudo-breakup} are satisfied for the sets $(S,(S_P)_P)$ and the event $\Omega$ that $A$ approximates some breakup in $\breakups_{L,M,N}$.

\begin{lemma}\label{lem:vertices-with-unique-pattern-approx}
Every vertex in $S \setminus U$ has a unique pattern. That is, $S \setminus U \subset S^\Omega_\unique$.
\end{lemma}
\begin{proof}
	The proof is essentially the same as that of Lemma~\ref{lem:vertices-with-unique-pattern}.
Let $v \in S \setminus U$ and note that $v \notin A^{**} \cup A_\bad$ so that $v \in A_P$ for some $P$.
Assume first that $v$ is $P$-even. Then, by~\eqref{eq:prob-of-approx-1}, we have $g(v) \in P_\bdry$ for all $g \in \Omega$. Thus, by~\eqref{eq:restricted_B_def}, if $g(N(v)) \neq P_\inner$ and $|g(N(v))| \in \{\lfloor \tfrac{q}{2} \rfloor, \lceil \tfrac{q}{2} \rceil\}$, then all edges in $\dpartial v$ are restricted in $g$. Hence, $v$ has a unique pattern.
Assume next that $v$ is $P$-odd. Then, since $A_P$ is $P$-even, $v^+ \subset A_P$ so that $g(N(v)) \subset P_\bdry$ for all $g \in \Omega$ by~\eqref{eq:prob-of-approx-1}. Thus, either $g(N(v)) = P_\bdry$ or $|g(N(v))|<\lfloor\tfrac{q}{2}\rfloor$. In particular, $v$ has a unique pattern.	
\end{proof}

The proof of Proposition~\ref{prop:prob-of-odd-approx} is based on the idea that one of two situations can occur: either there are enough restricted edges so that one may directly apply Lemma~\ref{lem:bound-on-pseudo-breakup} to obtain the desired bound, or there are not many possible breakups so that one may apply Proposition~\ref{prop:prob-of-given-breakup} together with a union bound. At the heart of this approach lies the following lemma which informally states that an unknown vertex (of a certain type) either incurs an entropic loss (in the sense that it is non-dominant or adjacent to many restricted edges) or there is a unique way to determine to which~$X_P$'s it belongs. We now make this precise.

Denote
\[ S^{\Omega,f,1/2}_\rest := \Big\{ v : \big|\dpartial v \cap S^{\Omega,f}_\rest\big| \ge \tfrac{d}{2} \Big\} .\]
For an atlas $X$, let $\Omega_X$ denote the event that $X$ is a breakup.
With a slight abuse of notation, denote
\[ S^{\Omega,X,1/2}_\rest := \bigcap_{f \in \Omega_X} S^{\Omega,f,1/2}_\rest \qquad\text{and}\qquad S^X_\nondom := \bigcap_{f \in \Omega_X} S^f_\nondom .\]

\begin{lemma}\label{lem:breakup-recovery}
	Let $X$ be an atlas which is approximated by $A$, let $P$ be a dominant pattern and let $v \in A^*$ be a $P$-odd vertex. Then
	\[ \text{either}\qquad v \in S^{\Omega,X,1/2}_\rest \cup S^X_\nondom \qquad\text{or}\qquad  v \in X_P \iff v \in N_{d/2}(A_P) .\]
\end{lemma}
\begin{proof}
	Fix a dominant pattern $P$ and a $P$-odd vertex $v \in A^*$. Recall that $v^+ \subset (A^{**})^+ \subset S \subset X_*^{+4}$.
	Denote $I := \{ Q \simeq P : v \in X_Q \}$.
	We first show that
	\[ |I|>1 \implies v \in S^X_\nondom .\]
	Indeed, if $Q,T \in I$ are distinct, then $f(N(v)) \subset Q_\bdry \cap T_\bdry$ for any $f \in \Omega_X$ by~\eqref{eq:def-atlas} and~\eqref{eq:breakup-prop-even}, and it follows that $v$ is a non-dominant vertex in $f$.
	
	Next, we show that
	\[ \text{for every $Q \simeq P$ and $u \in N(v) \cap A_Q$,}\qquad Q \notin I \implies (v,u) \in S^{\Omega,f}_\rest\text{ for all }f \in \Omega_X .\]
	To this end, let $Q \simeq P$, $u \in N(v) \cap A_Q$ and $f \in \Omega_X$, and note that $g(u) \in Q_\bdry$ for all $g \in \Omega$ by~\eqref{eq:prob-of-approx-1}.
	If $Q \notin I$ then $f(N(v)) \not\subset Q_\bdry$ by~\eqref{eq:breakup-prop-bdry-X_P} so that $(v,u)$ is restricted by~\eqref{eq:restricted_A_def}.
	
	Suppose now that $v \notin S^{\Omega,X,1/2}_\rest \cup S^X_\nondom$. Note that $v \in N_d(\bigcup_{Q \simeq P} A_Q)$ by~\ref{it:approx-A*_P}. It therefore follows from what we have just shown that $I=\{Q\}=\{T \simeq P : v \in N_{d/2}(A_T)\}$ for some $Q \simeq P$. In particular, $v \in X_P$ if and only if $P=Q$ if and only if $v \in N_{d/2}(A_P)$.
\end{proof}

\begin{proof}[Proof of Proposition~\ref{prop:prob-of-odd-approx}]
Let $a>0$ be a small constant to be chosen later.
Consider the event
\[ \Omega' := \Omega \cap \Big\{ \big|S^{\Omega,f,1/2}_\rest\big| + \tfrac{1}{q} \big|S^f_\nondom\big| \ge \tfrac{a}{q^2(q+\log d)} \big(\tfrac{L}{d} + \tfrac{M}{q} + \tfrac{N}{q^2}\big) \Big\} .\]
We bound separately the probabilities of $\Omega'$ and $\Omega \setminus \Omega'$.
Let us begin with $\Omega'$.
Note that
\[ \big|S^{\Omega',f}_\rest\big| \ge \big|S^{\Omega,f}_\rest\big| \ge \tfrac14 d\big|S^{\Omega,f,1/2}_\rest\big| \qquad\text{for any }f \in \Omega .\]
By~\ref{it:approx-unknown-size} and Claim~\ref{claim:S},
\begin{equation}\label{eq:A**-U-S-sizes}
|A^{**}| \le CLd^{-1/2}\log d, \qquad |U| \le M+N+|A^{**}|, \qquad |S| \le 2L+M+N+(2d+1)|A^{**}| .
\end{equation}
Using~\eqref{eq:dim-assump} and Lemma~\ref{lem:vertices-with-unique-pattern-approx}, we may apply Lemma~\ref{lem:bound-on-pseudo-breakup} to obtain
\[ \Pr(\Omega') \le \exp\Big(-\tfrac{ca}{q^3(q+\log d)}\big(\tfrac{L}{d}+\tfrac{M}{q}+\tfrac{N}{q^2}\big) + \tfrac{q}{d}|U| + e^{-cd/q^2}|S| \Big) \le e^{- \frac{ca}{q^3(q+\log d)} ( \frac{L}{d}+\frac{M}{q}+\frac{N}{q^2})} .\]

We now bound the probability of $\Omega \setminus \Omega'$.
To do this, as explained above, we recover the breakup and then apply Proposition~\ref{prop:prob-of-given-breakup}.
Formally, let $\cB$ be the collection of atlases $X \in \breakups_{L,M,N}$ which are approximated by $A$ and have
\begin{equation}\label{eq:prob-of-approx-size-of-restricted}
 \big|S^{\Omega,X,1/2}_\rest\big| + \tfrac{1}{q} \big|S^X_\nondom\big| < \tfrac{a}{q^2(q+\log d)}\big(\tfrac{L}{d}+\tfrac{M}{q}+\tfrac{N}{q^2}\big).
\end{equation}
Note that $\Omega \setminus \Omega' \subset \bigcup_{X \in \cB} \Omega_X$.
We shall show that
\begin{equation}\label{eq:prob-of-approx-size-of-B}
|\cB| \le e^{\frac{Ca}{q} (\frac{L}{d}+\frac{M}{q}+\frac{N}{q^2})}.
\end{equation}
Using Proposition~\ref{prop:prob-of-given-breakup}, when $a$ is chosen to be sufficiently small, this will then yield that
\[ \Pr(\Omega \setminus \Omega') \le \sum_{X\in\cB} \Pr(\Omega_X) \le e^{-\frac{c}{q}(\frac{L}{d}+\frac{M}{q}+\frac{N}{q^2})} . \]

Toward establishing~\eqref{eq:prob-of-approx-size-of-B}, we show that the mapping
\begin{equation}\label{eq:prob-of-approx-mapping}
X ~\longmapsto~ \left(S^{\Omega,X,1/2}_\rest \cup S^X_\nondom,~ (I_X(v))_{v \in S^{\Omega,X,1/2}_\rest \cup S^X_\nondom}\right)
\end{equation}
is injective on $\cB$, where
\[ I_X(v) := \big\{ P\in \phasedom : v \in \Odd_P \cap X_P \big\} .\]
By~\eqref{eq:def-atlas} and~\ref{it:approx-A_P}, we have
\[ X_P = (\Odd_P \cap X_P)^+ = (\Odd_P \cap (A_P \cup (X_P \cap A^*)))^+ \qquad\text{for all }X \in \cB\text{ and all }P .\]
Thus, to determine $X_P$, we only need to know the set $\Odd_P \cap X_P \cap A^*$. In other words, we only need to know for each vertex $v \in \Odd_P \cap A^*$, whether it belongs to $X_P$ or not. If $v \in S^{\Omega,X,1/2}_\rest \cup S^X_\nondom$ then this is given by $I_X(v)$, and otherwise, Lemma~\ref{lem:breakup-recovery} implies that this is determined by the approximation. Thus, the mapping is injective.

Let $\cR$ be the image of the mapping in~\eqref{eq:prob-of-approx-mapping} as $X$ ranges over $\cB$. As this mapping is injective, we have $|\cB| = |\cR|$.
The bound~\eqref{eq:prob-of-approx-size-of-B} will then easily follow once we show that
\begin{equation}\label{eq:possible-patterns-a-vertex-can-be-in}
|\{I_X(v) : X \in \breakups,\, v \in \Even\}| \le 2^q \qquad\text{and}\qquad |\{I_X(v) : X \in \breakups,\, v \in \Odd\}| \le 2^q.
\end{equation}
Indeed, \eqref{eq:A**-U-S-sizes}, \eqref{eq:prob-of-approx-size-of-restricted} and~\eqref{eq:possible-patterns-a-vertex-can-be-in} imply that
\[ |\cB| = |\cR| \le \binom{|S|}{\le \tfrac{a}{q(q+\log d)} (\tfrac{L}{d}+\tfrac{M}{q}+\tfrac{N}{q^2})} \cdot (2^q)^{\frac{a}{q(q+\log d)} (\frac{L}{d}+\frac{M}{q}+\frac{N}{q^2})} \le e^{\frac{Ca}{q} (\frac{L}{d}+\frac{M}{q}+\frac{N}{q^2})} .\]

To show~\eqref{eq:possible-patterns-a-vertex-can-be-in}, it suffices to show that, for any $X \in \breakups$ and $v \in \Z^d$, there exists $I \subset [q]$ such that
\[ I_X(v) = \{ P \in \phase_i : I \subset P_\bdry\} ,\qquad\text{where } i := \1_{\{v\text{ is even}\}} .\]
To see this, let $f \colon \Z^d \to [q]$ be such that $X$ is a breakup of $f$, and set $I := f(N(v))$. By~\eqref{eq:breakup-1}, for $P \in \phase_i$, we have $v \in X_P$ if and only if $I \subset P_\bdry$. For $P \in \phasedom \setminus \phase_i$, we clearly have $P \notin I_X(v)$, since $v$ is $P$-even.
\end{proof}

\section{Repair transformation and Shearer's inequality}
\label{sec:shift-trans}

In this section, we prove the following generalization of Lemma~\ref{lem:bound-on-pseudo-breakup}. Recall from Section~\ref{sec:overview-notation} that $\Lambda \subset \Z^d$ a fixed domain outside which the coloring is forced to be in the $P_0$-pattern.

\begin{lemma}\label{lem:bound-on-pseudo-breakup-expectation}
	Let $S \subset \Z^d$ be finite and let $\{ S_P \}_{P \in \phasedom}$ be a partition of $S^c$ such that $\intB S_P \subset \extB S$ for all $P$.
	Suppose that $S \cup S_{P_0}$ contains $(\Lambda^c)^+$.
	Let $E$ be an event which is determined by the values of $f$ on $S^+$. Let $\Omega$ be the event that $E$ occurs and $(\intB S_P)^+$ is in the $P$-pattern for every~$P$.
	Then
	\[ \Pr_{\Lambda,P_0}(\Omega) \le \exp\left[-\tfrac{1}{128q} \E\left(\big|S^f_\unbal\big| +\tfrac{1}{q}\big|S^f_\nondom\big| + \tfrac{1}{d}\big|S^{\Omega,f}_\rest\big| \right) + \tfrac{q}{d}\big|S \setminus S^\Omega_\unique\big| +e^{-d/65q^2}|S| \right] ,\]
	where the expectation is taken with respect to a random function $f$ chosen from $\Pr_{\Lambda,P_0}(\cdot \mid \Omega)$.
\end{lemma}

Let us show how this lemma yields Lemma~\ref{lem:bound-on-pseudo-breakup}.
\begin{proof}[Proof of Lemma~\ref{lem:bound-on-pseudo-breakup}]
Let $E$ be the event that $f|_{S^+}=\phi|_{S^+}$ for some $\phi \in \Omega$ and let $\Omega'$ be the event that $E$ occurs and $(\intB S_P)^+$ is in the $P$-pattern for all $P$.
	Note that $E$ is determined by $f|_{S^+}$, $\Omega \subset \Omega'$, $k(\Omega)=k(\Omega')$ and $S^\Omega_\unique = S^{\Omega'}_\unique$. Thus, Lemma~\ref{lem:bound-on-pseudo-breakup} follows from Lemma~\ref{lem:bound-on-pseudo-breakup-expectation}.
\end{proof}

The proof of Lemma~\ref{lem:bound-on-pseudo-breakup-expectation} is based on a general upper bound on the total number of colorings in an event, given in Proposition~\ref{lem:shearer-for-bad-set} below.
For a set $U \subset \Z^d$, we denote $U^\even := \Even \cap U$ and $U^\odd := \Odd \cap U$.
For two sets $U,V \subset \Z^d$, we denote
\[ \partial^\even(U,V) := \partial(U^\even,V^\odd) \qquad\text{and}\qquad \partial^\odd(U,V) := \partial(U^\odd,V^\even) ,\]
so that $\partial(U,V)=\partial^\even(U,V) \cup \partial^\odd(U,V)$. We also write $\partial^\even U := \partial^\even(U,U^c)$ and $\partial^\odd U := \partial^\odd(U,U^c)$, and for a dominant pattern $P$, we use the notation $\partial^{P\text{-}\even}$ and $\partial^{P\text{-}\odd}$ with the meanings inferred from the notions of $P$-even and $P$-odd.
Recall the notions of non-dominant vertex, restricted edge, unbalanced neighborhood and unique pattern defined in Section~\ref{sec:Shearer_overview}. Note that, although those notions were defined for proper colorings $f$ of $\Z^d$, they are well defined for any $v \in S$ when $f$ is a proper coloring of $S^+$.

\begin{prop}\label{lem:shearer-for-bad-set}
	Let $S \subset \Z^d$ be finite and let $\{ S_P \}_{P \in \phasedom}$ be a partition of $S^c$. Let $\cF$ be a set of proper colorings of $S^+$ satisfying that $S^+ \cap (\intB S_P)^+$ is in the $P$-pattern for every $P$.
	Sample $f \in \cF$ uniformly at random.
	Then
	\[ \begin{aligned} |\cF| \le (\lfloor\tfrac{q}{2}\rfloor \lceil\tfrac{q}{2}\rceil)^{\frac12 |S^+|} &\cdot \exp\left[-\tfrac{1}{128q} \E\Big(\big|S^f_\unbal\big| +\tfrac{1}{q}\big|S^f_\nondom\big|+ \tfrac{1}{d}\big|S^{\cF,f}_\rest\big| \Big)+ \tfrac{q}{d}\big|S \setminus S^{\cF}_\unique\big| +e^{-d/65q^2}|S|\right] \\ &\cdot \prod_P \Big(\lfloor \tfrac{q}{2} \rfloor/\lceil \tfrac{q}{2} \rceil\Big)^{\frac{1}{4d}(|\partial^{P\text{-}\even} (S^+,S_P \setminus S^+)|-|\partial^{P\text{-}\odd} (S^+,S_P \setminus S^+)|)} . \end{aligned}\]
\end{prop}

Before proving the proposition, let us show it implies Lemma~\ref{lem:bound-on-pseudo-breakup-expectation} and thus Lemma~\ref{lem:bound-on-pseudo-breakup}.

\begin{proof}[Proof of Lemma~\ref{lem:bound-on-pseudo-breakup-expectation}]
The reader may find it useful to consult Figure~\ref{fig:repairmap} where a repair transformation similar to the one used below is illustrated (with $S=X_*$ and $(S_P) = (X_P\setminus X_*)$ for a breakup~$X$).

	Note that $\Omega$ is determined by the values of $f$ on $S^{++}$. Let $\bar{\Lambda}$ be a finite subset of $\Z^d$ that contains $\Lambda \cup S^{++}$.
	Let $\bar{\Omega}$ be the support of the marginal of $\Pr_{\Lambda,P_0}$ on $[q]^{\bar{\Lambda}}$.
	We henceforth view $\Omega$ as a subset of $\bar{\Omega}$.
	Denote
	\[ \Omega_0 := \Big\{ f|_{\bar\Lambda \setminus S^+} : f \in \Omega \Big\} \subset [q]^{\bar{\Lambda} \setminus S^+} \quad\text{and}\quad \Omega_1 := \Big\{ f|_{S^+} : f \in \Omega \Big\} \subset [q]^{S^+} .\]
	Let $T \colon \Omega_0 \to 2^{\bar{\Omega}}$ be a map which satisfies $T(f) \cap T(f') = \emptyset$ for distinct $f,f' \in \Omega_0$.
	Recalling~\eqref{eq:finite_volume_measure}, \eqref{eq:prob_outside_Lambda_def1} and~\eqref{eq:prob_outside_Lambda_def2}, we have that the marginal of $\Pr_{\Lambda,P_0}$ on $\bar\Omega$ is uniform. Hence,
	\[ \Pr_{\Lambda,P_0}(\Omega)
	= \frac{|\Omega|}{|\bar{\Omega}|}
	\le \frac{|\Omega_0| \cdot |\Omega_1|}{\sum_{f \in \Omega_0} |T(f)|} \le \frac{|\Omega_1|}{\min_{f \in \Omega_0} |T(f)|} .\]
	
	Before defining $T$, let us bound $|\Omega_1|$.
	To this end, we aim to apply Proposition~\ref{lem:shearer-for-bad-set} with $\cF=\Omega_1$ (and $S$ and $(S_P)$ as here).
	Observe that, since $(\intB S_P)^+$ is in the $P$-pattern on $\Omega$ and since $E$ is determined by $f|_{S^+}$, the collection $\cF$ satisfies the assumption of the proposition and, moreover, $\Pr_{\Lambda,P_0}(f|_{S^+} \in \cdot \mid \Omega)$ is the uniform distribution on $\cF$.
	For $i \in \{0,1\}$, denote $S_i := \bigcup_{P \in \cP_i} S_P \setminus S^+$, where $\phase_0$ and $\phase_1$ were defined in Section~\ref{sec:overview-notation}.
	Then, by Proposition~\ref{lem:shearer-for-bad-set},
	\begin{align*}
	|\Omega_1| &\le (\lfloor\tfrac{q}{2}\rfloor \lceil\tfrac{q}{2}\rceil)^{\frac12|S^+|} \cdot e^{-\frac{1}{128q} \Big(\big|S^f_\unbal\big|+\tfrac{1}{q}\big|S^f_\nondom\big| + \tfrac{1}{d}\big|S^{\cF,f}_\rest\big| \Big)+\tfrac{q}{d}\big|S \setminus S^{\cF}_\unique\big| +e^{-d/65q^2}|S|} \\&\quad \cdot \left(\lfloor \tfrac{q}{2} \rfloor/\lceil \tfrac{q}{2} \rceil\right)^{\frac{1}{4d}(|\partial^\even (S^+,S_0)|-|\partial^\odd (S^+,S_0)|-|\partial^\even (S^+,S_1)|+|\partial^\odd (S^+,S_1)|)} .
	\end{align*}
	Thus, the lemma will follow if we find a map $T$ such that $T(f) \cap T(f') = \emptyset$ for distinct $f,f' \in \Omega_0$ and which also satisfies
	\begin{equation}\label{eq:lower-bound-for-psuedo-breakup}
	\min_{f \in \Omega_0} |T(f)| \ge (\lfloor\tfrac{q}{2}\rfloor \lceil\tfrac{q}{2}\rceil)^{\frac12|S^+|} \cdot \left(\lfloor \tfrac{q}{2} \rfloor/\lceil \tfrac{q}{2} \rceil\right)^{\frac{1}{4d}(|\partial^\even(S^+,S_0)|-|\partial^\even(S^+,S_1)| - |\partial^\odd (S^+,S_0)|+|\partial^\odd (S^+,S_1)|)} .
	\end{equation}
	
	We now turn to the definition of $T$. Fix a unit vector $e \in \Z^d$.
	For $u \in \Z^d$, we denote $u^{\up} := u+e$ and $u^{\down} := u-e$. For a set $U \subset \Z^d$, we also write $U^{\up} := \{ u^{\up} : u \in U\}$ and $U^{\down} := \{ u^{\down} : u \in U\}$.
	For each $P \in \phasedom$, let $\psi_P$ be a permutation of $[q]$ taking $P$ to $P_0=(A_0,B_0)$ if $P \in \phase_0$ or to $(B_0,A_0)$ otherwise (for $P=P_0$, we take $\psi_{P_0}$ to be the identity). Let $\cH$ be the set of all functions $h \colon S_* \to [q]$ which are in the $P_0$-pattern, where
	\[ S_* := (S_0 \cup S_1^{\down})^c .\]
	For $f \in \Omega_0$ and $h \in \cH$, define $\phi_{f,h} \colon \bar\Lambda \to [q]$ by
	\[ \phi_{f,h}(v) := \begin{cases}
	\psi_P(f(v)) &\text{if }v \in S_P \setminus S^+\text{ for }P \in \cP_0 \\
	\psi_P(f(v^{\up})) &\text{if }v \in (S_P \setminus S^+)^{\down}\text{ for }P \in \cP_1 \\
	h(v) &\text{if }v \in S_*
	\end{cases} .\]
	Note that $\phi_{f,h}$ is well defined, since the assumption that $\intB S_P \subset \extB S$ for all $P$ implies that
	\begin{equation}\label{eq:S_P-distance-3}
	\dist(S_P \setminus S^+,S_Q \setminus S^+) \ge 3 \qquad\text{for distinct }P\text{ and }Q,
	\end{equation}
	so that, in particular, $\{ S_0, S_1^{\down}, S_* \}$ is a partition of $\Z^d$.
	
	Let us check that $\phi:=\phi_{f,h}$ is a proper coloring.
	In light of~\eqref{eq:S_P-distance-3}, it suffices to show that $\intB S_0$, $\intB S_1^{\down}$ and $S_*$ are in the $P_0$-pattern in $\phi$. It is immediate from the definition that $S_*$ is in the $P_0$-pattern in $\phi$.
	If $w \in \intB S_0$ then $w \in \intB (S_P \setminus S^+) \subset (\intB S_P)^+$ for some $P \in \phase_0$. By the assumption of the lemma, $w$ is in the $P$-pattern in $f$, and thus, by the definition of $\psi_P$, $w$ is in the $P_0$-pattern in $\psi_P \circ f$ and hence also in $\phi$.
	Similarly, if $w \in \intB S_1^{\down}$ then $w^{\up} \in \intB (S_P \setminus S^+) \subset (\intB S_P)^+$ for some $P \in \phase_1$, so that $w^{\up}$ is in the $P$-pattern in $f$, and thus, $w^{\up}$ is in the $(B_0,A_0)$-pattern in $\psi_P \circ f$ so that $w$ is in the $P_0$-pattern in $\phi$.
		
	Let us check that $\phi \in \bar{\Omega}$. By~\eqref{eq:finite_volume_measure}, \eqref{eq:prob_outside_Lambda_def1} and~\eqref{eq:prob_outside_Lambda_def2}, we must check that $\bar\Lambda \setminus \Int(\Lambda) = \bar\Lambda \cap (\Lambda^c)^+$ is in the $P_0$-pattern in $\phi$.
	Let $v \in \bar\Lambda \setminus \Int(\Lambda)$ and recall that, by assumption, $(\Lambda^c)^+ \subset S \cup S_{P_0}$.
	Since $S \subset S_*$ and $S_*$ is in the $P_0$-pattern in $\phi$, we may assume that $v \in (S \cup S_{P_0}) \setminus S_* \subset S_{P_0} \setminus S^+$, in which case, $\phi(v)=f(v)$ and it is clear that $v$ is in the $P_0$-pattern in $\phi$.
		
	Finally, define
	\[ T(f) := \{ \phi_{f,h} : h \in \cH \} .\]
	To see that the desired property that $T(f) \cap T(f') = \emptyset$ for distinct $f,f' \in \Omega_0$ holds, we now show that the mapping $(f,h) \mapsto \phi_{f,h}$ is injective on $\Omega_0 \times \cH$. To this end, we show how to recover $(f,h)$ from a given $g$ in the image of this mapping. Indeed, it is straightforward to check that
	\[ f(v) =
	\begin{cases}
	\psi_P^{-1}(g(v)) &\text{if }v \in S_P \setminus S^+\text{ for }P \in \cP_0 \\
	\psi_P^{-1}(g(v^{\down})) &\text{if }v \in S_P \setminus S^+\text{ for }P \in \cP_1
	\end{cases}
	\qquad\text{and}\qquad
	h(v) = g(v)\text{ for }v \in S_* .\]
	
	It remains to check that~\eqref{eq:lower-bound-for-psuedo-breakup} holds.
	By injectivity, we have
	\[ |T(f)| = |\cH| \qquad\text{for all }f \in \Omega_0 .\]
	Since the definition of $\cH$ immediately implies that
	\[ |\cH| = \lfloor \tfrac{q}{2} \rfloor^{|S^\even_*|} \cdot \lceil \tfrac{q}{2} \rceil^{|S^\odd_*|} ,\]
	concluding~\eqref{eq:lower-bound-for-psuedo-breakup} is essentially just a computation. To see this, using the fact (which we prove below) that, for any finite set $U \subset \Z^d$,
	\begin{equation}\label{eq:even-odd-diff}
	|U^\even| - |U^\odd| = \tfrac{1}{2d} (|\partial^\even U| - |\partial^\odd U|) ,
	\end{equation}
	and writing $|S_*^\even| = \tfrac12 (|S_*| + |S_*^\even|-|S_*^\odd|)$, and similarly for $|S_*^\odd|$, we have
	\[ |\cH| = (\lfloor\tfrac{q}{2}\rfloor \lceil\tfrac{q}{2}\rceil)^{\frac12|S_*|} \cdot \left(\lfloor \tfrac{q}{2} \rfloor/\lceil \tfrac{q}{2} \rceil\right)^{\frac{1}{4d}(|\partial^\even S_*|-|\partial^\odd S_*|)} .\]
Noting that $|S_*|=|S^+|$, it thus suffices to show that
\begin{align*}
|\partial^\even S_*| &= |\partial^\even(S^+,S_0)|+|\partial^\odd (S^+,S_1)|,\\
|\partial^\odd S_*| &= |\partial^\odd (S^+,S_0)|+|\partial^\even(S^+,S_1)|.
\end{align*}
Since $\partial S_* = \partial(S^+,S_0) \cup \partial((S^+)^{\down},S_1^{\down})$, this easily follows.	

It remains to prove~\eqref{eq:even-odd-diff}. To see this, first observe that $u \mapsto u^{\down}$ is a bijection between $U^\even \cap U^{\up}$ and $U^\odd \cap U^{\down}$, so that
\[ |U^\even| - |U^\odd| = |(U \setminus U^{\up})^\even|-|(U \setminus U^{\down})^\odd| .\]
As this equality holds for any direction $\up$, summing it up over the $2d$ possible choices yields~\eqref{eq:even-odd-diff}.
\end{proof}

The proof of the Proposition~\ref{lem:shearer-for-bad-set} relies on two lemmas.
The first lemma, whose proof is based on Shearer's inequality, provides a bound on the size of a collection of colorings $f$, which is conveniently factorized into ``local terms'' involving the values of $f$ on a vertex and its neighbors.

\begin{lemma}\label{lem:shearer-extended}
	Let $S \subset \Z^d$ be finite and even and let $\{ A_u \}_{u \in \intB S}$ be a collection of subsets of $[q]$. Let $\cF \subset [q]^S$ be a set of proper colorings such that $f(u) \in A_u$ for every $f \in \cF$ and $u \in \intB S$. Let $f$ be an element of $\cF$ chosen uniformly at random. For each odd vertex $v \in S$, let $X_v$ be a random variable which is measurable with respect to $f|_{N(v)}$.
	Then
	\[ \log |\cF| \le \sum_{v \in S^\odd} \Big[\tfrac{1}{2d} \Ent(X_v)+ \tfrac{1}{2d} \E\log |\Psi_v| + \E\log |I_v| \Big] + \tfrac{1}{2d} \sum_{u \in \intB S} |\partial u \cap \partial S| \log |A_u| ,\]
	where $\Psi_v$ and $I_v$ are the supports of $f|_{N(v)}$ and $f(v)$ given $X_v$, respectively.
\end{lemma}

\begin{proof}	
	Since $f$ is uniformly chosen from $\cF$, we have $\Ent(f)=\log |\cF|$. Hence, our goal is to bound $\Ent(f)$.
	We make use of~\eqref{eq:entropy-chain-rule}-\eqref{eq:entropy-subadditivity} throughout the proof.
	We begin by writing
	\[ \Ent(f) = \Ent(f^\even) + \Ent(f^\odd \mid f^\even) .\]
	By the sub-additivity of entropy, we have
	\[ \Ent(f^\odd \mid f^{\even}) \le \sum_{v \in S^\odd} \Ent\big(f(v) \mid f|_{N(v)}\big) .\]
	We use Shearer's inequality to bound $\Ent(f^\even)$.
	Namely, Lemma~\ref{lem:shearer} applied with the random variables $(Z_i) = (f(v))_{v \in S^\even}$, the collection $\cI=\{N(v)\}_{v \in S^\odd} \cup \{N(v) \cap S\}_{v \in \extB S}$ and $k=2d$, yields
	\[ \Ent(f^\even) \le \tfrac{1}{2d} \sum_{v \in S^\odd} \Ent\big(f|_{N(v)}\big) + \tfrac{1}{2d} \sum_{v \in \extB S} \Ent\big(f|_{N(v) \cap S}\big) .\]
	Note that, by the assumption on $\cF$,
	\[ \sum_{v \in \extB S} \Ent\big(f|_{N(v) \cap S}\big) \le \sum_{v \in \extB S} \sum_{u \in N(v) \cap S} \Ent(f(u)) = \sum_{u \in \intB S} |\partial u \cap \partial S| \cdot \Ent(f(u)) .\]
	Thus, the lemma will follow once we show that
	\[ \tfrac{1}{2d} \cdot \Ent\big(f|_{N(v)}\big) + \Ent\big(f(v) \mid f|_{N(v)}\big) \le \tfrac{1}{2d} \Ent(X_v)+ \tfrac{1}{2d} \E\log |\Psi_v| + \E \log |I_v| .\]
	Indeed,
	\[ \Ent\big(f|_{N(v)}\big) \le \Ent(X_v) + \Ent\big(f|_{N(v)} \mid X_v\big) \le \Ent(X_v) + \E\log|\Psi_v| ,\]
	and
	\[ \Ent\big(f(v) \mid f|_{N(v)}\big) \le \Ent\big(f(v) \mid X_v\big) \le \E\log|I_v| . \qedhere \]
\end{proof}

Besides factorizing the bound on $|\cF|$ over the odd vertices in $S$, Lemma~\ref{lem:shearer-extended} allows exposing some information about $f|_{N(v)}$ which can then be used to bound $|\Psi_v| \cdot |I_v|^{2d}$. One could theoretically expose $f|_{N(v)}$ completely (i.e., by taking $X_v$ to equal $f|_{N(v)}$ above), but this would increase the entropy of $X_v$, making it harder to bound $\Ent(X_v)$ effectively. One would therefore like to expose as little information as possible, which still suffices to obtain good bounds on $|\Psi_v| \cdot |I_v|^{2d}$.

Recalling the notions of non-dominant vertex, restricted edge and unbalanced neighborhood introduced in Section~\ref{sec:Shearer_overview}, we aim to expose just enough information to allow determining the occurrence of these. We now proceed to define this information, which we call the \emph{type} of $f|_{N(v)}$.
Given a function $\psi \colon [2d] \to [q]$, which is later identified with $f|_{N(v)}$, let $\psi_\unbal$ be the indicator of whether there exists $i \in \psi([2d])$ such that $|\psi^{-1}(i)| \le d/q$.
The \emph{type} of $\psi$ is then defined to be $(\psi([2d]),\psi_\unbal)$.

In the proof of Proposition~\ref{lem:shearer-for-bad-set}, we will use Lemma~\ref{lem:shearer-extended} with the random variable $X_v$ taken to be the type of $f|_{N(v)}$.
To make use of the inequality given in Lemma~\ref{lem:shearer-extended}, we will need to accompany it with suitable bounds on $|\Psi| \cdot |I|^{2d}$, where $\Psi$ is a collection of functions of type $(J,z)$ and $I \subset [q]$ is disjoint from $J$. The next lemma provides such bounds. For $\Psi$ consisting of functions of type $(J,z)$, we say that $j \in [2d]$ is a \emph{semi-restricted index} in $\Psi$ if $\{ \psi(j) : \psi \in \Psi\} \neq J$. We note that restricted edges (in the sense of the definition in Section~\ref{sec:Shearer_overview}) correspond to either semi-restricted indices or to the condition that $I \cup J \neq [q]$.

\begin{lemma}\label{lem:Z-bounds}
	Let $\Psi$ be a collection of functions of type $(J,z)$ and let $I \subset J^c$. Suppose that there are $k$ semi-restricted indices in $\Psi$. Then
	\[ |\Psi| \cdot |I|^{2d} \le (\lfloor\tfrac{q}{2}\rfloor \lceil\tfrac{q}{2}\rceil)^{2d} \cdot
	\begin{cases}
	e^{-k/q} &\text{always} \\
	e^{-4d/q^2} &\text{if }|J| \notin \{\lfloor\tfrac{q}{2}\rfloor,\lceil\tfrac{q}{2}\rceil\}\\
	e^{-d/4q} &\text{if }I \cup J \neq [q] \text{ or }z=1
	\end{cases}
	.\]
\end{lemma}

Let us explain the terms in the above bound.
Observe that $|\Psi| \cdot |I|^{2d}$ is the number of proper colorings $\varphi$ of $K_{2d,2d}$ whose restriction to the left side of $K_{2d,2d}$ belongs to $\Psi$ and whose restriction to the right side belongs to $I^{[2d]}$.
The first term, $(\lfloor\tfrac{q}{2}\rfloor \lceil\tfrac{q}{2}\rceil)^{2d}$, comes from considering those $\varphi$ whose left and right sides takes values in $A$ and $B$, respectively, for some dominant pattern $(A,B)$.
In the second term, the first case reflects the reduction in the number of choices for $\varphi$ on the left side caused by the existence of semi-restricted indices.
The second case corresponds to a non-dominant vertex.
Finally, the third case corresponds to either an unbalanced neighborhood or a partial restriction on the values of $\varphi$ on the right side.

\begin{proof}
	For the first inequality in the lemma, we note that, by definition, $\{ \psi(j) : \psi \in \Psi \} \subsetneq J$ for any $j \in [2d]$ which is semi-restricted in $\Psi$. Thus,
	\[ |\Psi| \le \prod_{j \in [2d]} |\{ \psi(j) : \psi \in \Psi \}| \le |J|^{2d} \cdot (1-\tfrac{1}{|J|})^k \le |J|^{2d} \cdot e^{-k/q} ,\]
	so that
	\[ |\Psi| \cdot |I|^{2d} \le (|J| \cdot |J^c|)^{2d} \cdot e^{-k/q} \le (\lfloor\tfrac{q}{2}\rfloor \lceil\tfrac{q}{2}\rceil)^{2d} \cdot e^{-k/q} .\]
	For the second inequality in the lemma, observe that if $|J| \notin \{\lfloor\tfrac{q}{2}\rfloor,\lceil\tfrac{q}{2}\rceil\}$, then
	\[ |\Psi|^{1/2d} \cdot |I| \le |J| \cdot |J^c| \le (\lfloor \tfrac{q}{2} \rfloor - 1) \cdot (\lceil \tfrac{q}{2} \rceil + 1) \le \lfloor\tfrac{q}{2}\rfloor \lceil\tfrac{q}{2}\rceil \cdot e^{-4/q^2} .\]
	For the third inequality in the lemma, suppose first that $I \cup J \neq [q]$ and note that
	\[ |\Psi|^{1/2d} \cdot |I| \le |J| \cdot |I| \le \lfloor \tfrac{q}{2} \rfloor \cdot (\lceil \tfrac{q}{2} \rceil - 1) \le \lfloor\tfrac{q}{2}\rfloor \lceil\tfrac{q}{2}\rceil \cdot e^{-1/q} .\]
	Suppose now that $z=1$ and note that, by a Chernoff bound and~\eqref{eq:dim-assump},
	\[ |\Psi| \cdot |J|^{-2d} \le |J| \cdot \Pr\Big(\text{Bin}\big(2d,\tfrac{1}{|J|}\big) \le \tfrac{d}{q}\Big) \le q e^{-\tfrac18\tfrac{2d}{q-1}} \le e^{-d/4q} , \]
	so that
	\[ |\Psi| \cdot |I|^{2d} \le |\Psi| \cdot |J^c|^{2d} \le (\lfloor\tfrac{q}{2}\rfloor \lceil\tfrac{q}{2}\rceil)^{2d} \cdot |\Psi| \cdot |J|^{-2d} \le (\lfloor\tfrac{q}{2}\rfloor \lceil\tfrac{q}{2}\rceil)^{2d-d/4q} . \qedhere \]
\end{proof}

Let us now give the proof of the main proposition.

\begin{proof}[Proof of Proposition~\ref{lem:shearer-for-bad-set}]
	We prove something slightly stronger than the inequality stated in the lemma. Namely, we show that
	\begin{equation}\label{eq:shearer-ub}
	\begin{aligned}
	|\cF| \le (\lfloor\tfrac{q}{2}\rfloor \lceil\tfrac{q}{2}\rceil)^{|S^\odd|} & \cdot e^{-\frac{1}{64q} \E\big(\big|S^{f,\odd}_\unbal\big| + \frac{1}{q} \big|S^{f,\odd}_\nondom\big| +\frac{1}{d}\big|S^{\cF,f,\odd}_\rest\big|\big) + \frac{q}{d}|\Int(S) \setminus S^\cF_\unique| + e^{-d/65q^2}|\Int(S)|} \\ &\cdot \prod_P (\lambda_P)^{\frac{1}{2d}(|\partial^\even (S,S_P)|-|\partial^\odd (S,S_P)|)} ,
	\end{aligned}
	\end{equation}
	where $S^{f,\odd}_\unbal := (S^f_\unbal)^\odd$, $S^{f,\odd}_\nondom = (S^f_\nondom)^\odd$ and $S^{\cF,f,\odd}_\rest$ is the set of restricted edges $(v,u)$ with $v \in (\Int(S))^\odd$, and
	\[ \lambda_P := \begin{cases}
	 \lfloor \tfrac{q}{2} \rfloor &\text{if }P \in \phase_0\\
	 \lceil\tfrac{q}{2} \rceil &\text{if }P \in \phase_1
	\end{cases} .\]
	Indeed, the lemma then follows by taking the geometric average of the above bound and its symmetric version in which the roles of odd and even are exchanged.
	
	In proving~\eqref{eq:shearer-ub}, instead of working directly with $S$, it is convenient to work with its even expansion, defined as
	\[ S' := S \cup (\extB S)^\even = S^+ \setminus (\extB S)^\odd .\]
	Note that $S \subset S' \subset S^+$ and $S^\odd = (S')^\odd$. Let $\cF'$ be the set of functions $f' \in [q]^{S'}$ such that $f'|_S \in \cF$ and for which $S_P$ is in the $P$-pattern for every $P$.
	Observe that if one samples an element $f' \in \cF'$ uniformly at random, then $f'|_S$ has the same distribution as $f$, and the random variables $\{f'(u)\}_{u \in S' \setminus S}$ are independent and uniformly distributed on $A$, where $P=(A,B)$ is the unique dominant pattern such that $u \in S_P$. It follows that
	\[ |\cF'| = |\cF| \cdot \prod_P (\lambda_P)^{|S' \cap S_P|} .\]
	Thus, noting that $|S' \cap S_P|=|\partial^\odd(S,S_P)|+|\partial^\odd(S^c,S_P)|$, we see that~\eqref{eq:shearer-ub} is equivalent to
	\begin{equation}\label{eq:shearer-ub2}
	\begin{aligned}
	|\cF'| \le (\lfloor\tfrac{q}{2}\rfloor \lceil\tfrac{q}{2}\rceil)^{|S^\odd|} &\cdot e^{-\frac{1}{64q} \E\big(\big|S^{f,\odd}_\unbal\big|+ \frac{1}{q} \big|S^{f,\odd}_\nondom\big| + \frac{1}{d}\big|S^{\cF,f,\odd}_\rest\big| \big) +\frac{q}{d}|\Int(S) \setminus S^\cF_\unique| +e^{-d/65q^2}|\Int(S)|} \\ &\cdot \prod_P (\lambda_P)^{\frac{1}{2d}(|\partial^\even (S,S_P)| + |\partial^\odd (S^c,S_P)|)} .
	\end{aligned}
	\end{equation}
	We also note at this point that $S^{\cF,f,\odd}_\rest = S^{\cF',f,\odd}_\rest$ and $S^{\cF}_\unique = S^{\cF'}_\unique$.

	We now aim to apply Lemma~\ref{lem:shearer-extended} with $S'$ and $\cF'$. For $u \in \intB S'$, define
	\[ A_u := \begin{cases} A &\text{if }u \in \extB S \cap S_{(A,B)}\\ \bigcap \{ A : (A,B) \in \phasedom : u \in N(S_{(A,B)}) \} &\text{if }u \in \intB S \end{cases} .\]
	Note that, by the assumption on $\cF$ and by the definition of $\cF'$, we have $\phi(u) \in A_u$ for all $\phi \in \cF'$ and all $u \in \intB S'$.
	For an odd vertex $v \in S$, define
	\[ X_v := \begin{cases} \text{the type of }f'|_{N(v)} &\text{if }v \in \Int(S)\\0&\text{if }v \in \intB S \end{cases}.\]
	Then, by Lemma~\ref{lem:shearer-extended},
	\[ \log |\cF'| \le \sum_{v \in S^\odd} \Big[\tfrac{1}{2d} \Ent(X_v)+ \tfrac{1}{2d} \E\log |\Psi_v| + \E\log |I_v| \Big] + \tfrac{1}{2d} \sum_{u \in \intB S'} |\partial u \cap \partial S'| \log |A_u| ,\]
	where $\Psi_v$ and $I_v$ are the supports of $f'|_{N(v)}$ and $f'(v)$ given $X_v$, respectively. We stress that the expectations above are with respect to $f'$, but we also remind that $f'|_S$ equals $f$ in distribution so that these expectations are the same when taken with respect to~$f$.

	We first show that
	\[ \sum_{u \in \intB S'} |\partial u \cap \partial S'| \log |A_u| \le \sum_{P \in \phase_0} (|\partial^\even (S,S_P)| + |\partial^\odd (S^c,S_P)|) \cdot \log \lambda_P .\]
	Since $\{ \partial u \cap \partial S'\}_{u \in \intB S'}$ and $\{ \partial^\even (S,S_P), \partial^\odd (S^c,S_P) \}_P$ are two partitions of $\partial S'$, it suffices to show an inequality for each edge separately, namely, that $|A_u| \le \lambda_P$ for any $u \in \intB S'$ and $P=(A,B)$ such that $\partial u \cap (\partial^\even (S,S_P) \cup \partial^\odd (S^c,S_P)) \neq \emptyset$.
	To this end, suppose that $\{u,w\} \in \partial^\even (S,S_P) \cup \partial^\odd (S^c,S_P)$ for some $w$ and note that $u$ is even. If $u \notin S$ then $u \in S_P$ so that $A_u = A$ and $|A_u| = \lambda_P$. If $u \in S$ then $w \in S_P$ so that $A_u \subset A$ and $|A_u| \le \lambda_P$.
	
	Thus, to obtain~\eqref{eq:shearer-ub2}, it suffices to show that, for any $v \in S^\odd$,
	\begin{equation}\label{eq:shearer-ub-3}
	\frac{\Ent(X_v)+ \E\log (|\Psi_v| \cdot |I_v|^{2d})}{2d} \le \log(\lfloor\tfrac{q}{2}\rfloor \lceil\tfrac{q}{2}\rceil) + \begin{cases} -\frac{1}{64q} p_v + \frac{q}{d}\1_{v \notin S^\cF_\unique} + e^{-d/65q^2} &\text{if }v \in \Int(S)\\0&\text{if }v \in \intB S\end{cases} ,
	\end{equation}
	where
	\[ p_v := \Pr\big(v \in S^f_\unbal\big) + \tfrac{1}{q} \cdot \Pr\big(v \in S^f_\nondom\big) + \tfrac{1}{d} \cdot \E\big|\vec\partial v \cap S^{\cF,f}_\rest\big| .\]
	Suppose first that $v \in \intB S$. By the assumption on $\cF$ and by definition of $\cF'$, we have that $\Psi_v \subset A^{N(v)}$ and $I_v \subset B$, where $P=(A,B)$ is a dominant pattern such that $v \in S_P^+$. Thus,
	\[ \frac{\Ent(X_v)+ \E\log (|\Psi_v| \cdot |I_v|^{2d})}{2d} = \tfrac{1}{2d} \log (|\Psi_v| \cdot |I_v|^{2d}) \le \log(|A|\cdot |B|) = \log(\lfloor\tfrac{q}{2}\rfloor \lceil\tfrac{q}{2}\rceil) .\]
	
	Suppose now that $v \in \Int(S)$.
	The desired inequality in~\eqref{eq:shearer-ub-3} will follow if we show that
	\begin{equation}\label{eq:shearer-ub-Z}
	\tfrac{1}{2d} \E\log (|\Psi_v| \cdot |I_v|^{2d}) \le \log (\lfloor\tfrac{q}{2}\rfloor \lceil\tfrac{q}{2}\rceil) -\tfrac{1}{32q} \cdot p_v
	\end{equation}
	and
	\begin{equation}\label{eq:shearer-ub-entropy}
	\tfrac{1}{2d} \Ent(X_v) \le \tfrac{1}{64q} \cdot p_v + \tfrac{q}{d}\1_{v \notin S^\cF_\unique} +e^{-d/65q^2} .
	\end{equation}

	We begin by showing~\eqref{eq:shearer-ub-Z}.
	Consider the random set
	\[ \cR_v := \big\{ u \sim v : u\text{ is a semi-restricted index in }\Psi_v \} .\]
	Denote $X_v=(J_v,Z_v)$. By Lemma~\ref{lem:Z-bounds},
	\[ \tfrac{1}{2d} \E\log (|\Psi_v| \cdot |I_v|^{2d}) \le \log (\lfloor\tfrac{q}{2}\rfloor \lceil\tfrac{q}{2}\rceil) -
	\begin{cases}
	\tfrac{1}{2qd} |\cR_v| &\text{always}\\
	\tfrac{2}{q^2} &\text{if }|J_v| \notin \{\lfloor\tfrac{q}{2}\rfloor,\lceil\tfrac{q}{2}\rceil\}\\
	\tfrac{1}{8q} &\text{if }I_v \cup J_v \neq [q] \text{ or }Z_v=1
	\end{cases}
	\]
	Then, recalling the definition of $p_v$, \eqref{eq:shearer-ub-Z} will follow if we show that
\[ \1_{\{v \in S^f_\unbal\}} + \tfrac{1}{q} \1_{\{v \in S^f_\nondom\}} + \tfrac{1}{d} \big|\vec\partial v \cap S^{\cF,f}_\rest\big| \le \begin{cases}
	\tfrac{16}{d} |\cR_v| &\text{if }|J_v| \in \{\lfloor\tfrac{q}{2}\rfloor,\lceil\tfrac{q}{2}\rceil\} \text{ and }Z_v=0\\
	\tfrac{32}{q} + \tfrac{8}{d}|\cR_v| &\text{if }|J_v| \notin \{\lfloor\tfrac{q}{2}\rfloor,\lceil\tfrac{q}{2}\rceil\} \text{ and }Z_v=0\\
	4 &\text{if }I_v \cup J_v \neq [q] \text{ or }Z_v=1
	\end{cases} .\]
	Indeed, if $I_v \cup J_v \neq [q]$ or $Z_v=1$ then the inequality is clear. Otherwise, $I_v \cup J_v = [q]$ and $Z_v=0$ so that $\vec\partial v \cap S^{\cF,f}_\rest = \{v\} \times \cR_v$ and $v \notin S^f_\unbal$. Since $|J_v| \in \{\lfloor\tfrac{q}{2}\rfloor,\lceil\tfrac{q}{2}\rceil\}$ if and only if $v \notin S^f_\nondom$, the inequality follows.
	
	It remains to show~\eqref{eq:shearer-ub-entropy}.
	By~\eqref{eq:entropy-support}, we always have the trivial bound
	\[ \Ent(X_v) \le \log |\supp(X_v)| \le \log 2^{q+1} \le 2q.\]
	Thus, it suffices to show that, for any $v \in S^\cF_\unique$,
	\[ \Ent(X_v) \le \tfrac{d}{32q} \cdot p_v + 2de^{-d/65q^2} .\]
	Fix $v \in S^\cF_\unique$ and denote $p:=p_v$. When $p \ge 1/2q$, the above bound follows from the trivial bound on $\Ent(X_v)$ using~\eqref{eq:dim-assump}. Thus, we may assume that $p<1/2q$.
	By the definition of unique pattern, there exists some $J$ for which $X_v \neq (J,0)$ implies that $v \in S^f_\nondom$ or $\vec\partial v \subset S^{\cF,f}_\rest$.
	In particular, $\Pr(X_v \neq (J,0)) \le pq < 1/2$.
	Hence, using the chain rule for entropy~\eqref{eq:entropy-chain-rule}, we obtain
	\begin{align*}
	\Ent(X_v) &= \Ent(\1_{\{X_v=(J,0)\}}) + \Ent(X_v \mid \1_{\{X_v=(J,0)\}}) \\&\le pq\log \tfrac{2^{q+1}}{pq} + (1-pq)\log\tfrac{1}{1-pq} \le 2pq\log \tfrac{2^{q+1}}{pq} .
	\end{align*}
	Thus,
	\[ \Ent(X_v) - \tfrac{d}{32q} \cdot p \le 2pq \log \left( \tfrac{2^{q+1}}{pq} \cdot e^{-d/64q^2} \right) \le \tfrac {2^{q+2}}e \cdot e^{-d/64q^2} \le e^{-d/65q^2} ,\]
	where we used the fact that $x \log(a/x) \le a/e$ for $0<x<1$ in the second inequality (with $x=pq$) and we used~\eqref{eq:dim-assump} in the last inequality.
\end{proof}

\section{Approximations}
\label{sec:approx}

In this section, we prove Proposition~\ref{prop:family-of-odd-approx}.
That is, we show that there exists a small family of approximations which contains an approximation of every atlas in $\breakups_{L,M,N}$ that is seen from a given set.
The construction of the family of approximations is done in two steps, as we now explain.

Say that a set $W$ \emph{separates} an atlas $X$ if every edge in $\bigcup_P \partial X_P$ has an endpoint in~$W$, and that it \emph{tightly separates} $X$ if also $W \subset \bigcup_P (\intextB X_P)^{+2}$.
The first step is to construct a small family of small sets which contains a separating set of every atlas in $\breakups_{L,M,N}$ that is seen from a given set.

\begin{lemma}\label{lem:family-of-separating-sets}
	For any integers $d \ge 2$ and $L,M,N \ge 0$ and any finite set $V \subset \Z^d$, there exists a family $\cW$ of subsets of $\Z^d$, each of size at most $CL (\log d)/\sqrt{d}$,
	such that
	\[ |\cW| \le 2^{|V|} \cdot \exp\Big(\tfrac{CL \log^2 d}{d^{3/2}} + \tfrac{C(M+N) \log^2 d}{d} \Big) \]
	and any atlas $X \in \breakups_{L,M,N}$ seen from $V$ is tightly separated by some set in $\cW$.
\end{lemma}

Most of the arguments in this section are applied separately to the two collections $(X_P)_{P \in \phase_0}$ and $(X_P)_{P \in \phase_1}$, which consist of even and odd sets, respectively. The desired approximation defined in Section~\ref{sec:approx-overview} is then constructed by combining the two independent pieces. For simplicity of writing, we fix the parity of the sets we work with here to be odd, even sets being completely analogous.

The definition of an atlas does not require any relation between $X_P$ for different $P$. In particular, the set of $P$ for which a given vertex belongs to $X_P$ could be any subset of the dominant patterns. Since there are doubly-exponentially in $q$ many such subsets, this would not lead to the correct dependency on $q$. In light of this, we require an additional property of atlases, satisfied by any breakup, namely, \eqref{eq:possible-patterns-a-vertex-can-be-in}.
In order to keep this section as independent as possible, we introduce some abstract definitions.

Let $S=(S_i)_i$ be a collection of regular odd sets (we do not explicitly specify the index set as it has no significance in what follows). A \emph{rule} is a family $\cQ$ of subsets of indices. We say that a rule $\cQ$ has \emph{rank at most $q$} if $|\cQ| \le 2^q$. We say that $S$ is an odd \emph{$\cQ$-collection} if it obeys the rule $\cQ$ in the following sense:
\[ \{ i : v \in S_i \} \in \cQ \qquad\text{for any even vertex }v .\]
An \emph{approximation} of $S$ is a collection $A=((A_i)_i,A_*)$ such that $A_i \subset S_i \subset A_i \cup A_*$ and $A_i$ is odd for all $i$ and such that $\Even \cap A_* \subset N_d(\bigcup_i A_i)$.
We say that $A$ is \emph{controlled by} a set $W$ if $|A_*| \le C|W|$ and $A_* \subset W^+$, and that $W$ \emph{separates} $S$ if every edge in $\bigcup_i \partial S_i$ has an endpoint in $W$.

\begin{lemma}\label{lem:family-of-approx-from-separating-set}
For any integers $d \ge 2$ and $q \ge 1$, any rule $\cQ$ of rank at most $q$ and any finite set $W \subset \Z^d$, there exists a family $\cA$ of approximations, each of which is controlled by~$W$, such that
	\[ |\cA| \le \exp\Big(\tfrac{C|W|(q+\log d)}{d} \Big) \]
	and any odd $\cQ$-collection which is separated by $W$ is approximated by some element in $\cA$.
\end{lemma}

Lemma~\ref{lem:family-of-separating-sets} and Lemma~\ref{lem:family-of-approx-from-separating-set} are proved in Sections~\ref{sec:separating-sets} and~\ref{sec:Q-approx} below.

\begin{proof}[Proof of Proposition~\ref{prop:family-of-odd-approx}]
	Applying Lemma~\ref{lem:family-of-separating-sets}, we obtain a family $\cW$ of subsets of $\Z^d$, each of size at most $r := CL (\log d) / \sqrt{d}$, such that every $X \in \breakups_{L,M,N}$ seen from $V$ is tightly separated by some set in $\cW$.
	By~\eqref{eq:def-atlas} and~\eqref{eq:possible-patterns-a-vertex-can-be-in}, there exists a rule $\cQ$ of rank at most $q$ such that $(X_P)_{P \in \phase_1}$ is an odd $\cQ$-collection for any $X \in \breakups$.
	Now, for each $W \in \cW$, we apply Lemma~\ref{lem:family-of-approx-from-separating-set} to obtain a family $\cA^1_W$ of approximations, each of which is controlled by $W$, such that $|\cA^1_W| \le \exp( Cr(q+\log d)/d)$ and satisfying that any odd $\cQ$-collection which is separated by $W$ is approximated by some element in $\cA^1_W$.
	Reversing the roles of even and odd, we also obtain a family $\cA^0_W$ in a similar manner.
	Finally, define $\cA := \bigcup_{W \in \cW} \bigcup_{A^0 \in \cA^0_W,\, A^1 \in \cA^1_W} \phi(A^0,A^1)$, where
	\[ \phi(A^0,A^1):=\big((A^0_P)_{P \in \phase_0} \cup (A^1_P)_{P \in \phase_1}, (\Odd \cap A^0_*) \cup (\Even \cap A^1_*), A^0_* \cup A^1_*\big).\]
	It is straightforward to verify that $\cA$ satisfies the requirements of the lemma.
\end{proof}

\subsection{Constructing separating sets}
\label{sec:separating-sets}

This section is devoted to the proof of Lemma~\ref{lem:family-of-separating-sets}.
That is, we construct a small family of sets, each of size at most $CL(\log d) / \sqrt{d}$, which contains a tightly separating set of every atlas $X \in \breakups_{L,M,N}$ seen from $V$.
We begin by showing that for every collection $S=(S_i)_i$ of regular odd sets, there exists a small set $U$ such that $N(U)$ tightly separates $S$.
For such a collection, denote $\partial S := \bigcup_i \partial S_i$ and $\intextB S := \bigcup_i \intextB S_i$.

\begin{lemma}\label{lem:existence-of-U}
	Let $S=(S_i)_i$ be a collection of regular odd sets.
	Then there exists $U \subset (\intextB S)^+$ of size at most $|\partial S| \cdot Cd^{-3/2}\log d$ such that $N(U)$ separates $S$.
\end{lemma}

The proof of Lemma~\ref{lem:existence-of-U} is given at the end of the section. Before proving Lemma~\ref{lem:family-of-separating-sets}, we require another lemma.

\begin{lemma}\label{lem:number-of-disconnecting-sets}
For any $n \ge 1$, the number of sets $U \subset \Z^d$ of size at most $n$ such that $U^{+10}$ is connected and disconnects the origin from infinity is at most $\exp(Cn \log d)$.
\end{lemma}
\begin{proof}
Let $\cU$ be the collection of all sets $U \subset \Z^d$ of size at most $n$ which are connected in $(\Z^d)^{\otimes 21}$ and intersect $V^{+10}$, where $V := \{ \zero + ie_1 : 0 \le i < n(2d+1)^{10} \}$.
Since the maximum degree of $(\Z^d)^{\otimes 21}$ is at most $(2d)^{21}$ and since $|V^{+10}| \le n (2d+1)^{20}$, Lemma~\ref{lem:number-of-connected-graphs} implies that
\[ |\cU| \le |V^{+10}| \cdot (e(2d)^{21})^n \le e^{Cn \log d} ,\]
Thus, the lemma will follow if we show that $\cU$ contains every set $U$ satisfying the assumption of the lemma. Indeed, $U$ intersects $V^{+10}$, or equivalently, $U^{+10}$ intersects $V$, since $|U^{+10}| \le n (2d+1)^{10}$ and $U^{+10}$ disconnects the origin from infinity. Finally, the fact that $U^{+10}$ is connected implies that $U$ is connected in $(\Z^d)^{\otimes 21}$.
\end{proof}

\begin{proof}[Proof of Lemma~\ref{lem:family-of-separating-sets}]
	Let $L,M,N \ge 0$ be integers and let $V \subset \Z^d$ be finite.
	Let $\cU$ be the collection of all subsets $U$ of $\Z^d$ of size at most
	\[ r := CL d^{-3/2}\log d + C(M+N) d^{-1} \log d \]
	such that every connected component of $U^{+7}$ disconnects some vertex $v \in V$ from infinity. Define
	\[ \cW := \big\{ N(U') : U \in \cU,~ U' \subset U,~|U'| \le CL d^{-3/2} \log d \big\} .\]
	Let us show that $\cW$ satisfies the requirements of the lemma. Note first that every $W \in \cW$ has $|W| \le CL d^{-1/2} \log d$. Next, to bound the size of $\cW$, observe that $|\cW| \le |\cU| \cdot 2^r$. Consider a set $U \in \cU$ and let $\{U_l\}_{l=1}^n$ be the connected components of $U^{+7}$ and denote $r_l := |U \cap U_l|$. For each~$l$, choose a vertex $v_l \in V$ such that $U_l$ disconnects $v_l$ from infinity. There are at most $2^{|V|}$ choices for $\{ v_l \}_{l=1}^n$, and given such a choice, there are then at most $\binom{r+n}{n} \le 4^r$ choices for $(v_l,r_l)_l$.
	Thus, Lemma~\ref{lem:number-of-disconnecting-sets} implies that
	\[ |\cU| \le 2^{|V|} \cdot 4^r \cdot \exp(Cr \log d) \le 2^{|V|} \cdot \exp\left(\tfrac{CL \log^2 d}{d^{3/2}} + \tfrac{C(M+N) \log^2 d}{d} \right) .  \]
	
	It remains to show that any $X \in \breakups_{L,M,N}$ seen from $V$ is tightly separated by some set in $\cW$. Let $X$ be such an atlas and denote $S^j := (X_P)_{P \in \phase_j}$ and $L^j := |\partial S^j|$ for $j \in \{0,1\}$.
	By Lemma~\ref{lem:existence-of-U}, there exists a set $U^j \subset (\intextB S^j)^+$ such that $|U^j| \le CL^j d^{-3/2} \log d$ and $N(U^j)$ separates $S^j$. Denote $U' := U^0 \cup U^1$ and note that $|U'| \le CL d^{-3/2} \log d$ and $N(U')$ tightly separates $X$. Hence, to obtain that $N(U') \in \cW$ and thus conclude the proof, it remains to show that $U' \subset U$ for some $U \in \cU$.
	
	By Lemma~\ref{lem:existence-of-covering2}, there exists $U'' \subset X_\bad \cup X_\overlap$ such that $|U''| \le C(M+N)d^{-1}\log d$ and $N_{2d}(X_\bad \cup X_\overlap) \subset N(U'')$.
	Denote $U := U' \cup U''$ and note that $X_* \subset U^{++}$, $U \subset X_*^+$ and $|U| \le r$. In particular, every connected component of $U^{+7}$ disconnects some vertex $v \in V$ from infinity so that $U \in \cU$.
\end{proof}

Before proving Lemma~\ref{lem:existence-of-U}, we start with a basic geometric property of odd sets which we require for the construction of the separating set.

\begin{lemma}\label{lem:four-cycle-property}
	Let $S$ be an odd set and let $\{u,v\} \in \partial S$. Then, for any unit vector $e \in \Z^d$, either $\{u,u+e\}$ or $\{v,v+e\}$ belongs to $\partial S$. In particular,
	\[ |\partial u \cap \partial S| + |\partial v \cap \partial S| \ge 2d .\]
\end{lemma}

\begin{proof}
	Assume without loss of generality that $u$ is odd. Since $S$ is odd, we have $u \in S$ and $v \notin S$.
	Similarly, if $u+e \in S$ then $v+e \in S$. Thus, either $\{u,u+e\} \in \partial S$ or $\{v,v+e\} \in \partial S$.
\end{proof}

For a set $S$, denote the {\em revealed vertices} in $S$ by
\[ S^{\rev} := \{ v \in \Z^d ~:~ |\partial v \cap \partial S| \geq d \} .\]
That is, a vertex is revealed if it sees the boundary in at least half of the $2d$ directions.
The following is an immediate corollary of Lemma~\ref{lem:four-cycle-property}.

\begin{cor}\label{cor:revealed-separate}
	Let $S$ be an odd set. Then $S^{\rev}$ separates $S$.
\end{cor}

\begin{proof}[Proof of Lemma~\ref{lem:existence-of-U}]
	Let $S=(S_i)_i$ be a collection of regular odd sets and denote $L := |\partial S|$, $\intB S := \bigcup_i \intB S_i$ and $\extB S := \bigcup_i \extB S_i$.
	Note that a set separates $S$ if and only if it separates $S_i$ for all $i$.
	Note also that $\partial S_i = \partial S_i^c$ implies that $S_i^{\rev} = (S_i^c)^{\rev}$. Thus, in light of Corollary~\ref{cor:revealed-separate} and by even-odd symmetry, it suffices to show that there exists a set $U \subset N(\bigcup_i \intB S_i)$ such that $\bigcup_i (S_i \cap S_i^{\rev}) \subset N(U)$ and $|U| \le CL d^{-3/2} \log d$.
	
	Denote $s:=\sqrt{d}$ and $t:=d/6$, and define
	\[ A := \{ v\text{ even} : |\partial v \cap \partial S| \ge s \} \qquad\text{and}\qquad A_i := \{ u\text{ odd} : |\partial u \cap \partial S_i| \ge 2d-s \} ,\]
	and observe that, by Lemma~\ref{lem:sizes},
	\[ |A| \le \frac{L}{s} \quad\text{ and }\quad \Big|\bigcup_i A_i\Big| \le \frac{L}{2d-s}. \]
	For an odd vertex $w$ and a vertex $v \sim w$, denote
	\[ M(w) := \big|\big\{ z \sim w : I(w,z) \neq \emptyset \big\}\big|, \qquad M(w,v) := \big|\big\{ z \sim w : I(w,z) \not\subset I(w,v) \big\}\big|, \]
	where
	\[ I(w,z) := \{ i : w \in A_i,~z \in S_i \} .\]
	Denote
	\[ T := \big\{ v\text{ even} : \exists w \sim v~ M(w,v) < \tfrac12 M(w) \big\}, \qquad T' := \big\{ w\text{ odd} : 1 \le M(w) \le 2s \big\} .\]
	We claim that
	\[ |T| \le 2s \cdot \Big|\bigcup_i A_i\Big| \qquad\text{and}\qquad |T'| \le \Big|\bigcup_i A_i\Big| .\]
	The second inequality is straightforward since $M(w) \ge 1$ implies that $w \in \bigcup_i A_i$. Let us show the first inequality. 	
	Observe that $T = \bigcup_w T(w)$, where the union is over odd $w$ and
	\[ T(w) := \big\{ v \sim w : M(w,v) < \tfrac12 M(w) \big\} .\]
	Then
	\[ \tfrac12 M(w) \cdot |T(w)| < \left|\Big\{ (v,z) \in N(w)^2 : \emptyset \neq I(w,z) \subset I(w,v) \Big\}\right| \le s M(w) .\]
	Since $T(w)\neq \emptyset$ implies $M(w) \ge 1$, it follows that $|T(w)| \le 2s$. Since $T(w) \neq \emptyset$ also implies that $w \in \bigcup_i A_i$, the desired inequality follows.
	
	We now use Lemma~\ref{lem:existence-of-covering2} with $A$ to obtain a set $B \subset A \subset \extB S$ such that
	\[ |B|\le\frac{4\log d}t|A| \qquad\text{and}\qquad N_t(A) \subset N(B). \]
	Applying the same lemma again, we obtain a set $B' \subset T \subset N(\intB S)$ such that
	\[ |B'|\le\frac{4\log d}t|T| \qquad\text{and}\qquad N_t(T) \subset N(B'). \]
	We also define
	\[ B'' := \bigcup_i (S_i \cap N_t(A_i \cap T')) .\]
	By Lemma~\ref{lem:sizes} and the definition of $T'$, we have
	\[ |B''| \le \frac{2s}{t}|T'| .\]
	Finally, we define $U:=B \cup B' \cup B''$. Clearly, $U \subset N(\intB S)$ and
	\[ |U| \le \frac{4L \log d}{t}\left(\frac{1}{s} + \frac{2s}{2d-s}\right) +\frac{2sL}{t(2d-s)} \le
	\frac{CL \log d}{d^{3/2}} .\]
	
	It remains to show that $S_i \cap S_i^{\rev} \subset N(U)$ for all $i$.
	Towards showing this, let $u \in S_i \cap S_i^{\rev} = \intB S_i \cap N_d(\extB S_i)$ for some $i$.
	Since $S_i$ is regular, there exists a vertex $z \in N(u) \cap S_i$.
	Let $F$ denote the set of pairs $(v,w)$ such that $(u,v,w,z)$ is a four-cycle and $v \in \extB S_i$, and note that $|F| \ge d-1$.
	Define
	\[ G^0 := \big\{ (v,w) \in F : v \in A \big\},~~ G^1 := \big\{ (v,w) \in F : v \in T \big\},~~ G^2 := \big\{ (v,w) \in F : w \in A_i \cap T' \big\} .\]
	It suffices to show that $F = G^0 \cup G^1 \cup G^2$, since then, either $|G^0| \ge |F|/3 \ge t$ in which case $u \in N_t(A) \subset N(B) \subset N(U)$, or $|G^1| \ge t$ in which case $u \in N_t(T) \subset N(B') \subset N(U)$, or $|G^2| \ge t$ in which case $z \in N_t(A_i \cap T')$ so that $z \in B''$ and $u \in N(B'') \subset N(U)$.
	
	Towards showing this, let $(v,w) \in F$ and note that $w \in S_i$. By Lemma~\ref{lem:four-cycle-property}, $v \in A$ or $w \in A_i$. In the former case, $(v,w) \in G^0$, so we may assume that $v \notin A$ and $w \in A_i$. Thus, if $w \in T'$ then $(v,w) \in G^2$ so that we may also assume that $w \notin T'$. Since $w \in A_i \setminus T'$, we have $M(w) > 2s$. Thus, to obtain that $v \in T$ and hence that $(v,w) \in G^1$, it suffices to show that $M(w,v) \le s$. Since $v \notin A$, this will follow if we show that $|\partial v \cap \partial S| \ge M(w,v)-1$. For this, it is enough to show that if $(v,w,x,y)$ is a four-cycle such that $I(w,x) \not\subset I(w,v)$, then $\{v,y\} \in \partial S$. Indeed, this statement is straightforward, since $j \in I(w,x) \setminus I(w,v)$ implies that $x \in S_j$ (so that $y \in S_j$) and $v \notin S_j$.
\end{proof}

\subsection{Constructing approximations}
\label{sec:Q-approx}

The proof of Lemma~\ref{lem:family-of-approx-from-separating-set} is split into two parts. We first show that every separating set gives rise to a small family of weak approximations.
A \emph{weak approximation} of a collection $S=(S_i)_i$ is a collection $A=((A_i)_i,A_*)$ such that $A_i \subset S_i \subset A_i \cup A_*$ for all~$i$. As before, we say that $A$ is \emph{controlled by} $W$ if $|A_*| \le C|W|$ and $A_* \subset W^+$.

\begin{lemma}\label{lem:family-of-small-approx-from-separating-set}
	For any integers $d \ge 2$ and $q \ge 1$, any rule $\cQ$ of rank at most $q$ and any finite set $W \subset \Z^d$, there exists a family $\cA$ of weak approximations, each controlled by $W$, such that
	\[ |\cA| \le 4^{\frac qd |W|} \]
	and any odd $\cQ$-collection which is separated by $W$ is weakly approximated by some $A \in \cA$.
\end{lemma}

The second step is to upgrade a weak approximation to a small family of approximations which covers at least the same set of $\cQ$-collections.

\begin{lemma}\label{lem:family-of-approx-from-small-approx}
	For any integers $d \ge 2$ and $q \ge 1$, any rule $\cQ$ of rank at most $q$ and any weak approximation $A$ controlled by some $W$, there exists a family $\cA$ of approximations, each of which is also controlled by $W$, such that
	\[ |\cA| \le 2^{\frac{q+1+\log d}d |A_*|} \]
	and any odd $\cQ$-collection which is weakly approximated by $A$ is approximated by some element in~$\cA$.
\end{lemma}

Note that Lemma~\ref{lem:family-of-approx-from-separating-set} follows immediately from Lemma~\ref{lem:family-of-small-approx-from-separating-set} and Lemma~\ref{lem:family-of-approx-from-small-approx}. We now prove these two lemmas.

\begin{proof}[Proof of Lemma~\ref{lem:family-of-small-approx-from-separating-set}]
Let $\cQ$ be a rule of rank at most $q$ and let $W \subset \Z^d$ be finite.
Consider the set $X := \Z^d \setminus W$.
Say that a connected component of $X$ is {\em small} if its
size is at most $d$, and that it is {\em large} otherwise.

Let $S=(S_i)_i$ be a collection of regular odd sets which is tightly separated by $W$, and observe that, for each $i$, every connected component $T$ of $X$ is entirely contained in either $S_i$ or $S_i^c$.
Define
\[ A_i := \bigcup \big\{ T\text{ large component of }X : i \in I(T) \big\}, \qquad\text{where }I(T) := \{ i : T \subset S_i \} .\]
Note that $A_i$ is contained in $S_i$ and that if $I(T)=\emptyset$ then $T \subset (\bigcup_i S_i)^c$.
Let $Y$ be the union of all the small components of $X$ and define $A_* := Y \cup W$.
Clearly, $A = A(S) := ((A_i)_i,A_*)$ is a weak approximation of $S$.

Next, we bound the size of $A_*$. For this we require a simple consequence of a well-known isoperimetric inequality (see, e.g., \cite[Corollary~2.3]{feldheim2015long}), namely,
\[ |\partial T| \ge d \cdot \min\{d,|T|\} \qquad\text{ for any finite }T \subset \Z^d .\]
Since any small component $T$ of $X$ has $|T| \le |\partial T|/d$ and $\partial T \subset \partial W$, we obtain
\[ |Y| \le \frac{|\partial W|}{d} \le \frac{2d |W|}{d} \le 2 |W| .\]
Thus, $|A_*| = |Y \cup W| \le 3 |W|$.

Let us now show that $A$ is controlled by $W$. For this, it remains only to show that $A_* \subset W^+$.
It suffices to show that $Y \subset N(W)$. To this end, let $v \in Y$ and note that $v^+ \not\subset Y$ by the definition of small component. Since $\extB Y \subset W$, we see that $v \in \extB W$.

Now, denote by $\cA$ the collection of weak approximations $A(S)$ constructed above for all odd $\cQ$-collections $S$ which are separated by $W$. To conclude the proof, it remains to bound $|\cA|$.
Let~$\ell$ be the number of large components of $X$.
Since every large component $T$ must contain an even vertex (it is a connected set of size at least 2), and since every $S$ in question is a $\cQ$-collection, the set $I(T)$ defined above always belongs to $\cQ$. Hence, as $\cQ$ has rank at most $q$, we have $|\cA| \le |\cQ|^{\ell} \le 2^{q \ell}$.
Since any large component $T$ of $X$ has
$|\partial T| \ge d^2$ and $\partial T \subset \partial W$, we obtain
$\ell \le |\partial W|/d^2 \le 2|W|/d$ so that $|\cA| \le 4^{|W|q/d}$, as required.
\end{proof}
The statement and proof of Lemma~\ref{lem:family-of-small-approx-from-separating-set} will be modified in the setting of $\Z^{d_1}\times\T_{2m}^{d_2}$, $d_1\ge 2$. The conclusion of the lemma will be weakened to $|\cA| \le C^{\frac qd |W|}$ (recalling that $d:=d_1+d_2$), which still suffices for our use in proving Lemma~\ref{lem:family-of-approx-from-separating-set}. In the proof, the isoperimetric inequality will be changed to
\[ |\partial T| \ge c d \cdot \min\{d,|T|\} \qquad\text{ for any finite }T \subset \Z^{d_1}\times\T_{2m}^{d_2} .\]
This inequality is clear if all vertices $v\in T$ satisfy $|N(v)\cap T|\le \frac{d}{2}$. Otherwise, let $v\in T$ be a vertex with $|N(v)\cap T|> \frac{d}{2}$ and note that the inequality follows similarly if $\sum_{w\in N(v)\cap T} |N(w)\cap T|\le \frac{d^2}{4}$. If the latter bound fails then one may check that $|\pi(T)|\ge c d^2$ where $\pi$ is the projection map $\pi(x_1,\ldots, x_d)=(x_2,\ldots, x_d)$, and then one may use the simple bound $|\partial T|\ge 2|\pi(T)|$.

\begin{proof}[Proof of Lemma~\ref{lem:family-of-approx-from-small-approx}]
Let $\cQ$ be a rule of rank at most $q$ and let $A=((A_i)_i,A_*)$ be a weak approximation.
Let us first show that we may assume that
\begin{equation}\label{eq:A_*-odd}
\Odd \cap N(A_*) \subset \bigcup_i A_i \cup A_* .
\end{equation}
Define $A'_* := A_* \setminus N(U)$, where $U:=\Odd \cap (\bigcup_i A_i \cup A_*)^c$. Note that $\Odd \cap A'_* = \Odd \cap A_*$ and $\Odd \cap N(A'_*) \subset \bigcup_i A_i \cup A'_*$. Let us show that any odd $\cQ$-collection $S$ which is weakly approximated by $A$ is also weakly approximated by $A'=((A_i)_i,A'_*)$. For this, it suffices to show that $S_i \subset A_i \cup A'_*$ for any $i$. Let $v \in S_i \setminus A_i \subset A_*$. If $v$ is odd then clearly $v \in A'_*$. If $v$ is even, then we must show that $v \notin N(U)$. This follows since $U \subset S_i^c$ and $S_i$ is odd.
Finally, since $A'_* \subset A_*$, we have that $|A'_*| \le |A_*|$ and that $A'$ is controlled by $W$. The lemma is thus reduced to the case that~\eqref{eq:A_*-odd} holds.

For a set $W \subset \Even \cap A_*$, define
\[ W_\out := \Even\cap N_d(A_*\setminus W^+) .\]
Observe that $W^+$ and $W_\out$ are disjoint.
Here one should think of $W$ as recording the location of a subset of even vertices in $A_* \cap (\bigcup_i S_i)$. We shall see that if this subset is chosen suitably then $W^+ \subset \bigcup_i S_i$ and $W_\out \subset \bigcap_i S_i^c$.

Let $\cW$ denote the family of sets $W \subset \Even \cap A_*$ having size at most $m/d$, where $m := |A_*|$.
We say that a collection $(W_i)_i$ is a $\cQ$-partition of $W$ if $W=\bigcup_i W_i$ and $\{ i : v \in W_i \} \in \cQ$ for all $v \in W$.
Define
\[ \cA := \left\{ \big((A_i \cup W_i^+)_i, A_* \setminus W_\out\big) ~:~ W \in \cW,~(W_i)_i\text{ is a $\cQ$-partition of $W$} \right\} .\]
Let us show that $\cA$ satisfies the requirements of the lemma.
To this end, we first bound the size of~$\cA$. We have
\[ |\cW| \le \binom{m}{\le m/d} \le (ed)^{m/d} = e^{(1+\log d)m/d} .\]
Hence,
\[ |\cA| \le |\cW| \cdot |\cQ|^{m/d} \le 2^{(q+1+\log d)m/d} .\]

Next, let us show that, for any $B=((B_i)_i,B_*)\in \cA$, we have
$\Even \cap B_* \subset N_d(\bigcup_i B_i)$.
To this end, let $W \in \cW$ be such that $\bigcup_i B_i = \bigcup_i A_i \cup W^+$ and $B_* = A_* \setminus W_\out$.
Let $v \in \Even \cap B_*$ and note that, by~\eqref{eq:A_*-odd}, $N(v) \setminus A_* \subset \bigcup_i A_i$. Thus, it suffices to show that $v \in N_d(W^+ \cup (A^*)^c)$. This in turn follows from $v \notin W_\out$.

It remains to show that any odd $\cQ$-collection $S$ which is weakly approximated by $A$ is approximated by some element in $\cA$. Let $S$ be such a collection. Let $W$ be a maximal subset of
$\Even \cap A_* \cap (\bigcup_i S_i)$ among those satisfying $d|W| \le |A_* \cap W^+|$, and note that $W \in \cW$.
Now define $B_i := A_i \cup W_i^+$ and $B_* := A_* \setminus W_\out$, where $W_i := W \cap S_i$. To show that $B:=((B_i)_i,B_*) \in \cA$, we must show that $(W_i)_i$ is a $\cQ$-partition of $W$.
Indeed, since $W$ is a set of even vertices, this follows from the fact that $S$ is a $\cQ$-collection.

To conclude that $B$ approximates $S$, we must show that $B_i \subset S_i \subset B_i \cup B_*$ for all $i$ and that $B_* \subset (\intextB S)^{+3}$. Since $A_i \subset S_i$ and $W_i \subset \Even \cap S_i$, and since $S_i$ is odd, it follows that $B_i \subset S_i$. Let $v \in S_i \setminus B_i$ and note that $v \in A_*$ since $A_i \subset B_i$. If $v$ is odd, then $v \in \Odd \cap A_* = \Odd \cap B_*$. Suppose that $v$ is even. To obtain that $v \in B_*$, it remains to show that $v \notin N_d(A_* \setminus W^+)$.
Indeed, by the maximality of $W$, and since $v \in A_* \setminus W$, we have
\[ d|W \cup \{v\}| > |A_* \cap (W \cup\{v\})^+| = |A_* \cap W^+| + |A_* \cap v^+ \setminus W^+| \ge d|W| + |A_* \cap v^+ \setminus W^+| ,\]
so that $|A_* \cap v^+ \setminus W^+| < d$.
Finally, $B_* \subset (\intextB S)^{+3}$ follows from $B_* \subset A_*$ and the fact that $S$ is weakly approximated by $A$.
\end{proof}

\section{Infinite-volume Gibbs states}
\label{sec:gibbs}

In this section, we prove Theorem~\ref{thm:existence_Gibbs_states} and Theorem~\ref{thm:characterization_of_Gibbs_states}.
The former is about the existence of a limiting Gibbs state for each dominant pattern and the properties of this measure. The latter is about the characterization of all maximal-entropy Gibbs states. The first is proven in Section~\ref{sec:P-pattern-Gibbs-state} (modulo the fact the measure has maximal entropy, which is deferred to Section~\ref{sec:max-entropy-states}) and the second in Section~\ref{sec:max-entropy-states}. We assume throughout this section that $q \ge 3$ and that $d$ satisfies~\eqref{eq:dim-assump}.

Let us first provide a formal definition of a Gibbs state (for uniform proper $q$-colorings).
A probability measure $\mu$ on $[q]^{\Z^d}$ (with the natural product $\sigma$-algebra) is a \emph{Gibbs state} if it is supported on proper $q$-colorings of $\Z^d$ and a random coloring $f$ sampled from $\mu$ has the property that, for any finite $\Lambda \subset \Z^d$, conditioned on the restriction $f|_{\Lambda^c}$, the restriction $f|_{\Lambda^+}$ is almost surely uniformly distributed on the set of proper $q$-colorings of $\Lambda^+$ that agree with $f$ on $\extB \Lambda$.

For a distribution $\mu$ on $[q]^{\Z^d}$, we denote by $\mu|_U$ the marginal distribution of $\mu$ on $[q]^U$. Given two discrete distributions $\mu$ and $\lambda$ on a common space, we denote their total-variation distance by $\distTV(\mu,\lambda) := \max_{A} |\mu(A)-\lambda(A)|$, where the maximum is over all events $A$. Recall that a domain is a finite, non-empty, connected and co-connected subset of $\Z^d$.

\subsection{Large violations}\label{sec:large-violations}
For the proofs of Theorem~\ref{thm:existence_Gibbs_states} and Theorem~\ref{thm:characterization_of_Gibbs_states}, we require two extensions of Theorem~\ref{thm:long-range-order} to larger violations of the boundary pattern rather than just single-site violations. Recall the definition of $Z_*(f)$ from~\eqref{eq:Z_*-def}.
Let $Z_*^{+5}(f,V)$ denote the union of the connected components of $Z_*(f)^{+5}$ that are either infinite or disconnect some vertex in $V$ from infinity.

\begin{prop}\label{prop:Z_*-bound}
	Let $\Lambda$ be a domain and let $V \subset \Z^d$ be finite. Then, for any $k \ge 1$,
	\begin{equation*}
	\Pr_{\Lambda,P_0}\big(|Z_*(f) \cap Z_*^{+5}(f,V)| \ge k\big) \le 2^{|V|} \cdot e^{-\frac{ck}{q^3(q + \log d)d}} .
	\end{equation*}
\end{prop}

\begin{proof}
	We omit $f$ from notation.
	Let $\Omega_{L,M,N}$ denote the event that there exists a breakup in $\breakups_{L,M,N}$ seen from~$V$. Let us show that $|Z_* \cap  Z_*^{+5}(V)| \ge k$ implies the occurrence of $\Omega_{L,M,N}$ for some $L,M,N \ge 0$ satisfying that $L/2 + M+N \ge k$.
	
	Lemma~\ref{lem:existence-of-breakup} implies the existence of a breakup $X$ such that $X_*^{+5}=Z_*^{+5}(V)$.
	Note that this implies that $X_*= Z_* \cap Z_*^{+5}(V)$ so that $|X_*| \ge k$. Since every vertex in $\bigcup_P \intextB X_P$ is an endpoint of an edge in $\bigcup_P \partial X_P$, and since every edge has only two endpoints, we see that $X \in \breakups_{L,M,N}$ implies that $L/2+M+N \ge k$. Note also that Lemma~\ref{lem:boundary-size-of-odd-set} implies that $\breakups_{L,M,N}=\emptyset$ when $L<d^2$.
	Therefore, by Proposition~\ref{prop:prob-of-breakup-associated-to-V},
	\[ \Pr\big(|Z_* \cap Z_*^{+5}(V)| \ge k\big) \le 2^{|V|} \sum_{\substack{L \ge d^2,\,M,N \ge 0\\L/2+M+N \ge k}} \exp\left(- \tfrac{c}{q^3(q+\log d)} \big( \tfrac{L}{d}+\tfrac{M}{q}+\tfrac{N}{q^2} \big) \right) .\]
	Using~\eqref{eq:dim-assump}, the desired inequality follows.
\end{proof}

Recall the definition of $Z_P(f)$ from~\eqref{eq:Z-def}, and that, while the $P$-even vertices in $Z_P(f)$ are always in the $P$-pattern, the $P$-odd vertices there need not be. Let $\bar{Z}_P(f)$ denote the subset of $Z_P(f)$ that is in the $P$-pattern.
For $V \subset \Z^d$, define $\cB_P(f,V)$ to be the union of the $(\Z^d)^{\otimes 2}$-connected components of $\bar{Z}_P(f)^c$ that intersect $V$.
For a set $U \subset \Z^d$, define $\diam^* U := 2m+\diam U_1 + \dots + \diam U_m$, where $\{U_i\}_{i=1}^m$ are the $(\Z^d)^{\otimes 2}$-connected components of~$U$.

\begin{prop}\label{prop:B_P-bound}
	Let $\Lambda$ be a domain and let $V \subset \Z^d$ be finite. Then, for any $k \ge 1$,
	\begin{equation}\label{eq:bound-on-diameter-B_P}
	\Pr_{\Lambda,P}\big(\diam^* \cB_P(f,V) \ge k \big) \le 2^{|V|} \cdot e^{-\frac{cdk}{q^4(q + \log d)}}.
	\end{equation}
\end{prop}

For the proof, we require the following adaptation of~\cite[Lemma~2.4]{feldheim2015long}, which states that a finite, odd, connected set $A$ has $|\partial A| \ge (d-1)^2 \diam A$.
For $A \subset \Z^d$, denote the isolated vertices in $A$ by
\begin{equation}\label{eq:def-partial+}
	A_\iso := \{ v \in A : N(v) \cap A = \emptyset \} .
\end{equation}

\begin{lemma}\label{lem:boundary-size-via-diameter2}
	Let $A \subset \Z^d$ be finite, odd and $(\Z^d)^{\otimes 2}$-connected. Then
	\[ |\partial A| + |\partial(A_\iso^+)| \ge \tfrac12(d-1)^2 (2+\diam A) .\]
\end{lemma}
\begin{proof}
	Let $u,v \in A$ be such that $k := \dist(u,v) = \diam A$ and let $p$ be a $(\Z^d)^{\otimes 2}$-path in $A$ between $u$ and $v$.
	Let $\{A_i\}_i$ be the connected components of $A \setminus A_\iso$ and $A_\iso^+$ that intersect~$p$. Observe that each $A_i$ is even or odd and that $k+2 \le \sum_i (2+\diam A_i) \le 2\sum_i \diam A_i$. Since $\partial A_i \subset \partial A \cup \partial(A_\iso^+)$ for all~$i$, the lemma follows using that $|\partial A_i| \ge (d-1)^2 \diam A_i$ by~\cite[Lemma~2.4]{feldheim2015long}.
\end{proof}
Lemma~\ref{lem:boundary-size-via-diameter2} will be modified in the setting of $\Z^{d_1}\times\T_{2m}^{d_2}$, $d_1\ge 2$ (with $d = d_1 + d_2$): Its conclusion will be weakened to $|\partial A| + |\partial(A_\iso^+)| \ge c(d-1)^2 (2+\diam A)$, which still suffices for our use in proving Proposition~\ref{prop:B_P-bound}, and its proof will use an extension of \cite[Lemma~2.4]{feldheim2015long} which states that a finite, odd, connected set $A \subset \Z^{d_1}\times\T_{2m}^{d_2}$ has $|\partial A| \ge c (d-1)^2\diam A$. Indeed, the proof of \cite[Lemma~2.4]{feldheim2015long} applies with the minor modification that the index $i$ there should be chosen as one of the $d_1$ infinite directions.

\begin{proof}[Proof of Proposition~\ref{prop:B_P-bound}]
	We denote $\cB:=\cB_P(f,V)$ and omit $f$ from notation.
	Let $\Omega_{L,M,N}$ denote the event that there exists a breakup in $\breakups_{L,M,N}$ seen from~$V$. Let us show that $\diam^* \cB \ge k$ implies the occurrence of $\Omega_{L,M,N}$ for some $L,M,N \ge 0$ satisfying that $L+2dM \ge cd^2k$.
	
	Note that $\intB \cB \subset \extB \bar{Z}_P$ so that, in particular, $\cB$ is a $P$-odd set.
	Note also that $\cB^{+2} \setminus \cB \subset \bar{Z}_P$ and that $\cB_\iso \subset Z_P \setminus \bar{Z}_P$.
	We claim that $\partial (\cB \setminus \cB_\iso) \subset \partial Z_P$ and $\cB_\iso^+ \subset Z_\overlap$, so that, in particular, $\intextB \cB \subset Z_*$. To see the former, let $(u,v) \in \dpartial (\cB \setminus \cB_\iso)$ and $w \in N(u) \cap \cB$, so that $v \in Z_P$ and $w \notin Z_P$. It follows that $u \notin Z_P$ and hence that $\{u,v\} \in \partial Z_P$ as required. To see the latter, let $u \in \cB_\iso^+$ and $w \in u^+ \cap \cB_\iso$, so that $f(w) \in P_\bdry$ and $f(N(w)) \subset P_\bdry$. It follows that $w^+ \subset Z_Q$ for any dominant pattern $Q=(A,B)$ such that $P_\bdry \setminus \{f(w)\}$, so that $u \in Z_\overlap$ as required.
	
	Lemma~\ref{lem:existence-of-breakup} implies the existence of a breakup $X$ such that $X_*^{+5}=Z_*^{+5}(V)$.
	Note that this implies that $X_*= Z_* \cap Z_*^{+5}(V)$.
	Since $\intextB \cB \subset Z_*$ and since every $(\Z^d)^{\otimes 2}$-connected component of $\cB$ intersects $V$, it follows that $\intextB \cB \subset Z_*^{+5}(V)$ and hence that $\intextB \cB \subset X_*$. In particular, by~\eqref{eq:breakup-1}, $X_Q$ and $Z_Q$ coincide near the boundary of $\cB$ for all $Q$, so that $\partial (\cB \setminus \cB_\iso) \subset \partial X_P$ and $\cB_\iso^+ \subset X_\overlap$.
	
	Let $L,M,N \ge 0$ be such that $X \in \breakups_{L,M,N}$.
	Applying Lemma~\ref{lem:boundary-size-via-diameter2} to each $(\Z^d)^{\otimes 2}$-connected component of~$\cB$ yields that $|\partial \cB| +|\partial(\cB_\iso^+)| \ge c d^2 k$. Since $|\partial(\cB_\iso^+)| \le 2d|\cB_\iso^+|$, we conclude that $L+2dM \ge cd^2k$.
	Therefore, by Proposition~\ref{prop:prob-of-breakup-associated-to-V},
	\[ \Pr\big(\diam^* \cB \ge k\big) \le 2^{|V|} \sum_{\substack{L,M,N \ge 0\\L+2dM\ge cd^2k}} \exp\left(- \tfrac{c}{q^3(q+\log d)} \big( \tfrac{L}{d}+\tfrac{M}{q}+\tfrac{N}{q^2} \big) \right) .\]
	Using~\eqref{eq:dim-assump}, the desired inequality follows.
\end{proof}

For the proof of Theorem~\ref{thm:existence_Gibbs_states}, we also require a corollary of Proposition~\ref{prop:B_P-bound} for violations of the boundary pattern in a pair of proper colorings.
Given two proper colorings $f$ and $f'$ of $\Z^d$, define $\cB_P(f,f',u)$ to be the $(\Z^d)^{\otimes 2}$-connected component of $u$ in $(\bar{Z}_P(f) \cap \bar{Z}_P(f'))^c$.

\begin{cor}\label{cor:prob-of-joint-breakup-core}
	Let $\Lambda$ and $\Lambda'$ be two domains and $f \sim \Pr_{\Lambda,P}$ and $f' \sim \Pr_{\Lambda',P}$ be independent. Then
	\[ \Pr\big(\diam \cB_P(f,f',u) \ge r \big) \le e^{-\frac{cdr}{q^4(q+\log d)}} \qquad\text{for any }r \ge 1\text{ and }u \in \Z^d .\]
\end{cor}

For the proof of Corollary~\ref{cor:prob-of-joint-breakup-core}, we require the following simple adaptation of~\cite[Lemma~6.9]{feldheim2015long}.

\begin{lemma}\label{lem:diam-witness}
	Let $U,V \subset \Z^d$ be finite and assume that $U \cup V$ is $(\Z^d)^{\otimes 2}$-connected. Then for any $u,v \in U \cup V$ there exists a path $p$ from $u$ to $v$ of length at most $\diam^* U_p + \diam^* V_p$, where $U_p$ and $V_p$ are the union of $(\Z^d)^{\otimes 2}$-connected components of $U$ and $V$ which intersect $p$.
\end{lemma}
\begin{proof}
	Let $\mathcal{W}$ be the collection of $(\Z^d)^{\otimes 2}$-connected components of $U$ and $V$.
	Consider the graph $G$ on vertex set $\mathcal{W}$ in which $W,W' \in \mathcal{W}$ are adjacent if and only if $\dist(W,W') \le 2$.
	Note that $G$ is connected.
	Consider a simple path $q=(W_1,\dots,W_k)$ in $G$, where $u \in W_1$ and $v \in W_k$.
	For each $1 \le i \le k-1$, let $u_i \in W_i$ and $v_i \in W_{i+1}$ be such that $\dist(u_i,v_i) \le 2$.
	Let $p$ be a path from $u$ to $v$ constructed by connecting $v_{i-1}$ to $u_i$ by a shortest-path for every $1 \le i \le k$ (where we set $v_0 := u$ and $u_k := v$), and $u_i$ to $v_i$ by at most one other vertex for every $1 \le i \le k-1$.
	Then the length of $p$ is at most $\sum_{i=1}^k (\diam W_i + 2)$. On the other hand, $\diam^* U_p + \diam^* V_p \ge \sum_{i=1}^k (\diam W_i + 2)$, and the lemma follows.
\end{proof}

\begin{proof}[Proof of Corollary~\ref{cor:prob-of-joint-breakup-core}]
	Denote $\cB:=\cB_P(f,f',u)$ and suppose that $\diam \cB \ge r$.
	Let $v \in \cB$ be such that $\dist(u,v) \ge r/2$.
	Note that $\cB$ is contained in $\cB(f,\Z^d) \cup \cB(f',\Z^d)$.	
	By Lemma~\ref{lem:diam-witness} applied to $\cB \cap \cB(f,\Z^d)$ and $\cB \cap \cB(f',\Z^d)$, there exists a path $p$ from $u$ to $v$ of length $s \le \diam^* \cB(f,p) + \diam^* \cB(f',p)$.	
	In particular, $T := \max \{\diam^* \cB(f,p),\diam^* \cB(f',p)\}$ is at least $s/2$.
	Thus, by a union bound on the choices of $p$ and $T$, Proposition~\ref{prop:B_P-bound} and~\eqref{eq:dim-assump},
	\[ \Pr\big(\diam \cB \ge r \big) \le \sum_{t=\lceil r/4 \rceil}^\infty 2 (2d)^{2t} 2^{2t+1} e^{- \frac{cdt}{q^4(q+\log d)}} \le e^{-\frac{cdr}{q^4(q+\log d)}} . \qedhere \]	
\end{proof}

\subsection{The $P$-pattern Gibbs state}
\label{sec:P-pattern-Gibbs-state}

In this section, we show that $\Pr_{\Lambda,P}$ converges as $\Lambda \uparrow \Z^d$ to an infinite-volume Gibbs state $\mu_P$ that satisfies a mixing property which, in particular, implies that $\mu_P$ is extremal. This is the content of the following two lemmas.

\begin{lemma}\label{lem:convergence}
	Let $\Lambda$ and $\Lambda'$ be two domains.
	Let $r \ge 1$ and let $U$ be a domain such that $U^{+r} \subset \Lambda \cap \Lambda'$.
	Then
	\[ \distTV\big(\Pr_{\Lambda,P}|_U, \Pr_{\Lambda',P}|_U\big) \le |U| \cdot e^{-\frac{cdr}{q^4(q+\log d)}} .\]
\end{lemma}

\begin{lemma}\label{lem:almost-independence-of-colorings}
	Let $\Lambda$ be a domain, let $V \subset \Lambda$ be a domain, let $r \ge 1$ and let $U \subset \Lambda$ be such that $U^{+2r} \subset V$. Then
\[ \distTV\big(\Pr_{\Lambda,P}|_{U \cup (\Lambda \setminus V)}, \Pr_{\Lambda,P}|_U \times \Pr_{\Lambda,P}|_{\Lambda \setminus V}\big) \le |U| \cdot e^{-\frac{cdr}{q^4(q+\log d)}} .\]
\end{lemma}

Lemma~\ref{lem:convergence} easily implies that the finite-volume $P$-pattern measures converge to an infinite-volume Gibbs state $\mu_P$.
Indeed, if $(\Lambda_n)$ is a sequence of domains increasing to $\Z^d$, then for any domain $U$, $\dist(U,\Lambda_n^c) \to \infty$ as $n \to \infty$, so that Lemma~\ref{lem:convergence} implies that the sequence of measures $(\Pr_{\Lambda_n,P}|_U)_{n=1}^{\infty}$ is a Cauchy sequence with respect to the total-variation metric, and therefore, converges. This establishes the convergence of $\Pr_{\Lambda_n,P}$ as $n \to \infty$ towards an infinite-volume measure $\mu_P$ and it follows that this limit is a Gibbs state. Since this holds for any such sequence $(\Lambda_n)$, it follows that $\mu_P$ is invariant to all automorphisms preserving the two sublattices.
Lemma~\ref{lem:almost-independence-of-colorings} then easily implies that $\mu_P$ satisfies the following mixing property: for any $0<\delta<1$, there exist constants $A,a>0$ such that
\[ \distTV\big(\mu_P|_{B_{\delta n} \cup (\Z^d \setminus B_n)}, \mu_P|_{B_{\delta n}} \times \mu_P|_{\Z^d \setminus B_n}\big) \le A e^{-an} \qquad\text{for all }n \ge 1 ,\]
where $B_m := [-m,m]^d \cap \Z^d$ (this property is termed \emph{quite weak Bernoulli with exponential rate} in~\cite{burton1995quite} in the context of translation-invariant measures). In particular, for any $k \ge 1$,
\[ \lim_{n \to \infty} \distTV\big(\mu_P|_{B_k \cup (\Z^d \setminus B_n)}, \mu_P|_{B_k} \times \mu_P|_{\Z^d \setminus B_n}\big) = 0 .\]
It is fairly standard to conclude from this that $\mu_P$ is tail trivial~(see~\cite[Proposition~7.9]{georgii2011gibbs}), which is equivalent to extremality within the set of all Gibbs states~(see~\cite[Theorem~7.7]{georgii2011gibbs}). Noting that~\eqref{eq:color_distribution} implies that different~$P$ yield different measures $\mu_P$, Theorem~\ref{thm:existence_Gibbs_states} will follow once we show that $\mu_P$ is of maximal entropy. We postpone this part to Section~\ref{sec:max-entropy-states} (it is a consequence of Proposition~\ref{prop:max-entropy-states-are-mixtures}. A direct proof, using Kempe chains, appears in~\cite{galvin2012phase}, where it is stated for the $q=3$ case but the proof may be extended to general $q$).

The proofs of Lemma~\ref{lem:convergence} and Lemma~\ref{lem:almost-independence-of-colorings} make use of the following fact which exploits the domain Markov property of the model.
We say that a collection $\cS$ of proper subsets of $\Z^d$ is a \emph{boundary semi-lattice} if for any $S_1,S_2 \in \cS$ there exists $S \in \cS$ such that $S_1 \cup S_2 \subset S$ and $\partial S \subset \partial S_1 \cup \partial S_2$.
Two boundary semi-lattices which we require are $\cS(U,V) := \{ S \subsetneq \Z^d : U \subset S \subset V \}$ and $\cS(f,P) := \{ S \subsetneq \Z^d : \intextB S\text{ is in the $P$-pattern with respect to }f \}$.
The latter has the property that if $\cS$ is any boundary semi-lattice, then $\cS \cap \cS(f,P)$ is also a boundary semi-lattice.

\begin{lemma}\label{lem:marginal-distribution-given-agreement}
	Let $\Lambda,\Lambda' \subset \Z^d$ be finite and let $U \subset V \subset \Lambda \cap \Lambda'$ be non-empty.
	Let $f \sim \Pr_{\Lambda,P}$ and $f' \sim \Pr_{\Lambda',P}$ be independent.
	\begin{enumerate}[label=(\alph*)]
		\item \label{it:marginal-distribution-given-agreement-pair} $\distTV(\Pr_{\Lambda,P}|_U,\Pr_{\Lambda',P}|_U) \le \Pr(\cS(U,V) \cap \cS(f,P) \cap \cS(f',P)=\emptyset)$.
		\item \label{it:marginal-distribution-given-agreement-single} Assume that $U$ is connected, $V$ is co-connected and $\Pr(\cS(U,V) \cap \cS(f,P) \neq \emptyset)>0$. Then, conditioned on $\{\cS(U,V) \cap \cS(f,P) \neq \emptyset \}$, the distribution of $f|_U$ is a convex combination of the measures $\{ \Pr_{S,P}|_U \}_{S \in \cS^{\text{dom}}(U,V)}$, where $\cS^{\text{dom}}(U,V)$ is the collection of domains in $\cS(U,V)$.
	\end{enumerate}
\end{lemma}

\begin{proof}
	We shall prove both items together.
	To this end, let $f''$ be either $f$ or $f'$, and denote $\cS := \cS(U,V) \cap \cS(f,P) \cap \cS(f'',P)$.
	Since $\cS$ is a finite boundary semi-lattice, it has a unique maximal element $\sf S$ (if  $\cS = \emptyset$, we set $\sf S := \emptyset$).
	Let $S \neq \emptyset$ be such that $\Pr({\sf S}=S)>0$.
	Observe that the event $\{{\sf S}=S\}$ is determined by $f|_{(S^c)^+}$ and $f''|_{(S^c)^+}$.
	Therefore, by the domain Markov property, conditioned on $\{{\sf S}=S\}$, $f|_S$ and $f''|_S$ are distributed as $\Pr_{S,P}|_S$. In particular, conditioned on $\{\cS \neq \emptyset \}$, the distribution of both $f|_U$ and $f''|_U$ is $\sum_S \Pr({\sf S}=S \mid \cS \neq \emptyset) \Pr_{S,P}|_U$, from which the first item follows. Moreover, if $U$ is connected and $V$ is co-connected, then $\sf S$ is always a domain, since Lemma~\ref{lem:co-connect-properties}\ref{it:co-connect-reduces-boundary} and Lemma~\ref{lem:co-connect-properties}\ref{it:co-connect-kills-components} imply that the co-connected closure of $S$ (with respect to infinity) belongs to $\cS$ for any $S \in \cS$. Hence, the second item also follows.
\end{proof}

We are now ready to prove Lemma~\ref{lem:convergence} and Lemma~\ref{lem:almost-independence-of-colorings}.

\begin{proof}[Proof of Lemma~\ref{lem:convergence}]
	Denote $S := U \cup \bigcup_{u \in \intextB U} \cB_P(f,f',u)^+$ and observe that, by definition, $\intextB S$ is in the $P$-pattern with respect to both $f$ and $f'$.
	Let $\cE$ be the event that $S$ intersects $(U^{+r})^c$, so that $S \subset U^{+r}$ on the complement of $\cE$.
	Then, by Lemma~\ref{lem:marginal-distribution-given-agreement} and Corollary~\ref{cor:prob-of-joint-breakup-core},
	\[ \distTV\big(\Pr_{\Lambda,P}|_U, \Pr_{\Lambda',P}|_U\big) \le \Pr(\mathcal E) \le \sum_{u \in U} \Pr\big(\diam \cB_P(f,f',u) \ge r\big) \le |U| \cdot e^{-\frac{cdr}{q^4(q+\log d)}} . \qedhere \]
\end{proof}

\begin{proof}[Proof of Lemma~\ref{lem:almost-independence-of-colorings}]
We begin with a simple observation. Let $X$ and $Y$ be discrete random variables and let $\mu_{X|Y}$ denote the conditional (random) distribution of $X$ given $Y$. Then
\[ \distTV(\mu_{(X,Y)}, \mu_X \times \mu_Y) = \E[\distTV(\mu_{X|Y},\mu_X)] ,\]
where we write $\mu_Z$ for the distribution of a random variable $Z$.
Indeed, the verification of this is straightforward using that $\distTV(\mu,\lambda)=\frac12 \sum_i |\mu(i)-\lambda(i)|$.

Let $\mu$ be the conditional (random) distribution of $f|_U$ given $f|_{V^c}$. Let $\cE'$ be the event that there exists a set $S$ such that $U^{+r} \subset S \subset V$ and such that $\intextB S$ is in the $P$-pattern.
By Lemma~\ref{lem:marginal-distribution-given-agreement}, conditioned on $\cE'$, $\mu$ is a convex combination of measures $\Pr_{S,P}|_U$, where $S$ is a domain containing $U^{+r}$.
For any such $S$, by Lemma~\ref{lem:convergence}, we have \[ \distTV(\Pr_{S,P}|_U,\Pr_{\Lambda,P}|_U) \le |U| \cdot e^{-\frac{cdr}{q^4(q+\log d)}} .\]
Let $\cE$ be the event that $\cB_P(f,u)^+$ intersects $V^c$ for some $u \in U^{+r}$, and note that $\cE^c \subset \cE'$.
Hence,
\[ \E[\distTV(\mu,\Pr_{\Lambda,P}|_U)] \le |U| \cdot e^{-\frac{cdr}{q^4(q+\log d)}} + \E[\mu(\cE)] .\]
By Proposition~\ref{prop:B_P-bound},
\[ \E[\mu(\cE)] = \Pr(\cE) \le |U^{+r}| \cdot e^{-\frac{cdr}{q^4(q+\log d)}} \le |U| \cdot (Cd)^r \cdot e^{-\frac{cdr}{q^4(q+\log d)}} \le |U| \cdot e^{-\frac{cdr}{q^4(q+\log d)}} .\]
Thus, $\E[\distTV(\mu,\Pr_{\Lambda,P}|_U)] \le |U| \cdot e^{-\frac{cdr}{q^4(q+\log d)}}$, and the lemma follows from the above observation.
\end{proof}

\subsection{The maximal-entropy Gibbs states}
\label{sec:max-entropy-states}

The purpose of this section is to characterize all maximal-entropy Gibbs states.
Let us begin by defining the relevant notions.
Let $\mu$ be a probability measure on $[q]^{\Z^d}$.
Given a transformation $T \colon \Z^d \to \Z^d$, we say that $\mu$ is $T$-invariant if $\mu(T^{-1}A)=\mu(A)$ for any measurable event $A$. We say that $\mu$ is \emph{periodic} if it is $\Gamma$-invariant for a (full-dimensional) lattice $\Gamma$ of translations of $\Z^d$.
Observe that every periodic measure $\mu$ is $(N \Z^d)$-invariant for some positive integer $N$.

To define the notion of a maximal-entropy Gibbs state, we first require some other definitions.
Let $\Omega^{\text{free}}_\Lambda$ be the set of proper colorings of $\Lambda$.
The \emph{topological entropy} of proper colorings is the exponential rate of growth of the number of proper colorings, i.e.,
\[ h_{\text{top}} := \lim_{n \to \infty} \frac{\log \big|\Omega^{\text{free}}_{[n]^d}\big|}{n^d} .\]
The above limit exists by subadditivity (see~\cite[Lemma~15.11]{georgii2011gibbs}).
Note also that $\tfrac12 \log (\lfloor \frac{q}{2}\rfloor \lceil \frac{q}{2}\rceil)$ is a trivial lower bound on $h_{\text{top}}$.
Let $\mu$ be a periodic measure which is supported on proper $q$-colorings of $\Z^d$.
The \emph{measure-theoretic entropy} (also known as Kolmogorov--Sinai entropy) of $\mu$ is
\[ h(\mu) := \lim_{n \to \infty} \frac{\Ent(\mu|_{\Lambda_n})}{|\Lambda_n|} , \qquad \text{where }\Lambda_n := \{0,1,\dots,n\}^d ,\]
which also exists by subadditivity (see~\cite[Theorem~15.12]{georgii2011gibbs}).
Using~\eqref{eq:entropy-support}, one easily checks that $h(\mu) \le h_{\text{top}}$.
The variational principle tells us that equality is achieved by some $\mu$. Such a $\mu$ is said to be of \emph{maximal entropy}.
A theorem of Lanford--Ruelle (see, e.g., \cite{misiurewicz1976short}) tells us that every measure of maximal entropy is also a Gibbs state (so that there is some redundancy when speaking about a maximal-entropy Gibbs state). We stress that a measure of maximal entropy is, by definition, always assumed to be periodic.

Before proceeding with the proof of Theorem~\ref{thm:characterization_of_Gibbs_states}, let us give a simple consequence of our results to the enumeration of proper colorings.
Using the sub-additivity of entropy~\eqref{eq:entropy-subadditivity}, it is straightforward to see that Theorem~\ref{thm:long-range-order}, together with the fact that $\mu_P$ is of maximal entropy, implies that, when~\eqref{eq:dim-assump} holds, the topological entropy is bounded by
\[ h_{\text{top}} \le \tfrac12 \log (\lfloor \tfrac{q}{2}\rfloor \lceil \tfrac{q}{2}\rceil) + e^{-\frac{cd}{q^3(q+\log d)}} .\]
Galvin--Tetali~\cite{galvin2004weighted} showed a weaker bound of this form (where the exponential correction term is replaced by a term of order~$\frac{1}{d}$) on any bipartite regular graph (in which case their bound is of the correct order) and in the context of general graph homomorphisms. Using either bound, we see that $h_{\text{top}} \to \tfrac12 \log (\lfloor \frac{q}{2}\rfloor \lceil \frac{q}{2}\rceil)$ as $d \to \infty$. An analogue of this for isotropic subshifts was shown by Meyerovitch--Pavlov~\cite{meyerovitch2014independence}. On the hypercube, the asymptotics of the number of proper $3$-colorings were found by Galvin~\cite{Galvin2003hammingcube} (following Kahn--Lawrentz~\cite{kahn1999generalized} and Kahn~\cite{Kahn2001hypercube}) and, recently, the asymptotics of the number of proper $4$-colorings were determined by Kahn--Park~\cite{kahn2020number} verifying a conjecture of Engbers--Galvin~\cite{engbers2012h2}.

Let us come back to the proof of Theorem~\ref{thm:characterization_of_Gibbs_states}.
We wish to show that the $P$-pattern Gibbs states are the only extremal maximal-entropy measures. Our technique is inspired by the work of Gallavotti and Miracle-Sol{\'e}~\cite{gallavotti1972equilibrium} on the translation-invariant Gibbs states of the low-temperature Ising model. In order to allow ourselves to appeal directly to Proposition~\ref{prop:Z_*-bound} in the proof (instead of repeating similar arguments), we first show that proper colorings with periodic boundary conditions may be extended to $P$-pattern boundary conditions.

A proper coloring $f$ of $\{-n,\dots,n\}^{d-1}$ is \emph{symmetric} if $f(x_1,\dots,x_{d-1})=f(|x_1|,\dots,|x_{d-1}|)$ for all $x \in \{-n,\dots,n\}^{d-1}$.
A proper coloring $f$ of $U \subset \Z^{d-1}$ is \emph{$n$-periodic} if $f(x)$ depends only on $(x_1\text{ mod }n,\dots,x_{d-1}\text{ mod }n)$ for $x \in U$. A proper coloring of $\{-kn,\dots,kn\}^{d-1}$ is \emph{$n$-symmetric} if it is $2n$-periodic and its restriction to $\{-n,\dots,n\}^{d-1}$ is symmetric. A proper coloring of $\Lambda_{2kn}$ is \emph{$n$-symmetric} if its restriction to any of the $2d$ faces is $n$-symmetric (after an appropriate translation).
Finally, a proper coloring of $U$ has $(a,b)$-boundary conditions if the even vertices in $\intB U$ take the value $a$ and the odd ones take $b$.

\begin{lemma}\label{lem:pattern-extends-to-ab-boundary-conditions}
Any $n$-symmetric proper coloring $f$ of $\Lambda_{2kn}$ can be extended to a proper coloring of $(\Lambda_{2kn})^{+dn}$ having $(a,b)$-boundary conditions, where $a := f(0,\dots,0)$ and $b := f(1,0,\dots,0)$.
\end{lemma}
\begin{proof}
	Let $K_q$ be the complete graph on $[q]$.
Say that two paths $p=(p_m)_{m \ge 0}$ and $q=(q_m)_{m \ge 0}$ in $K_q$ are adjacent if $p_m \neq q_m$ for all $m \ge 0$.
Denote $\Lambda := \Lambda_{2kn}$.
Let $(p^u)_{u \in \intB \Lambda}$ be a family of paths such that $p^u$ and $p^v$ are adjacent whenever $u \sim v$ and such that $p^u_0 = f(u)$ for every $u \in \intB \Lambda$.
Observe that every $x \in \Z^d$ has a unique $u(x) \in \Lambda$ closest to $x$ and that $\dist(u(x),u(y)) \le \dist(x,y)$. In particular, if $x \sim y$ then either $\dist(u(x),x)=\dist(u(y),y)$ and $u(x) \sim u(y)$ or $\dist(u(x),x)=\dist(u(y),y) \pm 1$ and $u(x)=u(y)$. Hence, defining $g \colon \Z^d \to [q]$ by
\[ g(x) := \begin{cases} f(x) &\text{if }x \in \Lambda\\p^{u(x)}_{\dist(u(x),x)} &\text{if }x \notin \Lambda \end{cases} ,\]
we have that $g(x) \neq g(y)$ whenever $x \sim y$. Thus, $g$ is a proper coloring of $\Z^d$ which extends $f$.
To conclude, it suffices to show the existence of such a family of paths $(p^u)_{u \in \intB \Lambda}$ which also satisfies that
\begin{equation}\label{eq:a-b-path}
\text{for every $u \in \intB \Lambda$ there exists $0 \le m \le dn$ such that }(p^u_m,p^u_{m+1},\dots)=(a,b,a,b,\dots) .
\end{equation}
Indeed, the lemma will then follow as $g|_{\Lambda^{+dn}}$ has $(a,b)$-boundary conditions.
To construct such a family, we first define $p^u$ for $u \in \Lambda' := \Lambda_n \cap \intB \Lambda$ by
\[ p^u := \big(f(u),f(S u),f(S^2u),\dots,f(S^{\ell_u}u),b,a,b,a,\dots\big) ,\]
where $S \colon \Lambda' \setminus \{0\} \to \intB \Lambda$ is the lexicographical successor operator defined by $Su := u-e_j$, where $j := \min \{ i \ge 1 : u_i > 0 \}$, and $\ell_u := \min\{ \ell \ge 0 : S^\ell u =0 \}$.
Note that $\ell_u = |u| := \sum_{1 \le i \le d} |u_i| \le dn$ so that~\eqref{eq:a-b-path} holds for all $u \in \Lambda'$. For $u \in \intB \Lambda \setminus \Lambda'$, define $p^u := p^{(|r_1|,\dots,|r_d|)}$, where $r_i$ is uniquely determined by writing $u_i=2k_in+r_i$ for $k_i \in \{0,1,\dots,k\}$ and $r_i \in \{-n+1,\dots,n\}$. It is easy to see that $p^u$ and $p^v$ are adjacent whenever $u,v \in \intB \Lambda$ are adjacent. It remains to check that $p^u_0=f(u)$ for all $u \in \intB \Lambda$. For $u \in \Lambda_0$ this follows from the definition, and in general, this holds since the fact that $f$ is $n$-symmetric implies that $f(u)$ depends only on $(|r_1|,\dots,|r_d|)$, where $r_i$ is defined as before.
\end{proof}

Recall the definition of $Z_*(f)$ from~\eqref{eq:Z_*-def}.

\begin{lemma}\label{lem:no-infinite-cluster-of-interface-vertices}
Assume that \eqref{eq:dim-assump} holds and suppose that $f$ is sampled from some (periodic) measure of maximal entropy. Then $Z_*(f)$ almost surely has no infinite $(\Z^d)^{\otimes 2}$-connected component.
\end{lemma}
\begin{proof}
Let $\mu$ be a measure of maximal entropy and let $f$ be sampled from $\mu$. Denote the lattice of $\mu$-preserving translations by $\Gamma$. We call the elements of $Z_*$ interface vertices.
For a vertex $u$, let $E_u$ be the event that $u$ belongs to an infinite $(\Z^d)^{\otimes 2}$-path of interface vertices.
Since $\mu$ is $\Gamma$-periodic, $\mu(E_u)$ depends only on the $\Gamma$-equivalence class $[u]$ of $u$. Assume towards a contradiction that $\mu(E_u)>\delta$ for some $u$ and $\delta>0$. By ergodic decomposition, we may assume that $\mu$ is ergodic with respect to the $\Gamma$-action. Then by the ergodic theorem, the density of the set of vertices $v \in [u]$ for which $E_v$ occurs is $\mu(E_u)$ almost surely. In particular, $\mu(\cE_n) \to 1$ as $n \to \infty$, where $\cE_n$ is the event that at least a $\delta$-proportion of vertices in $\Lambda_n$ are connected to $(\intB \Lambda_n)^{+4}$ by a $(\Z^d)^{\otimes 2}$-path of interface vertices in $\Lambda_n \setminus (\intB \Lambda_n)^{+2}$. Note that the event that a vertex $v$ is an interface vertex is measurable with respect to the values of $f$ on $v^{+3}$, and thus, $\cE_n$ is measurable with respect to the values of $f$ on $\Lambda_n$ so that we may regard it as a collection of proper colorings of $\Lambda_n$.

Denote by $\Omega^{\tau,B}_\Lambda$ the set of proper colorings of $\Lambda$ that agree with $\tau$ on $B$. Denote also $\Omega^\tau_\Lambda := \Omega^{\tau,\intB \Lambda}_\Lambda$.
Then, using~\eqref{eq:entropy-chain-rule}-\eqref{eq:entropy-subadditivity},
\[ \Ent(f|_{\Lambda_n}) \le \Ent(f|_{\intB \Lambda_n})+\Ent(\cE_n) + \mu(\cE_n^c) \cdot \log |\Omega^{\text{free}}_{\Lambda_n}| + \max_{\tau \in [q]^{\intB \Lambda_n}} \log |\Omega^\tau_{\Lambda_n} \cap \cE_n| .\]
In particular, there exists a fixed (deterministic) boundary condition $\tau \in [q]^{\Z^d}$ such that
\[ \frac{\log |\Omega^\tau_{\Lambda_n} \cap \cE_n|}{|\Lambda_n|} \to h(\mu) \qquad\text{as }n \to \infty .\]
Using the assumption that $\mu$ has maximal entropy, we shall show that this is impossible.

\begin{figure}
	\centering
	\begin{subfigure}[t]{.4\textwidth}
		\centering
		\includegraphics[scale=1]{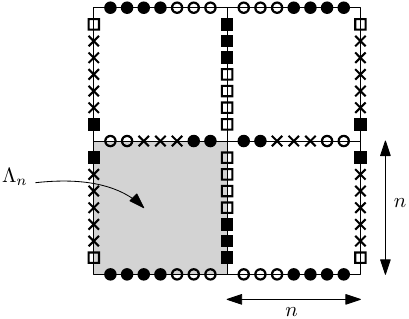}
		\caption{Boundary conditions on a box of side-length $n$ are reflected to obtain boundary conditions on a box of side-length $2n$ (more precisely, on the union of the boundaries of the $2^d$ boxes of side-length $n$).}
		\label{fig:maximal-entropy1}
	\end{subfigure}%
	\begin{subfigure}{15pt}
		\quad
	\end{subfigure}%
	\begin{subfigure}[t]{.58\textwidth}
		\centering
		\includegraphics[scale=0.54]{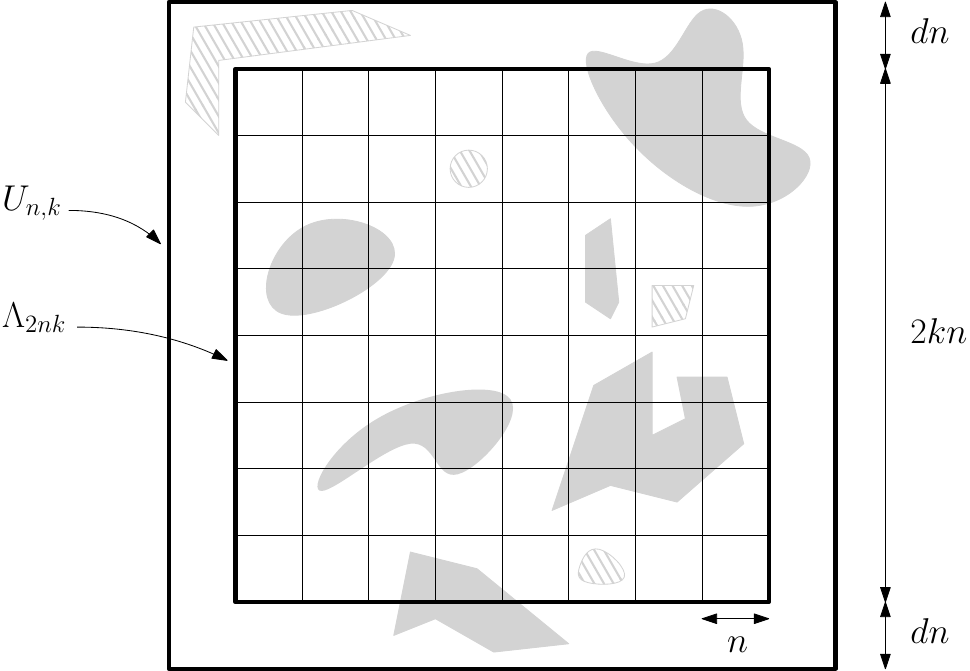}
		\caption{Many translated copies of $\Lambda_{2n}$ with a suitably chosen boundary condition are placed within a slightly larger box $U_{n,k}$. Connected components of $Z_*$ are depicted, with shaded regions representing components that intersect $B_{n,k}$. Proposition~\ref{prop:Z_*-bound} implies that the total area covered by the latter is typically not large when $f$ is sampled from a measure of maximal entropy.}
		\label{fig:maximal-entropy2}
	\end{subfigure}	
	\caption{Excluding the possibility of infinite components of $Z_*$.}
	\label{fig:maximal-entropy}
\end{figure}

The first step is to magnify the effect at a given scale $n$ by replicating it many times.
Namely, we take the model in domain $\Lambda_n$ with $\tau$ boundary conditions, and duplicate it to obtain a model in domain $\Lambda_{2kn}$, with each of the $(2k)^d$ shifted copies of the smaller box $\Lambda_n$ having the same boundary conditions (up to reflections). Indeed, by reflecting $\tau$ along the sides of the box $\Lambda_n$ some $2k-1$ number of times in each coordinate direction, we get boundary conditions $\tau_{n,k}$ defined on $B_{n,k} := n\{0,1,\dots,2k-1\}^d + \intB \Lambda_n$. Let $\cE_{n,k}$ denote the event that at least a $\delta$-proportion of vertices in $\Lambda_{2nk}$ are connected to $B_{n,k}^{+4}$ by a $(\Z^d)^{\otimes 2}$-path of interface vertices in $\Lambda_{2nk} \setminus (\intB \Lambda_{2nk})^{+2}$. With a slight abuse of notation, we regard $\cE_{n,k}$ below as a collection of proper colorings of either $\Lambda_{2nk}$ or $U_{n,k} := \{-dn,\dots,2kn+dn\}^d$, according to the context.
Then
\[ \frac{\log |\Omega^{\tau_{n,k},B_{n,k}}_{\Lambda_{2kn}} \cap \cE_{n,k}|}{|\Lambda_{2kn}|} \ge (2k)^d \cdot \frac{\log |\Omega^\tau_{\Lambda_n} \cap \cE_n|}{{|\Lambda_{2kn}|}} = h(\mu) - o(1) \qquad\text{as }n \to \infty .\]
By Lemma~\ref{lem:pattern-extends-to-ab-boundary-conditions}, each proper coloring of $\Lambda_{2kn}$ having $\tau_{n,k}$ boundary conditions can be extended to a proper coloring of $U_{n,k}$ having $(a,b)$-boundary conditions for some $a \neq b$ depending only on $\tau_{n,k}$. Thus, letting $P_{n,k}$ be a dominant pattern extending $(\{a\},\{b\})$ and letting $\Omega^P_\Lambda$ be the set of proper colorings of $\Lambda$ for which $\intB \Lambda$ is in the $P$-pattern, we have
\[ \Pr_{U_{n,k},P_{n,k}}(\cE_{n,k}) \cdot |\Omega^{P_{n,k}}_{U_{n,k}}| = |\Omega^{P_{n,k}}_{U_{n,k}} \cap \cE_{n,k}| \ge |\Omega^{\tau_{n,k},B_{n,k}}_{\Lambda_{2kn}} \cap \cE_{n,k}| .\]
On the other hand,
\[ \log |\Omega^{P_{n,k}}_{U_{n,k}}| - \log |\Omega^{\text{free}}_{\Lambda_{2kn}}| \le \log |\Omega^{\text{free}}_{U_{n,k} \setminus \Lambda_{2kn}}| \le |U_{n,k} \setminus \Lambda_{2kn}| \log q \le C_{d,q} n^d k^{d-1} ,\]
so that
\[ h(\mu) \le \frac{\log |\Omega^{\text{free}}_{\Lambda_{2kn}}|}{|\Lambda_{2kn}|} + \frac{C_{d,q}}{k} + \frac{\log \Pr_{U_{n,k},P_{n,k}}(\cE_{n,k})}{|\Lambda_{2kn}|} + o(1) \qquad\text{as }n \to \infty .\]
Thus, since $\mu$ has maximal entropy, we will arrive at a contradiction if
\[ \limsup_{k \to \infty} \limsup_{n \to \infty} \frac{\log \Pr_{U_{n,k},P_{n,k}}(\cE_{n,k})}{|\Lambda_{2kn}|} < 0 .\]
This follows from Proposition~\ref{prop:Z_*-bound} as it implies that
\[ \Pr_{U_{n,k},P_{n,k}}(\cE_{n,k}) \le 2^{C_d k^d n^{d-1}} \cdot e^{-c_{d,q}\delta(kn)^d} . \qedhere \]
\end{proof}

\begin{prop}\label{prop:max-entropy-states-are-mixtures}
Assume that \eqref{eq:dim-assump} holds.
Then every (periodic) measure of maximal entropy is a mixture of the $P$-pattern Gibbs states.
\end{prop}

\begin{proof}
Let $f$ be sampled from a Gibbs state $\mu$ under which $Z_*(f)$ almost surely has no infinite $(\Z^d)^{\otimes 2}$-connected components. In light of Lemma~\ref{lem:no-infinite-cluster-of-interface-vertices}, it suffices to show that such a measure $\mu$ is a mixture of the $P$-pattern Gibbs states.

Let $U \subset \Z^d$ be finite and connected. Let us show that, almost surely, there exists a dominant pattern $P$ and a finite set $V$ containing $U$ such that $(\intB V)^+$ is in the $P$-pattern. Indeed, if we let~$W$ denote the $(\Z^d)^{\otimes 2}$-connected component of $U \cup Z_*$ containing~$U$, then $W$ is almost surely finite. Thus, if $V$ denotes the co-connected closure of $W^+$ with respect to infinity, then $V$ is finite, connected, co-connected and contains $U$. Since $\intextB V$ is connected by Lemma~\ref{lem:int+ext-boundary-is-connected} (Corollary~\ref{cor:int+ext-boundary-is-connected torus} in the $\Z^{d_1}\times\T_{2m}^{d_2}$ setting) and is contained in $\intextB W^+ = W^{+2} \setminus W$, which is disjoint from $Z_*$, it follows from the definition of $Z_*$ that $(\intextB V)^+$ is in the $P$-pattern for some $P$.

Now consider the boxes $U_n := \{-n,\dots,n\}^d$ and let $P_n$ and $V_n$ be as above.
For a dominant pattern $P$, let $\cE_P$ be the event that $\{ n : P_n=P \}$ is infinite.
As there are finitely many dominant patterns, $\bigcup_P \cE_P$ occurs almost surely. By a similar argument as in the proof of Lemma~\ref{lem:marginal-distribution-given-agreement}, and using the fact that the finite-volume $P$-pattern measures converge, it follows that $\mu(\cdot \mid \cE_P)$ is precisely the $P$-pattern Gibbs state $\mu_P$. Thus, the events $\{\cE_P\}_P$ are disjoint and $\mu$ is the mixture $\sum_P \mu(\cE_P) \mu_P$.
\end{proof}

\bibliographystyle{amsplain}
\bibliography{biblio}

\end{document}